\pdfoutput=1
\documentclass [12pt]{amsart}
\usepackage[utf8]{inputenc}

\usepackage{amsmath,amssymb,amsthm,amsfonts}
\usepackage{mathtools}

\usepackage[english]{babel}
\usepackage{comment}
\usepackage{hyperref}
\usepackage{bbm}
\usepackage{mathrsfs}

\usepackage{tikz,graphicx,color}
\usepackage{tikz-cd}
\usepackage{tikz-3dplot}
\usetikzlibrary{calc}
\usetikzlibrary{arrows}
\usetikzlibrary{shapes}
\usetikzlibrary{patterns}
\usetikzlibrary{positioning}
\usetikzlibrary{arrows.meta}
\usetikzlibrary{decorations.markings}
\usepackage{epstopdf}

\usepackage{enumerate}
\usepackage[arrow]{xy}
\usepackage{diagbox}
\usepackage{subfig}
\usepackage{chngcntr} 
\usepackage{arcs}
\usepackage{xcolor}
\usepackage{tabu}
\usepackage{booktabs}
\usepackage{esvect}
\usepackage[margin=1.0in]{geometry}

\usepackage{arydshln}

\newtheorem{theorem}{Theorem}[section]

\newtheorem{lemma}[theorem]{Lemma}
\newtheorem{proposition}[theorem]{Proposition}
\newtheorem{corollary}[theorem]{Corollary}
\theoremstyle{definition}
\newtheorem{remark}[theorem]{Remark}
\theoremstyle{definition}
\newtheorem{definition}[theorem]{Definition}
\newtheorem{conjecture}[theorem]{Conjecture}

\theoremstyle{definition}
\newtheorem{problem}[theorem]{Problem}
\theoremstyle{definition}
\newtheorem{example}[theorem]{Example}

\def\Ccal{\mathcal{C}}\def\Ecal{\mathcal{E}}\def\Ical{\mathcal{I}}\def\Mcal{\mathcal{M}}\def\Tcal{\mathcal{T}}\def\Zcal{\mathcal{Z}}

% ...

\def\C{\mathbb{C}}
\def\R{\mathbb{R}}
\def\N{\mathbb{N}}
\def\Z{\mathbb{Z}}

% ...

\newcommand\parr[1]{{({#1})}}
% ...

\def\<{{\langle}}
\def\>{{\rangle}}

\def\RP{{\R P}}

\def\Id{\operatorname{Id}}

\def\det{{ \operatorname{det}}}

\def\Ker{{ \operatorname{Ker}}}

\def\diag{{ \operatorname{diag}}}

\def\rk{{\mathrm{rk}}}

\def\op{{ \operatorname{op}}}

\def\Span{ \operatorname{span}}

\def\wt{\operatorname{wt}}
\def\dim{\operatorname{dim}}
\def\Res{\operatorname{Res}}

\def\RR{{\mathbb R}}

\def\RP{{\RR\mathbb P}}

\def\GL{\operatorname{GL}}

\def\Mat{\operatorname{Mat}}
\def\Gr{\operatorname{Gr}}
\def\Grtnn{\Gr_{\ge 0}}
\def\Grtp{\Gr_{>0}}

\newcommand\Grge[1]{\Gr_{\ge #1}}
\def\Grtnnm{\Grge m}
\def\Grtnnk{\Gr_{\ge m}}
\def\Grtnnl{\Gr_{\ge m}}

\def\am{\mathcal{A}}

\def\uamk{\am_{n,k,\ge m}}
\def\uaml{\am_{n,\l,\ge m}}

\def\ampl{\mathcal{A}_{n,k,m}(Z)}
\def\amplnlmZ{\mathcal{A}_{n,\l,m}(Z)}
\def\amplk{\uamk(Z)}
\def\ampll{\uaml(\Zt)}
\def\amplint{\mathcal{A}^{\circ}_{n,k,m}(Z)}
\def\Id{\operatorname{Id}}

\def\V{V}
\def\v{V}
\def\Zp{W}
\def\zp{W}
\def\VZ{U}
\def\vz{U}

\def\Vt{\tilde{\V}}
\def\vt{\tilde{\v}}
\def\Zt{\tilde{Z}}
\def\Zpt{\tilde{\Zp}}
\def\zpt{\tilde{\zp}}
\def\VZt{\tilde{\VZ}}
\def\vzt{\tilde{\vz}}

\def\Fib{\operatorname{F}}

\def\Fl{{\rm Fl}}

\def\Aff{\tilde{S_n}}
\newcommand\id[1]{\operatorname{id}_{#1}}
\newcommand\shift[2]{#1+#2}
\def\Bound{{\mathcal{B}}}
\newcommand\ce[2]{\Pi_{\shift{#1}{#2}}^\C}

\newcommand\poscell[1]{\Pi^{>0}_{#1}}
\newcommand\poscelltnn[1]{\Pi^{\ge0}_{#1}}
\newcommand\pc[2]{\poscell{\shift{#1}{#2}}}

\newcommand\pctnn[2]{\poscelltnn{\shift{#1}{#2}}}
\def\inv{\operatorname{inv}}
\def\alt{\operatorname{alt}}

\def\twistpm{\tau^{\parr k}}
\def\twistpm{\tau}
\def\twist{\tau}

\newcommand{\dashedrightarrow}{\dashrightarrow}

\def\l{\ell}

\def\stw{\theta} 
\def\tstw{{\tilde\stw}} 
\def\stack{\operatorname{stack}}
\def\op{{\operatorname{op}}}

\def\transp{t}

\def\Le{\scalebox{-1}[1]{L}}

\def\V{V}
\def\v{V}
\def\W{W}
\def\w{W}
\def\U{U}
\def\u{U}

\def\Vt{{\tilde \V}}
\def\Wt{{\tilde \W}}
\def\Ut{{\tilde \U}}
\def\Zt{{\tilde Z}}
\def\vt{{\tilde \v}}
\def\wt{{\tilde \w}}
\def\ut{{\tilde \u}}

\newcommand\Grtperp[2]{\Gr_{>0}^\perp(#1,#2)}

\def\omk{\omega_{k,n}}
\def\oml{\omega_{\l,n}}
\def\omkl{\omega_{k+\l,n}}
\newcommand\om[2]{\omega_{\shift{#1}{#2}}}
\newcommand\D[3]{d^{#1\times#2}(#3)}
\newcommand\positroid[2]{\Mcal_{\shift{#1}{#2}}}
\newcommand\grneck[2]{\Ical_{\shift{#1}{#2}}}

\def\simkl{\simeq_{k,\l}}

\newcommand\pmin[2]{\Delta_{#1}^{\operatorname{circ}}(#2)}

\def\Deltas{D}

\def\omref{\omega^{\operatorname{ref}}}
\def\omrefklk{\omref_{\Fl(k,k+\l;n)}}

\begin{document}
\numberwithin{equation}{section}
\title{Parity duality for the amplituhedron}
\author{Pavel Galashin}
\address{Department of Mathematics, Massachusetts Institute of Technology, 77 Massachusetts Avenue,
Cambridge, MA 02139, USA}
\email{\href{mailto:galashin@mit.edu}{galashin@mit.edu}}
\address{Department of Mathematics, University of Michigan, 2074 East Hall, 530 Church Street, Ann Arbor, MI 48109-1043, USA}
\email{\href{mailto:tfylam@umich.edu}{tfylam@umich.edu}}
\author{Thomas Lam}
\thanks{ T.L.\ acknowledges support from the NSF under agreement No.\ DMS-1464693.}

\subjclass[2010]{Primary:  14M15. 
  Secondary: 13F60, 
  52B99.
}

\keywords{Amplituhedron, twist map, total positivity, Grassmannian, positroid cells, scattering amplitudes, cyclic polytopes.}

\begin{abstract}
The (tree) amplituhedron $\mathcal A_{n,k,m}(Z)$ is a certain subset of the Grassmannian introduced by Arkani-Hamed and Trnka in 2013 in order to study scattering amplitudes in $N=4$ supersymmetric Yang-Mills theory.  Confirming a conjecture of the first author, we show that when $m$ is even, a collection of affine permutations yields a triangulation of $\mathcal A_{n,k,m}(Z)$ for any $Z\in \operatorname{Gr}_{>0}(k+m,n)$ if and only if the collection of their inverses yields a triangulation of $\mathcal A_{n,n-m-k,m}(Z)$ for any $Z\in\operatorname{Gr}_{>0}(n-k,n)$. We prove this duality using the twist map of Marsh and Scott. We also show that this map preserves the canonical differential forms associated with the corresponding positroid cells, and hence obtain a parity duality for amplituhedron differential forms.
\end{abstract}

\date{\today}
\maketitle

\setcounter{tocdepth}{1}
\tableofcontents

\section{Introduction}
 \subsection{The amplituhedron and its triangulations}
The amplituhedron $\ampl$ is a remarkable compact subspace of the real Grassmannian $\Gr(k,k+m)$ that was introduced by Arkani-Hamed and Trnka \cite{AT} in order to give a geometric basis for computing scattering amplitudes in $N= 4$ super Yang-Mills theory. 
It is defined to be the image of Postnikov's totally nonnegative Grassmannian $\Grtnn(k,n)$ under the (rational) map $Z: \Gr(k,n) \dashedrightarrow \Gr(k,k+m)$ induced by a $(k+m) \times n$ matrix $Z$ with positive maximal minors. Triangulations of the amplituhedron play a key role in the connection between the geometry of the amplituhedron and the physics of scattering amplitudes.

The totally nonnegative Grassmannian $\Gr(k,n)_{\geq 0}$ has a stratification by positroid cells denoted $\pc k f$, where $f$ is an \emph{affine permutation}.  We say that a collection of affine permutations $f_1,f_2,\ldots,f_N$ forms a {\it $Z$-triangulation} if the images $Z(\pc k {f_i})$ all have the same dimension as $\ampl$ and are mutually non-overlapping, the union of the images is dense in $\ampl$, and the maps $Z: \pc k {f_i} \to Z(\pc k {f_i})$ are injective.  We say that these affine permutations form an \emph{$(n,k,m)$-triangulation} if they form a $Z$-triangulation for all $Z \in \Grtp(k+m,n)$. 

The amplituhedron depends on nonnegative integer parameters $k, m, n$ where $k + m \leq n$.  Setting $\l := n-k-m$, we have $k+\l+m = n$. Throughout the paper, we only deal with the case where $m$ is even. Our first main result states that $(n,k,m)$-triangulations of the amplituhedron are in bijection with $(n,\l,m)$-triangulations of the amplituhedron.  This bijection is obtained by taking an affine permutation to its inverse.  One of the enumerative motivations of our work is a conjectural formula of Karp, Williams, and Zhang \cite{KWZ} for the number of cells in a particular triangulation of the amplituhedron.  This formula is invariant under the operation of swapping the parameters $k$ and $\l$, which is called {\it parity duality} in physics.

\subsection{The stacked twist map}
 The twist map is a rational map $\twist: \Mat(r,n) \dashedrightarrow \Mat(r,n)$ on matrices first defined by Marsh and Scott \cite{MS}.  It was motivated by the twist map on a unipotent group, due to Berenstein, Fomin, and Zelevinsky \cite{BFZ1,BFZ2}.  Marsh and Scott's work has subsequently been extended by Muller and Speyer \cite{MuS}.  The main interest of Marsh and Scott, and Muller and Speyer, is the induced twist map $\Gr(r,n) \dashedrightarrow \Gr(r,n)$ on the Grassmannian.  Taking $r := k+\l$, we study instead the {\it stacked twist map}
 $$
 \stw: \Gr(k,n) \times \Gr(\l,n) \dashedrightarrow \Gr(\l,n) \times \Gr(k,n).
 $$
 At the heart of our approach is the result that the (rational) stacked twist map restricts to a diffeomorphism
 $$
\stw: \pc k f\times \Grtperp{k+m}n\to \Grtperp{\l+m}n\times \pc \l {f^{-1}}
 $$
 on positroid cells indexed by affine permutations related by inversion. (See~\eqref{eq:Grtperp} for a definition of $\Grtperp{k+m}n$.) This allows us not only to relate triangulations under parity duality, but to study other geometric questions such as subdivisions and whether images of positroid cells overlap.

\subsection{The amplituhedron form}
The physical motivation for the study of the amplituhedron is an important (but conjectural) top-degree rational differential form $\omega_{\ampl}$ on the Grassmannian $\Gr(k,k+m)$ called the {\it amplituhedron form}.  The case of main interest in physics is the case $m = 4$.  Roughly speaking, scattering amplitudes in $N=4$ super Yang-Mills theory can be computed by formally integrating $\omega_{\ampl}$ to obtain a rational function depending on a $4 \times n$ matrix, thought of as the momentum four-vectors of $n$ particles.  The integral removes the dependence of the form on the Grassmannian $\Gr(k,k+4)$, leaving only the dependence on the $(k+4) \times n$ matrix $Z$; furthermore, $k$ of the $k+4$ rows of $Z$ are taken to be Grassmann or fermionic parameters to account for the supersymmetry of the theory~\cite{AT}.
 
The amplituhedron form $\omega_{\ampl}$ can be obtained from a triangulation of the amplituhedron as follows.  Each positroid cell $\pc k f$ has a {\it canonical form} $\om k f$ of top degree in $\pc k f$.  Conjecturally, $\omega_{\ampl}$ is obtained by summing the pushforwards $Z_* \om k f$ over any triangulation of the amplituhedron. 
 
The form $\omega_{\ampl}$ has degree $km$ while the form $\omega_{\mathcal{A}_{n,\l,m}(Z')}$ has degree $\l m$.  To relate them, we consider the {\it universal amplituhedron} $ \am_{n,k,m}$, a compact subset of the two-step flag variety $\Fl(\l,k+\l;n)$.  The fiber of $ \am_{n,k,m}$ over a point $Z^\perp \in \Gr(\l,n)$ can be canonically identified with the amplituhedron $\ampl$, hence the terminology ``universal''.  A related construction in the complex algebraic setting was studied in \cite{LamCDM}.  The dimension of $\am_{n,k,m}$ is equal to $k \l +  km+\l m$.
 
We show that the stacked twist map descends to a homeomorphism $\tstw: \uamk \to \uaml$ between subsets $\uamk \subset \am_{n,k,m}$ and $\uaml \subset \am_{n,\l,m}$ that contain the respective interiors.  In particular, triangulations of $\am_{n,k,m}$ are taken to triangulations of $\am_{n,\l,m}$ via the stacked twist map.  We introduce the {\it universal amplituhedron form}, a top-degree rational differential form $\omega_{\uamk}$ defined by summing pushforward canonical forms over a triangulation of the universal amplituhedron.  Our main result concerning universal amplituhedron forms is that they are related by pullback under the stacked twist map: $\tstw^*\omega_{\uaml} = \omega_{\uamk}$.

\subsection*{Outline}
We recall some preliminaries on the positroid stratification of $\Grtnn(k,n)$ in Section~\ref{sec:tnn_gr}. Next, we define triangulations of the amplituhedron in Section~\ref{sec:triang-ampl} and state a duality result for them (Theorem~\ref{thm:main}). We then explain the construction and some properties of the stacked twist map $\stw$ and use them to give a short proof to Theorem~\ref{thm:main} in Section~\ref{sec:stw}. After that, we discuss canonical forms of positroid cells in Section~\ref{sec:can_forms} and state that they are preserved by $\stw$ in Theorems~\ref{thm:top_form_stw} and~\ref{thm:lower_form_stw}. In Section~\ref{sec:examples}, we give several examples that illustrate our main results. In Section~\ref{sec:subdivisions}, we define subdivisions of the amplituhedron, and give an analog of Theorem~\ref{thm:main} for them (Theorem~\ref{thm:main_subd}). In the same section, we define the universal amplituhedron and its subdivisions. We explain the relationship of our results with physics in Section~\ref{sec:ampl_form}, where we define the universal amplituhedron form and show that it is preserved by the stacked twist map (Theorem~\ref{thm:univ_ampl_form}). We then proceed to proving the numerous properties of the stacked twist map that we have mentioned in the previous sections. We discuss the topology of the amplituhedron in Section~\ref{sec:interior}. In Section~\ref{sec:stw_properties}, we prove that $\stw$ takes triangulations to triangulations. In Section~\ref{sec:can_forms_3}, we prove that $\stw$ preserves the canonical form, and in Section~\ref{sec:ampl_form_proofs}, we show that $\stw$ preserves the universal amplituhedron form. Finally, we discuss some open problems in Section~\ref{sec:future}.

\medskip
\noindent
{\bf Acknowledgments.} We are grateful to Steven Karp for numerous conversations on topics related to this work, and for helpful comments and suggestions at various stages of this project.  We also thank Nima Arkani-Hamed, Jake Bourjaily, and Alex Postnikov for inspiring conversations.

\section{The totally nonnegative Grassmannian}\label{sec:tnn_gr}

Let $n \geq 2$ be a positive integer, and denote $[n]:=\{1,2,\dots,n\}$.  Let $\Aff$ be the affine Coxeter group of type $A$: the elements of $\Aff$ are all bijections $f:\Z\to\Z$ satisfying
\begin{enumerate}
\item\label{item:add_n} $f(i+n)=f(i)+n$ for all $i\in \Z$;
\item\label{item:sum} $\sum_{i=1}^{n}\left(f(i)-i\right)=0$.
\end{enumerate}
Given any bijection $f:\Z\to\Z$ satisfying~\eqref{item:add_n}, its \emph{number of inversions} $\inv(f)$ is the number of pairs $(i,j)$ such that $i \in [n]$, $j \in \Z$, $i < j$, and $f(i)>f(j)$.  We will work with certain finite subsets of $\Aff$. Namely, for integers $a,b\geq0$, define
\[\Aff(-a,b):=\{f\in\Aff\mid i-a\leq f(i)\leq i+b\text{ for all $i\in\Z$}\}.\]

A \emph{$(k,n)$-bounded affine permutation} is a bijection $f:\Z\to\Z$ satisfying
\begin{itemize}
\item $f(i+n)=f(i)+n$ for all $i\in \Z$;
\item $\sum_{i=1}^{n}\left(f(i)-i\right)=kn$;
\item $i\leq f(i)\leq i+n$ for all $i\in\Z$.
\end{itemize}
We let $\Bound(k,n)$ denote the set of $(k,n)$-bounded affine permutations.
There is a distinguished element $\id k\in\Bound(k,n)$ defined by $\id k(i):=i+k$ for all $i\in\Z$, and it is clear that we have $\shift k f\in\Bound(k,n)$ if and only if $f\in\Aff(-k,n-k)$. Here $\shift kf$ is defined by $(\shift k f)(i):=f(i)+k$ for all $i\in\Z$. Throughout the text, we will work with $\Aff(-k,n-k)$ rather than with $\Bound(k,n)$.

Let $k \leq n$ and $\Gr(k,n)$ denote the Grassmannian of $k$-planes in $\R^n$.
We let ${[n]\choose k}$ be the collection of all $k$-element subsets of $[n]$. For $I\in{[n]\choose k}$ and $X\in\Gr(k,n)$, we denote by $\Delta_I(X)$ the \emph{Pl\"ucker coordinate} of $X$, i.e., the maximal minor of $X$ with column set $I$. The \emph{totally nonnegative Grassmannian} $\Grtnn(k,n)\subset\Gr(k,n)$ consists of all $X\in\Gr(k,n)$ such that $\Delta_I(X)\geq0$ for all $I\in{[n]\choose k}$. Similarly, the \emph{totally positive Grassmannian} $\Grtp(k,n)\subset\Grtnn(k,n)$ is the set of $X\in\Gr(k,n)$ with all Pl\"ucker coordinates strictly positive.

The totally nonnegative Grassmannian has a stratification by positroid cells that are indexed by bounded affine permutations, or equivalently by $\Aff(-k,n-k)$.
As we will explain in Definition~\ref{dfn:pos_cell}, each $f\in\Aff(-k,n-k)$ gives rise to a certain collection $\positroid k f$ of $k$-element subsets of $[n]$ called a \emph{positroid}, and then we associate an open \emph{positroid cell} $\pc k f\subset \Grtnn(k,n)$ to $f$ by
\[\pc k f:=\{X\in\Grtnn(k,n)\mid \text{$\Delta_J(X)>0$ for $J\in\positroid k f$ and $\Delta_J(X)=0$ for $J\notin\positroid k f$}\}.\]
We similarly define a closed \emph{positroid cell} $\pctnn k f$ by requiring $\Delta_J(X)\geq0$ for $J\in\positroid k f$ and $\Delta_J(X)=0$ for $J\notin\positroid k f$. The positroid cell has dimension $\dim(f):=k(n-k)-\inv(f)$, so we say that it has \emph{codimension} $\inv(f)$.  The unique cell $\poscell {\id k}$ of codimension zero coincides with $\Grtp(k,n)$. Each element $X\in\Grtnn(k,n)$ belongs to $\pc k f$ for a unique $f\in\Aff(-k,n-k)$, and thus the totally nonnegative Grassmannian decomposes as a union of positroid cells:
\begin{equation}\label{eq:Grtnn_strata}
  \Grtnn(k,n)
  =\bigsqcup_{f\in\Aff(-k,n-k)} \pc k f.
\end{equation}
See Lemma~\ref{lemma:ranks} for an alternative definition of $\pc k f$.

\section{Triangulations of amplituhedra}\label{sec:triang-ampl}
Fix four nonnegative integers $k,\l,m,n$ such that $m$ is even and $k+\l+m=n$. Let $Z\in\Grtp(k+m,n)$ be a matrix with all maximal minors positive.\footnote{Abusing notation, we treat elements of $\Gr(r,n)$ as full rank $r\times n$ matrices. For example, choosing different representatives of $Z\in\Grtp(k+m,n)$ yields different subsets of $\Gr(k,k+m)$, but does not affect the notion of a triangulation of the amplituhedron.} We treat $Z$ as a map from $\Grtnn(k,n)$ to $\Gr(k,k+m)$ sending (the row span of) a $k\times n$ matrix $\V$ to (the row span of) $\V\cdot Z^\transp$,  where $ ^\transp$ denotes matrix transpose.  The \emph{(tree) amplituhedron} $\ampl$ is defined to be the image of $\Grtnn(k,n)$ under $Z$. We note that $\ampl\subset\Gr(k,k+m)$ is a $km$-dimensional compact subset of $\Gr(k,k+m)$.

\begin{definition}
Let $f\in\Aff(-k,n-k)$ be an affine permutation. We say that $f$ is \emph{$Z$-admissible} if the restriction of the map $Z$ to the positroid cell $\pc k f$ is injective. We say that $f$ is \emph{$(n,k,m)$-admissible} if it is $Z$-admissible for any $Z\in\Grtp(k+m,n)$.
\end{definition}

Clearly, in order for an affine permutation $f$ to be $Z$-admissible, the dimension of $\pc k f$ must be at most $km$. Thus the codimension of $\pc k f$ in $\Gr(k,n)$ must be at least $k\l$, equivalently, we must have $\inv(f)\geq k\l$. If $f$ is $Z$-admissible, then in order for $Z(\pc k f)$ to be a full-dimensional subset of $\ampl$ it is necessary to have $\inv(f)=k\l$. However, it is not sufficient:
\begin{lemma}\label{lemma:non-admissible}
Let $f\in\Aff(-k,n-k)$ be an affine permutation with $\inv(f)=k\l$ and let $Z\in\Grtp(k+m,n)$. If $f\notin\Aff(-k,\l)$ then $f$ is not $Z$-admissible, the dimension of $Z(\pc k f)$ is less than $km$, and $Z(\pc k f)$ is a subset of the boundary of $\ampl$. 
\end{lemma}

\begin{definition}
Let $f,g\in\Aff(-k,n-k)$ be two affine permutations. We say that $f$ and $g$ are \emph{$Z$-compatible} if the images $Z(\pc k f)$ and $Z(\pc k g)$ do not intersect. We say that $f$ and $g$ are \emph{$(n,k,m)$-compatible} if they are $Z$-compatible for any $Z\in\Grtp(k+m,n)$. 
\end{definition}

\begin{definition}\label{dfn:triangulation}
Let $f_1,\dots,f_N\in\Aff(-k,n-k)$ be a collection of $Z$-admissible and pairwise $Z$-compatible affine permutations, and assume that $\inv(f_s)=k\l$ for $1\leq s\leq N$. We say that $f_1,\dots,f_N$ \emph{form a $Z$-triangulation} if the images $Z(\pc k{f_1}),\dots,Z(\pc k{f_N})$ form a dense subset of $\ampl$. We say that $f_1,\dots,f_N$ \emph{form an $(n,k,m)$-triangulation} if they form a $Z$-triangulation for any $Z\in\Grtp(k+m,n)$.
\end{definition}

Thus by Lemma~\ref{lemma:non-admissible}, if $f_1,\dots,f_N\in\Aff(-k,n-k)$ form a $Z$-triangulation for some $Z$ then they all must belong to $\Aff(-k,\l)$. We are now ready to state our first main result.

\begin{theorem}\label{thm:main}\leavevmode
  \begin{enumerate}
  \item\label{item:main:admissible} Let $f\in\Aff(-k,\l)$ be an affine permutation. Then $f$ is $(n,k,m)$-admissible if and only if $f^{-1}\in\Aff(-\l,k)$ is $(n,\l,m)$-admissible.
  \item\label{item:main:compatible} Let $f,g\in\Aff(-k,\l)$ be two affine permutations. Then $f$ and $g$ are $(n,k,m)$-compatible if and only if $f^{-1},g^{-1}\in\Aff(-\l,k)$ are $(n,\l,m)$-compatible.
  \item\label{item:main:triang}   Let $f_1,\dots,f_N\in\Aff(-k,\l)$ be a collection of affine permutations. Then they form an $(n,k,m)$-triangulation of the amplituhedron if and only if $f_1^{-1},\dots,f_N^{-1}\in\Aff(-\l,k)$ form an $(n,\l,m)$-triangulation of the amplituhedron.
  \end{enumerate}
\end{theorem}

Using the properties of the stacked twist map, we give a simple proof of this theorem in Section~\ref{sec:stw}.  In Section~\ref{sec:subdivisions}, we define a more general notion of \emph{$Z$-subdivisions} and give an analog of Theorem~\ref{thm:main}.

\subsection{Motivation and further remarks}\label{sec:motivation}
We discuss the relationship of Theorem~\ref{thm:main} with some results present in the literature on amplituhedra.

\begin{remark}
For $k=1$, the amplituhedron $\ampl\subset \Gr(1,m+1)$ is just the cyclic polytope in the $m$-dimensional projective space. The combinatorics of its triangulations\footnote{The notion of a triangulation of a polytope imposes an additional restriction that the intersection of any two simplices is their common face. However, this extra condition is not necessary for cyclic polytopes, see e.g.~\cite[Theorem~2.4]{OT}. We thank Alex Postnikov for bringing this subtle point to our attention.} is well-understood, see e.g.~\cite{Rambau,ERR,Thomas,OT}. Therefore in this case Theorem~\ref{thm:main} describes explicitly the triangulations of the $\l=1$ amplituhedron, which has not yet been studied in the literature. We refer the reader to Section~\ref{sec:examples} for an illustration.
\end{remark}

\begin{remark}\label{rmk:BCFW}
From the point of view of physics, the most interesting case is $m=4$, which corresponds to computing tree level scattering amplitudes in the planar $N=4$ supersymmetric Yang-Mills theory. In particular, the terms in the \emph{BCFW recurrence} introduced in~\cite{BCF,BCFW} for computing these amplitudes correspond to a collection of affine permutations called the \emph{BCFW triangulation}. (It is a conjecture that these affine permutations indeed form an $(n,k,m)$-triangulation of the amplituhedron.)  We note that in this special $m=4$ case, the map $f\mapsto f^{-1}$ sends a BCFW triangulation for $(k,\l)$  to a BCFW triangulation for $(\l,k)$. It amounts to switching the colors of the vertices in the corresponding plabic graphs in $\Gr(k+2,n)$ and $\Gr(\l+2,n)$ (the indices are shifted by $2$ because of a transition from the \emph{momentum space} to the \emph{momentum-twistor space}). See~\cite[Section~11]{AHTT}, \cite[Proposition~8.4]{KWZ}, or Section~\ref{sec:BCFW} for further discussion.
\end{remark}

The following conjecture states that the notions of $Z$-admissibility, $Z$-compatibility, and $Z$-triangulation do not depend on the choice of $Z\in\Grtp(k+m,n)$. It is widely believed to be true in the physics literature.
\begin{conjecture}\label{conj:physics}
  Fix some $Z\in\Grtp(k+m,n)$.
  \begin{enumerate}[(1)]
  \item Let $f\in\Aff(-k,n-k)$ be an affine permutation. Then $f$ is $Z$-admissible if and only if $f$ is $(n,k,m)$-admissible.
  \item Let $f,g\in\Aff(-k,n-k)$ be two affine permutations. Then $f$ and $g$ are $Z$-compatible if and only if they are $(n,k,m)$-compatible.
  \item  Let $f_1,\dots,f_N\in\Aff(-k,n-k)$ be a collection of affine permutations. Then they form a $Z$-triangulation if and only if they form an $(n,k,m)$-triangulation.
  \end{enumerate}
\end{conjecture}

Therefore, assuming Conjecture~\ref{conj:physics}, the notion of, say, $Z$-compatibility does not depend on $Z$, and thus gives rise to an interesting binary relation on the set $\Aff(-k,n-k)$.

Let us now explain some motivation that led to Theorem~\ref{thm:main}. For integers $a,b,c\geq1$, define
\[M(a,b,c):=\prod_{p=1}^a\prod_{q=1}^b\prod_{r=1}^c \frac{p+q+r-1}{p+q+r-2}\]
to be the number of \emph{plane partitions} that fit into an $a\times b\times c$ box~\cite{MacMahon}. This expression is symmetric in $a,b,c$. In~\cite[Conjecture~8.1]{KWZ}, the authors noted that in all cases where a (conjectural) triangulation of $\ampl$ is known, it has precisely $M(k,\l,m/2)$ cells, and they conjectured that for all $n,k,m$, there \emph{exists} a triangulation of the amplituhedron with $M(k,\l,m/2)$ cells. For example, for $m=4$, the number of terms in the \emph{BCFW recurrence}, which correspond to the cells in the conjectural BCFW triangulation of $\ampl$, equals $M(k,\l,2)=\frac1{n-3}{n-3\choose k+1}{n-3\choose k}$ (a \emph{Narayana number}), see~\cite[Eq.~(17.7)]{abcgpt}. We generalize~\cite[Conjecture~8.1]{KWZ} as follows.
\begin{conjecture}\label{conj:M(a,b,c)}
For \emph{every} triangulation $f_1,\dots,f_N\in\Aff(-k,n-k)$ of $\ampl$ (where $m$ is even), the number $N$ of cells equals $M(k,\l,m/2)$.
\end{conjecture}
This conjecture holds for $k=1$ (see e.g.~\cite{Rambau}), and therefore by Theorem~\ref{thm:main} it holds for $\l=1$ as well. Additionally, we have experimentally verified it in the case $k=\l=m=2$, $n=6$, where there are $120$ triangulations of the amplituhedron, and each of them has size $M(2,2,1)=6$. We discuss this example in more detail in Section~\ref{sec:baues}. See also~\cite[Section~8]{KWZ} for a list of other cases where Conjecture~\ref{conj:M(a,b,c)} holds in a weaker form.

\begin{remark}
In~\cite[Remark~8.2]{KWZ}, a similar conjecture is given for odd values of $m$. One could generalize this conjecture as follows: for odd $m$, the \emph{maximal} number of cells in a triangulation of the amplituhedron is equal to $M(k,\l,\frac{m+1}2)$. This conjecture remains open, but we note that the result of Theorem~\ref{thm:main} does not hold for odd $m$, as one can easily observe already in the case $k=1$, $\l=2$, $m=1$, $n=4$. In this case, the maximal number of cells in a triangulation is indeed $M(1,2,1)=3$ and is invariant under switching $k$ and $\l$, however,  replacing each affine permutation with its inverse does not take $(n,k,m)$-triangulations to $(n,\l,m)$-triangulations.
\end{remark}
\begin{remark}
Theorem~\ref{thm:main} provides partial progress to answering~\cite[Question~8.5]{KWZ}: assuming that for each $n$, $k$, and even $m$ there exists an $(n,k,m)$-triangulation of the amplituhedron, \cite[Question~8.5]{KWZ} has positive answer.
\end{remark}

\begin{remark}
For $m = 4$, it was already observed in \cite{motivic, AHTT} that the twist map of $\Gr(4,n)$ is related to the parity duality of scattering amplitudes.  However, without the stacked twist map, the remarkable combinatorics and geometry of triangulations is hidden.
\end{remark}

\section{The stacked twist map}\label{sec:stw}
In this section, we explain a construction that lies at the core of the proof of Theorem~\ref{thm:main}. It is based on the \emph{twist map} of Marsh and Scott~\cite{MS}. For our purposes, it is more convenient to use the version of the twist map from~\cite{MuS}, which differs from the one in~\cite{MS} by a column rescaling. See~\cite[Remark~6.3]{MuS} for details.

Let $k,\l,m,n$ be as in the previous section, that is, they are all nonnegative integers such that $m$ is even and $n=k+\l+m$. Let $\V$ be a $k\times n$ matrix and $\W$ be an $\l\times n$ matrix. For $j=1,2,\dots,n$, define $\v(j)\in\R^k$ (resp., $\w(j)\in\R^\l$) to be the $j$-th column of $\V$ (resp., $\W$). We extend this to any $j\in\Z$ by the condition that $\v(j+n)=(-1)^{k-1}\v(j)$ and $\w(j+n)=(-1)^{k-1}\w(j)$. Define
\[\U=\stack(\V,\W):=\begin{pmatrix}
    \V\\
    \W
  \end{pmatrix}\]
to be a $(k+\l)\times n$ block matrix obtained by stacking $\V$ on top of $\W$. That is, the $j$-th column $\u(j)$ of $\U$ equals $(\v(j),\w(j))\in\R^{k+\l}$, and thus we also have $\u(j+n)=(-1)^{k-1}\u(j)$ for all $j\in\Z$.

\def\topcell{\Mat^\circ(k+\l,n)}
\def\topcellC{\Mat_\C^\circ(k+\l,n)}
\begin{definition}\label{dfn:topcell}
Given a $(k+\l)\times n$ matrix $\U$, we write $\U\in\topcell$ if for all $j\in\Z$, the vectors $\u(j-\l),\u(j-\l+1),\dots,\u(j+k-1)$ form a basis of $\R^{k+\l}$.
\end{definition}

We now define a map $\twistpm:\topcell\to\topcell$ sending a matrix $\U$ to a matrix $\Ut$ with columns $\ut(j),j\in\Z$. Up to sign changes and a cyclic shift, this map coincides with the \emph{right twist} of~\cite{MuS}, see~\eqref{eq:twist_MS}. For each $j\in\Z$, define a vector $\ut(j)\in\R^{k+\l}$ by the following conditions:
\begin{equation}\label{eq:stw}
\<\ut(j),\u(j-\l)\>=(-1)^{j},\quad \<\ut(j),\u(j-\l+1)\>=\dots=\<\ut(j),\u(j+k-1)\>=0,
\end{equation}
where $\<\cdot,\cdot\>$ denotes the standard inner product on $\R^{k+\l}$. Since the vectors $\u(j-\l),\u(j-\l+1),\dots,\u(j+k-1)$ form a basis, the above conditions determine $\ut(j)$ uniquely. It is also not hard to see directly that $\Ut\in\topcell$ (alternatively, see~\eqref{eq:max_minors_of_U}).  Moreover, we have $\ut(j+n)=(-1)^{n+k-1}\ut(j)=(-1)^{\l-1}\ut(j)$, because $m$ is assumed to be even. This defines $\Ut=\twistpm(\U)\in\topcell$.  Muller and Speyer also define a \emph{left twist} $\stackrel{\leftarrow}{\tau}:\topcell\to\topcell$.  They show: 
\begin{theorem}[{\cite[Theorem 6.7]{MuS}}]\label{thm:MuS}
The left and right twists are inverse diffeomorphisms of $\topcell$.
\end{theorem}

Finally, let $\Wt$ (resp., $\Vt$) be the $k\times n$ matrix (resp., the $\l\times n$ matrix) such that $\Ut=\stack(\Wt,\Vt)$. In other words, we let the $j$-th column $\wt(j)$ of $\Wt$ and $\vt(j)$ of $\Vt$ be defined so that $\ut(j)=(\wt(j),\vt(j))$.

\begin{definition}
Given a $k\times n$ matrix $V$ and an $\l\times n$ matrix $W$ such that $\stack(V,W)\in\topcell$, we denote $\stw(V,W):=(\Wt,\Vt)$, where $\Wt$ and $\Vt$ are defined above. The map $\stw$ is called the \emph{stacked twist map}.
\end{definition}
The relationship between the maps $\twistpm$, $\stack$, and $\stw$ is shown in~\eqref{eq:stw:cd}. Dashed (resp., solid) arrows denote rational (resp., regular) maps.
\begin{equation}\label{eq:stw:cd}
\begin{tikzcd}
\Mat(k,n)\times \Mat(\l,n) \arrow[d,dashrightarrow,"\stack"] \arrow[r,dashrightarrow, "\stw"] & \Mat(k,n)\times \Mat(\l,n)  \arrow[d,dashrightarrow,"\stack"] \\
\topcell \arrow[r, "\twistpm"] &  \topcell
\end{tikzcd}
\quad
\begin{tikzcd}
(V,W) \arrow[d,dashrightarrow,mapsto,"\stack"] \arrow[r,dashrightarrow,mapsto, "\stw"] & (\Wt,\Vt)  \arrow[d,dashrightarrow,mapsto,"\stack"] \\
\U \arrow[r,mapsto, "\twistpm"] &  \Ut
\end{tikzcd}
\end{equation}

Our first goal is to view $\stw$ as a map on pairs of subspaces rather than pairs of matrices. Thus we would like to treat $V$ and $W$ not as matrices but as their row spans. The following result will be deduced as a simple consequence of Lemma~\ref{lemma:inverse_transpose}.

\begin{lemma}\label{lemma:stw:Gr}
  The map $\stw$ descends to a rational map
  \[\stw:\Gr(k,n)\times \Gr(\l,n)\dashedrightarrow \Gr(k,n)\times \Gr(\l,n).\]
\end{lemma}

It turns out that $\stw$ has remarkable properties when restricted to the totally nonnegative Grassmannian. For each $0\leq m\leq n-k$, define a subset $\Grtnnk(k,n)\subset \Grtnn(k,n)$ by
\begin{equation}\label{eq:Grtnnm}
\Grtnnm (k,n):=\bigsqcup_{f\in\Aff(-k,n-k-m)} \pc k f.
\end{equation}
Comparing this to~\eqref{eq:Grtnn_strata}, we see that setting $m=0$ indeed recovers $\Grtnn(k,n)$. More generally, we have
\[\Grge 0(k,n)\supset \Grge 1(k,n)\supset\dots\supset \Grge{n-k}(k,n)=\Grtp(k,n).\]
As we will see in Lemma~\ref{lemma:ranks}, the set $\Grtnnm(k,n)$ can be alternatively described as follows.
\begin{proposition}\label{prop:ranks}
For $0\leq m\leq n-k$ and $\l=n-k-m$, we have
  \[\Grtnnm(k,n)=\{\V\in\Grtnn(k,n): \rk([\v(j-\l)|\v(j-\l+1)|\dots|\v(j+k-1)])=k \text{ for all $j\in\Z$}\}.\]
\end{proposition}
Here  $[\v(j-\l)|\v(j-\l+1)|\dots|\v(j+k-1)]$ denotes the $k\times (k+\l)$ matrix with columns $\v(j-\l),\dots,\v(j+k-1)$, and $\rk$ denotes its rank. We set
\begin{equation}\label{eq:ampll}
\amplk:=Z(\Grtnnm(k,n))\subset\ampl.
\end{equation}

Let us also define
\begin{equation}\label{eq:Grtperp}
\Grtperp{k+m} n:=\{Z^\perp \mid Z\in\Grtp(k+m,n)\}\subset\Gr(\l,n).
\end{equation}
Here for $Z\in\Grtp(k+m,n)$ we denote by $\W:=Z^\perp\in\Gr(\l,n)$ the subspace orthogonal to $Z$.
\begin{remark}
The sets $\Grtperp{k+m}n$ and $\Grtp(\l,n)$ are related as follows. Define a map $\alt:\Gr(\l,n)\to\Gr(\l,n)$ sending a matrix $\W\in\Gr(\l,n)$ with columns $\W(j), 1\leq j\leq n$ to a matrix $\W':=\alt(\W)$ with columns $\W'(j)=(-1)^{j-1}\W(j)$. It turns out $\alt$ is a diffeomorphism between $\Grtp(\l,n)$ and $\Grtperp{k+m}n$. In fact, for any set $I\subset [n]$ of size $\l$, we have
\begin{equation}\label{eq:alt_duality}
\Delta_I(\alt(\W))=\Delta_{[n]\setminus I}(Z),
\end{equation}
where $Z:=\W^\perp$. See~\cite[Lemma~3.3]{KW} and references therein.
\end{remark}

The following result is the main ingredient of the proof of Theorem~\ref{thm:main}.
\begin{theorem}\label{thm:stw}\leavevmode
Let $\V\in \Grtnnm(k,n)$ and $\W\in\Grtperp{k+m}n$. 
  \begin{enumerate}[\normalfont (i)]
  \item\label{item:stw:well_defined} We have $\stack(\V,\W)\in\topcell$, thus $\stw(\V,\W)=(\Wt,\Vt)$ is well defined.
  \item\label{item:stw:positivity}  We have $\Wt\in\Grtperp{\l+m}n$ and $\Vt\in\Grtnnm(\l,n)$.
  \item\label{item:stw:inverse}  Let $f\in\Aff(-k,\l)$ be the unique affine permutation such that $\V\in\pc k f$. Then $\Vt~\in~\pc \l {f^{-1}}$. 
  \end{enumerate}
\end{theorem}
By Theorem \ref{thm:MuS}, we conclude that $\stw$ restricts to a homeomorphism 
\[\stw:\Grtnnm(k,n)\times \Grtperp{k+m}n\to \Grtperp{\l+m}n\times \Grtnnm(\l,n),\]
and for each $f\in\Aff(-k,\l)$, $\stw$ restricts to a diffeomorphism 
\begin{equation}\label{eq:stw_regular_map}
\stw:\pc k f\times \Grtperp{k+m}n\to \Grtperp{\l+m}n\times \pc \l {f^{-1}}.
\end{equation}

Finally, we note that $\stw$ ``preserves the fibers'' of the map $Z:\Grtnnm(k,n)\to\amplk$. More precisely, given $\V\in\Grtnnm(k,n)$, let us define a \emph{fiber} $\Fib(\V,Z)$ of $Z$ by
\begin{equation}\label{eq:fiber_dfn}
\Fib(\V,Z):=\{\V'\in\Grtnn(k,n)\mid \V\cdot Z^\transp=\V'\cdot Z^\transp\}.
\end{equation}
As we will see later in Proposition~\ref{prop:fibers_Grtnnm}, for $\V\in\Grtnnm(k,n)$ we have $\Fib(\V,Z)\subset\Grtnnm(k,n)$.
\begin{proposition}\label{prop:fibers}
  Let $\V\in\Grtnnm(k,n)$, $Z\in\Grtp(k+m,n)$, $\W:=Z^\perp\in\Grtperp{k+m}n$, $\stw(\V,\W)=(\Wt,\Vt)$, and $\Zt:=\Wt^\perp\in\Grtp(\l+m,n)$. Then for any $\V'\in\Fib(\V,Z)$, we have $\stw(\V',\W)=(\Wt,\Vt')$ for some $\Vt'\in\Fib(\Vt,\Zt)$. The map $\V'\mapsto \Vt'$ is a homeomorphism between $\Fib(\V,Z)\subset \Grtnnm(k,n)$ and $\Fib(\Vt,\Zt)\subset\Grtnnm(\l,n)$.
\end{proposition}

We postpone the proof of Theorem~\ref{thm:stw} and Proposition~\ref{prop:fibers} until Section~\ref{sec:stw_properties}. Let us now give a simple proof of Theorem~\ref{thm:main} using these results.

\begin{proof}[Proof of Theorem~\ref{thm:main}]
  Let $f\in\Aff(-k,\l)$ be an affine permutation. Then $f$ is $(n,k,m)$-admissible if and only if $\pc k f\cap\Fib(\V,Z)$ contains at most one point for each $\V\in\Grtnnm(k,n)$ and $Z\in\Grtp(k+m,n)$. By Proposition~\ref{prop:fibers}, this fiber $\Fib(\V,Z)$ maps bijectively to a fiber $\Fib(\Vt,\Zt)$ for some $\Zt\in\Grtp(\l+m,n)$ and $\Vt\in\Grtnnm(\l,n)$. Moreover, the points of $\pc k f\cap\Fib(\V,Z)$ map bijectively to the points of $\pc \l {f^{-1}} \cap \Fib(\Vt,\Zt)$ by Theorem~\ref{thm:stw}, part~\eqref{item:stw:inverse}. This proves Theorem~\ref{thm:main}, part~\eqref{item:main:admissible}.

  Similarly, if $f,g\in\Aff(-k,\l)$ are two affine permutations then they are $(n,k,m)$-compatible if and only if either one of $\pc k f\cap \Fib(\V,Z)$ or $\pc k g\cap \Fib(\V,Z)$ is empty for each $\V\in\Grtnnm(k,n)$ and $Z\in\Grtp(k+m,n)$. Again by Proposition~\ref{prop:fibers} and Theorem~\ref{thm:stw}, part~\eqref{item:stw:inverse}, this is equivalent to saying that either one of $\pc \l {f^{-1}}\cap \Fib(\Vt,\Zt)$ or $\pc \l {g^{-1}}\cap \Fib(\Vt,\Zt)$ is empty, showing Theorem~\ref{thm:main}, part~\eqref{item:main:compatible}.

  Finally, let $f_1,\dots,f_N\in\Aff(-k,\l)$ be a collection of $(n,k,m)$-admissible and pairwise $(n,k,m)$-compatible affine permutations. By the previous two parts, we see that $f_1^{-1},\dots,f_N^{-1}\in\Aff(-\l,k)$ are $(n,\l,m)$-admissible and pairwise $(n,\l,m)$-compatible. Suppose now that $f_1,\dots,f_N$ form an $(n,k,m)$-triangulation. This means that the union of the images $Z(\pctnn k {f_1})$, $\dots$, $Z(\pctnn k{f_N})$ equals to $\ampl$ for each $Z\in\Grtp(k+m,n)$. In particular, this implies that for all $\V\in\Grtnnm(k,n)$ and $Z\in\Grtp(k+m,n)$, $\Fib(\V,Z)$ contains a point $\V'$ of $\pctnn k {f_s}$ for some $1\leq s\leq N$. Since $\V$ belongs to $\Grtnnm(k,n)$, there is a unique affine permutation $f\in\Aff(-k,\l)$ such that $\V'\in\pc kf\subset\pctnn k{f_s}$. Applying Theorem~\ref{thm:stw}, part~\eqref{item:stw:inverse} together with Proposition~\ref{prop:fibers}, we get that for all $\Vt\in\Grtnnm(\l,n)$ and $\Zt\in\Grtp(\l+m,n)$, $\Fib(\Vt,\Zt)$ contains a point $\Vt'\in\pc \l{f^{-1}}\subset\pctnn \l {f_s^{-1}}$ for some $1\leq s\leq N$. Thus the union of the images $\Zt(\pctnn \l {f_1^{-1}}),\dots,\Zt(\pctnn \l{f_N^{-1}})$ contains $\ampll$. Since this union is closed, it must contain the closure of $\ampll$ which is clearly equal to $\amplnlmZ$ (see also Proposition~\ref{prop:open_dense}), finishing the proof of Theorem~\ref{thm:main}, part~\eqref{item:main:triang}.
\end{proof}

\section{The canonical form}\label{sec:can_forms}
In this section, we give an accessible introduction to the canonical form on the Grassmannian and explain its relationship with the stacked twist map.

\subsection{The canonical form for the top cell}
The Grassmannian $\Gr(k,n)$ is a $k(n-k)$-dimensional manifold. It can be covered by various coordinate charts. For example, one can represent any generic element $X\in\Gr(k,n)$ as the row span of a $k\times n$ block matrix $[\Id_k\mid A]$ for a unique $k\times (n-k)$ matrix $A$. Here $\Id_k$ denotes the $k\times k$ identity matrix. Alternatively, any generic $X$ is determined by its Pl\"ucker coordinates lying in a \emph{cluster} $\{\Delta_J\mid J\in\Ccal\}$ (see~\cite{FZ,scott_06} and Definition~\ref{dfn:clusters}), where $\Ccal$ is a collection of $k(n-k)+1$ subsets of $[n]$, each of size $k$, satisfying certain properties. Here we view $\{\Delta_J\mid J\in\Ccal\}$ as an element of the $k(n-k)$-dimensional projective space. There is a (rational) differential $k(n-k)$-form $\omk$ on $\Gr(k,n)$ with remarkable properties called \emph{the canonical form}. Before we give its definition, we note that $\omk$ is a rational \emph{top form} (i.e., an $r$-form on an $r$-dimensional manifold), and on the chart $[\Id_k\mid A]$ it can therefore be written as $f(A)\,\D k{(n-k)}A$, where $f(A)$ is some rational function of the entries of $A$, and
\[\D k{(n-k)}A:=da_{1,1}\wedge da_{1,2}\wedge\dots\wedge da_{k,n-k}.\]
Thus defining $\omk$ amounts to giving a formula for $f(A)$. For $X\in\Gr(k,n)$, define
\[\pmin k X:=\prod_{i=1}^n \Delta_{\{i,i+1,\dots,i+k-1\}}(X)\]
to be the product of all \emph{circular} $k\times k$ minors of $X$ (the indices are taken modulo $n$).
\begin{definition}\label{dfn:canonical_form_top_cell}
  The canonical form $\omk$ on $\Gr(k,n)$ is defined as follows: if $X\in\Gr(k,n)$ is the row span of $[\Id_k\mid A]$ for a $k\times (n-k)$ matrix $A$ then
  \[\omk(X):=\pm\frac{\D k{(n-k)}A}{\pmin k X}.\]
Until Section~\ref{sec:ampl_form}, we view $\omk$ as only being defined up to a sign.
\end{definition}

\begin{example}
Consider an element $X\in\Gr(2,4)$ given by the row span of
\begin{equation}\label{eq:2x4_matrix}
\begin{pmatrix}
    1 & 0 & p & q\\
    0 & 1 & r & s
  \end{pmatrix}.
\end{equation}
Then we have
\begin{equation}\label{eq:2x4_form}
\omk(X)=\pm\frac{dp\wedge dq\wedge dr\wedge ds}{p(ps-rq)s}.
\end{equation}
\end{example}

What makes $\omk$ special is that it can be alternatively computed using other coordinate charts on $\Gr(k,n)$ in a simple way. For the proof of the following result, see~\cite[Proposition~13.3]{LamCDM}.
\begin{proposition}\label{prop:form:clusters}
  Suppose that the Pl\"ucker coordinates $\{\Delta_J\mid J\in\Ccal\}$ form a cluster for some $\Ccal=\{J_0,J_1,\dots,J_{k(n-k)}\}$ (see Definition~\ref{dfn:clusters}), and suppose that they are normalized so that $\Delta_{J_0}(X')=1$ for all $X'$ in the neighborhood of $X\in\Gr(k,n)$. Then
  \begin{equation} \label{eq:cancluster} \omk(X)=\pm \frac{d(\Delta_{J_1}(X))\wedge d(\Delta_{J_2}(X))\wedge\dots\wedge d(\Delta_{J_{k(n-k)}}(X))}{\Delta_{J_1}(X) \cdot\Delta_{J_2}(X)\cdots\Delta_{J_{k(n-k)}}(X)}.\end{equation}
\end{proposition}
The normalization $\Delta_{J_0}(X')=1$ means that the homogeneous coordinates $\Delta_{J_i}(X)$ become rational functions on $\Gr(k,n)$.  Instead of normalization, one could also replace $\Delta_{J_i}(X)$ by the rational function $\Delta_{J_i}(X)/\Delta_{J_0}(X)$ in \eqref{eq:cancluster}. 

The form $\omk(X)$ depends neither on the choice of the cluster $\Ccal$ nor on the choice of $J_0\in\Ccal$. This result will serve as a definition of the canonical form for any positroid cell in the next section.

\begin{example}
There are two clusters in $\Gr(2,4)$: $\Ccal:=\{\{1,2\},\{2,3\},\{3,4\},\{1,4\},\{1,3\}\}$ and $\Ccal':=\{\{1,2\},\{2,3\},\{3,4\},\{1,4\},\{2,4\}\}$. For $X$ being the row span of the matrix given in~\eqref{eq:2x4_matrix}, we have
\[    \Delta_{12}(X)=1,\quad \Delta_{23}(X)=-p,\quad \Delta_{34}(X)=ps-rq,\quad \Delta_{14}(X)=s,\]
\[    \Delta_{13}(X)=r,\quad \Delta_{24}(X)=-q.\]
Letting $J_0$ be $\{1,2\}$ and applying Proposition~\ref{prop:form:clusters} to $\Ccal$ and $\Ccal'$ yields the following two expressions for $\omk(X)$:
\[\omk(X)=\pm\frac{dp\wedge d(ps-rq)\wedge ds\wedge dr}{p(ps-rq)sr}=\pm\frac{dp\wedge d(ps-rq)\wedge ds\wedge dq}{p(ps-rq)sq}.\]
One easily checks that each of these expressions is indeed equal up to a sign to the one given in~\eqref{eq:2x4_form}: we have $d(ps-rq)=p\,ds+s\,dp-r\,dq-q\,dr$, but only one of these four terms stays nonzero after taking the wedge with  $dp\wedge ds\wedge dr$ or $dp\wedge ds\wedge dq$.
\end{example}

\begin{remark}
Yet another simple alternative way to compute $\omk$ is to use Postnikov's parametrization~\cite{Pos} in terms of \emph{edge weights of a plabic graph}. See~\cite[Section~13]{LamCDM} for the proof that all three ways of computing $\omk$ produce the same answer.
\end{remark}
\begin{remark}
The definition of the canonical form given in Proposition~\ref{prop:form:clusters} generalizes to all cluster algebras, see~\cite[Section~5.7]{ABL}.
\end{remark}

Before we state the main result of this section, let us give a warm-up result on the relationship between the twist map of~\cite{MS} and the canonical form. Recall that $\twistpm:\topcell\to\topcell$ sends a matrix $U\in\topcell$ to another matrix $\Ut\in\topcell$ defined in~\eqref{eq:stw}. Similarly to Lemma~\ref{lemma:stw:Gr}, one can prove that $\twistpm$ descends to a rational map $\Gr(k+\l,n)\dashedrightarrow\Gr(k+\l,n)$. Let $\omkl$ denote the canonical form on $\Gr(k+\l,n)$. 

\begin{proposition}\label{prop:top_form_twist}
The twist map $\twistpm:\Gr(k+\l,n)\dashedrightarrow\Gr(k+\l,n)$ \emph{preserves} the canonical form $\omkl$ up to a sign. That is, $\omkl$ equals to the pullback $\pm\twistpm^\ast\omkl$ of $\omkl$ via $\twistpm$.
\end{proposition}

\begin{example}
Let $k+\l=2$, $n=4$, and suppose that $\U$ is the row span of~\eqref{eq:2x4_matrix}. After possibly switching the signs of some columns, $\Ut:=\twistpm(\U)$ is the row span of
\[
\begin{pmatrix}
  0 & 1 & r & s\\
  -1 & 0 & -p & -q
\end{pmatrix} \cdot \diag \left(\frac1s,1,\frac1p,\frac1{ps-rq}\right)\simeq_{k+\l}
\begin{pmatrix}
  1 & 0 & s & \frac{qs}{ps-rq}\\
  0 & 1 & \frac r p & \frac s{ps-rq}
\end{pmatrix}.
\]
Here $\diag$ denotes a diagonal matrix, and the matrix on the right hand side is obtained from the matrix on the left hand side by multiplying the bottom row by $-s$ and switching the two rows (these operations preserve its row span). Thus a cluster of Pl\"ucker coordinates of $\Ut$ is given by
\[\Delta_{12}=\pm1,\quad \Delta_{23}=\pm s,\quad \Delta_{34}=\pm \frac s p,\quad \Delta_{14}=\pm \frac s{ps-rq},\quad \Delta_{13}=\pm\frac r p.\]
Applying Proposition~\ref{prop:form:clusters}, we see that the pullback of $\omkl$ under $\twistpm$ is equal to
\begin{equation}\label{eq:2x4_form:twist}
  \pm\frac{ds\wedge d \left(\frac{s}{p}\right)\wedge d\left(\frac{r}{p}\right)\wedge d\left(\frac{s}{ps-rq}\right)}{s\cdot \frac{s}{p}\cdot \frac{r}{p}\cdot \frac{s}{ps-rq}}= \pm\frac{ds\wedge \frac{s\,dp}{p^2}\wedge \frac{dr}{p}\wedge \frac{sr\,dq}{(ps-rq)^2}}{s\cdot \frac sp\cdot \frac rp\cdot \frac s{ps-rq}}.
\end{equation}
Here e.g. $d(\Delta_{23})=\pm ds$ and $d(\Delta_{34})=\pm \left(\frac{ds}{p}-\frac{s\,dp}{p^2}\right)$, but because we are already taking the wedge with $ds$, only the $\frac{s\,dp}{p^2}$ term appears in the numerator of~\eqref{eq:2x4_form:twist}. We see that after some cancellations, \eqref{eq:2x4_form:twist} indeed yields the same result as~\eqref{eq:2x4_form}, in agreement with Proposition~\ref{prop:top_form_twist}.
\end{example}

Consider the manifold $\Gr(k,n)\times \Gr(\l,n)$ for some $k+\l+m=n$ as above with $m$ even. One can consider a top form $\omk\wedge\oml$ on this manifold. We are now ready to state the main result of this section. 
\begin{theorem}\label{thm:top_form_stw}
The stacked twist map $\stw:\Gr(k,n)\times \Gr(\l,n)\dashedrightarrow\Gr(k,n)\times \Gr(\l,n)$ preserves the form $\omk\wedge\oml$ up to a sign.
\end{theorem}

Proposition~\ref{prop:top_form_twist} has a short proof using the results of~\cite{MS} and~\cite{MuS}, see Section~\ref{sec:proof-prop-top-form}. On the other hand, the proof of Theorem~\ref{thm:top_form_stw} is quite involved and is given in Section~\ref{sec:proof-thm-forms}. See Section~\ref{sec:ex_forms} for an example.

\subsection{The canonical form for lower positroid cells}\label{sec:can_form_lower}

Given an affine permutation $f\in\Aff(-k,n-k)$, one can consider certain collections $\Ccal\subset{[n]\choose k}$ of $k$-element subsets of $[n]$ called \emph{clusters}, see Definition~\ref{dfn:pos_cell}. All clusters have the same size, and give a parametrization of $\pc k f$:
\begin{proposition}[\cite{OPS}]\label{prop:OPS}
  Let $\Ccal$ be a cluster for $f\in\Aff(-k,n-k)$. Then it has size
  \begin{equation}\label{eq:Ccal_size}
|\Ccal|=k(n-k)+1-\inv(f).
  \end{equation}
  The map $\pc k f \to \RP_{>0}^\Ccal$ given by $X\mapsto \{\Delta_J(X)\mid J\in\Ccal\}$ is a diffeomorphism.
\end{proposition}
Here $\RP_{>0}^{\Ccal}$ denotes the subset of the $(|\Ccal|-1)$-dimensional real projective space where all coordinates are positive.

Thus we have various coordinate charts on $\pc k f$, which now allows us to use Proposition~\ref{prop:form:clusters} to define a canonical form on this manifold.
\begin{definition}\label{dfn:form_lower_cell}
  Let $\Ccal=\{J_0,J_1,\dots,J_N\}$ be a cluster for $f\in\Aff(-k,n-k)$, where $N=k(n-k)-\inv(f)$. Suppose that the values of the Pl\"ucker coordinates are rescaled so that $\Delta_{J_0}(X')=1$ for $X'$ in the neighborhood of $X\in\pc k f$.  Then the \emph{canonical form} $\om k f$ on $\pc k f$ is defined by
  \[\om k f(X):= \pm \frac{d(\Delta_{J_1}(X))\wedge d(\Delta_{J_2}(X))\wedge\dots\wedge d(\Delta_{J_N}(X))}{\Delta_{J_1}(X) \cdot\Delta_{J_2}(X)\cdots\Delta_{J_N}(X)}.\]
\end{definition}

\begin{proposition}\label{prop:om_k_f_canonical}
The form $\om k f$ depends neither on the choice of $\Ccal$ nor on the choice of $J_0\in\Ccal$.
\end{proposition}
See~\cite[Theorem~13.2]{LamCDM} for a proof. In fact, we will see later that the form $\om k f$ can be obtained from the form $\omk=\om k{\id 0}$ by subsequently taking \emph{residues} (see Section~\ref{sec:residues} for a definition). This implies Proposition~\ref{prop:om_k_f_canonical}, as well as a generalization of Theorem~\ref{thm:top_form_stw} to lower cells which we now explain.

By~\eqref{eq:alt_duality}, the set $\Grtperp{k+m}n$ can be identified with $\Grtp(\l,n)$ by changing the sign of every second column, and since the form $\oml$ is only defined up to a sign, we can view $\oml$ as a top rational form on $\Grtperp{k+m}n$ instead.
The following is our main result regarding canonical forms.
\begin{theorem}\label{thm:lower_form_stw}
  For $f\in\Aff(-k,\l)$, under the stacked twist map diffeomorphism
  \[\stw:\pc k f\times \Grtperp{k+m}n\to \Grtperp{\l+m}n\times \pc \l {f^{-1}}\]
  (see~\eqref{eq:stw_regular_map}) we have 
  \begin{equation}\label{eq:lower_form_stw}
  \stw^*(\omk\wedge\om \l{f^{-1}}) =  \pm \om k f\wedge\oml.
  \end{equation}
\end{theorem}
We discuss the application of Theorem \ref{thm:lower_form_stw} to the amplituhedron form in Section~\ref{sec:ampl_form}.

\def\perms{P}
\def\notblack{black!50}

    \newcommand\drawperm[5]{
      \def\lw{1pt}
      \def\lblscl{0.8}
      \begin{tikzpicture}[xscale=0.5]
 
   \begin{scope}[even odd rule]
  \clip (-1.4,-1) rectangle (7.4,2);
  \foreach \i in {-4,-3,-2,-1,0,6,7,8,9}
  {
    \node[label=above:{\scalebox{\lblscl}{$\i$}},draw=\notblack,  scale=0.3,fill=\notblack] (u\i) at (\i,1) {};
    \node[label=below:{\scalebox{\lblscl}{$\i$}},draw=\notblack,  scale=0.3,fill=\notblack] (d\i) at (\i,0) {};
  }
  \foreach \i in {1,2,...,5}
  {
    \node[label=above:{\scalebox{\lblscl}{$\i$}},draw,  scale=0.3,fill=black] (u\i) at (\i,1) {};
    \node[label=below:{\scalebox{\lblscl}{$\i$}},draw,  scale=0.3,fill=black] (d\i) at (\i,0) {};
  }
  \pgfmathtruncatemacro{\zero}{#5 - 5}
  \pgfmathtruncatemacro{\minusone}{#4 - 5}
  \pgfmathtruncatemacro{\minustwo}{#3 - 5}
  \pgfmathtruncatemacro{\six}{#1 + 5}
  \pgfmathtruncatemacro{\seven}{#2 + 5}
  \pgfmathtruncatemacro{\eight}{#3 + 5}
  \draw[line width=\lw,\notblack] (u0) -- (d\zero);
  \draw[line width=\lw,\notblack] (u-1) -- (d\minusone);
  \draw[line width=\lw,\notblack] (u-2) -- (d\minustwo);
  \draw[line width=\lw,\notblack] (u6) -- (d\six);
  \draw[line width=\lw,\notblack] (u7) -- (d\seven);
  \draw[line width=\lw,\notblack] (u8) -- (d\eight);
  \draw[line width=\lw] (u1) -- (d#1);
  \draw[line width=\lw] (u2) -- (d#2);
  \draw[line width=\lw] (u3) -- (d#3);
  \draw[line width=\lw] (u4) -- (d#4);
  \draw[line width=\lw] (u5) -- (d#5);
      \end{scope}
\end{tikzpicture}
}
\begin{figure}
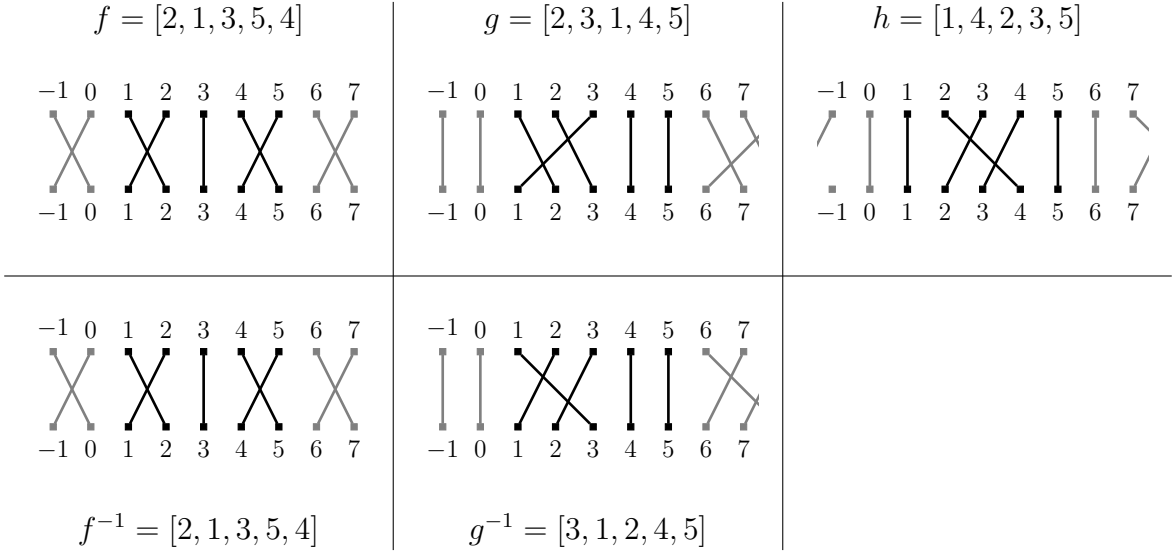

\begin{tabular}{c|c|c}

  $f=[2,1,3,5,4]$  & $g=[2,3,1,4,5]$ & $h=[1,4,2,3,5]$\\
\drawperm 2 1 3 5 4
&
\drawperm 2 3 1 4 5
&
\drawperm 1 4 2 3 5
  \\\hline
\drawperm 2 1 3 5 4
&
\drawperm 3 1 2 4 5
&
  \\
$f^{-1}=[2,1,3,5,4]$  & $g^{-1}=[3,1,2,4,5]$ &
\end{tabular}
    \caption{\label{fig:permutations} Up to rotation, there are three affine permutations in $\Aff(-k,n-k)$ for $k=2,n=5$ with two inversions (top). There are only two such affine permutations in $\Aff(-k,n-k)$ for $k=1,n=5$ (bottom).}
  \end{figure}
  
\section{Examples}\label{sec:examples}
In this section, we illustrate each of our main results by an example.

\subsection{Triangulations of a pentagon}
As a warm-up, we let $k=1$, $\l=2$, $m=2$, $n=5$. In this case, the amplituhedron $\ampl$ is a pentagon in $\RP^2$ whose vertices are the five columns of $Z\in\Grtp(3,5)$. We write $f\in\Aff$ in \emph{window notation}, i.e., $f=[f(1),f(2),\dots, f(n)]$. Let $\perms_{n,k,m}\subset\Aff(-k,n-k)$ be the set of affine permutations $f$ with $\inv(f)=k\l=2$. We say that $f,g\in\Aff$ are the same \emph{up to rotation} if for some $s\in\Z$ we have $f(i+s)=g(i)+s$ for all $i\in\Z$. There are ten affine permutations in $\perms_{5,1,2}$, and up to rotation, they form two classes of five affine permutations in each, shown in Figure~\ref{fig:permutations} (bottom). All of these ten affine permutations are $(n,k,m)$-admissible, and the images of the corresponding positroid cells are the ${5\choose 3}=10$  triangles in a pentagon. The affine permutation $[2,1,3,5,4]$ in Figure~\ref{fig:permutations} (bottom left) corresponds to the convex hull of the triangle with vertex set $\{1,3,4\}$ inside a pentagon, while the affine permutation $[3,1,2,4,5]$ in Figure~\ref{fig:permutations} (bottom right) corresponds to the convex hull of a triangle with vertex set $\{1,4,5\}$. These two permutations are therefore $(5,1,2)$-compatible and together with the permutation $[1,2,5,3,4]$ (corresponding to a triangle with vertex set $\{1,2,3\}$) they form a triangulation of $\ampl$.

\subsection{Triangulations of the $\l=1$ amplituhedron}
  We now study the ``dual'' case $k=2$, $\l=1$, $m=2$, $n=5$ in detail. There are fifteen affine permutations in $\perms_{5,2,2}$, and they form three classes (up to rotation) of five permutations in each, shown in Figure~\ref{fig:permutations} (top).

  The permutations $f=[2,1,3,5,4]$ and $g=[2,3,1,4,5]$ belong to $\Aff(-k,\l)$ while $h=[1,4,2,3,5]$ does not belong to $\Aff(-k,\l)$ because $h(2)=4>2+\l$. Thus by Lemma~\ref{lemma:non-admissible}, $h$ is not $(n,k,m)$-admissible (and it is indeed easy to check directly that $h$ is not $Z$-admissible for all $Z\in\Grtp(k+m,n)$). Note that $h^{-1}$ does not correspond to any triangle inside a pentagon.

Let us now fix $f=[2,1,3,5,4]$ and consider $\V\in\pc k f$ and $\W\in\Grtperp{k+m}{n}$ given by
  \begin{equation}\label{eq:ex:V:W}
  \V=\begin{pmatrix}
      1&0&0&-2&-1\\
      0&1&1&2&0
    \end{pmatrix};\quad  \W=\begin{pmatrix}
      1 &-1&3&-2&1
    \end{pmatrix}.
  \end{equation}

  \begin{figure}
    
  \newcommand\tabline[9]{
    \textcolor{\notblack}{$#1$} &     \textcolor{\notblack}{$#2$} & {$\bf #3$}&{$\bf #4$}&{$\bf #5$}&{$\bf #6$}&{$\bf #7$}&    \textcolor{\notblack}{$#8$} &     \textcolor{\notblack}{$#9$} 
  }
    \begin{tabular}{c|cccccccccc}
    $j$&-1&0&1&2&3&4&5&6&7\\\hline
       & \tabline 2   1 1 0  0{-2}{-1}{-1}{0}\\
$\U(j)$& \tabline {-2}0 0 1  1  2   0   0 {-1}\\\cdashline{2-10}
       & \tabline 2 {-1}1{-1}3{-2}{ 1}{-1}{ 1}\\
\multicolumn{1}{l}{} &\\
    $j$               &-1&0&1&2&3&4&5&6&7\\\hline
                      & \tabline{2}{ 0}{-1}{ 4}{-8}{ 2}{ 0}{-1}{4}\\
$(\Ut\cdot\Deltas)(j)$& \tabline{4}{-2}{ 1}{ 0}{-6}{ 4}{-2}{ 1}{0}\\\cdashline{2-10}
                      & \tabline{2}{ 0}{ 1}{ 0}{ 2}{ 2}{ 0}{ 1}{0}\\
    \end{tabular}

    \caption{\label{fig:stw}The columns of $\U$ (top) and of $\Ut\cdot \Deltas$ (bottom) for $j=-1,0,\dots,7$.}
  \end{figure}
  
The fact that $\V\in\pc k f$ can be easily checked using~\eqref{eq:ranks} while the fact that $\W\in\Grtperp{k+m}n$ follows since $\alt(\W)$ is totally positive, see~\eqref{eq:alt_duality}. We would like to compute $(\Wt,\Vt)=\stw(\V,\W)$. To do so, we first need to find the matrix $\Ut=\twistpm(\U)$. Let us introduce a diagonal $n\times n$ matrix $\Deltas=\diag(|\Delta_{J_1}(\U)|,\dots,|\Delta_{J_n}(\U)|)$, where $J_j=\{j-\l,j-\l+1,\dots,j+k-1\}$. We will consider the matrix $\Ut\cdot \Deltas$ rather than $\Ut$ itself in order to clear the denominators (cf.~\eqref{eq:twist_MS}). The matrix $\U$ is given in Figure~\ref{fig:stw} (top), so
\begin{equation}\label{eq:ex_D}
\Deltas=\diag(|-2|, |4|, |-8|, |10|,|-4|)=\diag(2,4,8,10,4).
\end{equation}
Thus the column $\Ut(j)$ gets rescaled by the absolute value of $\Delta_{J_j}(\U)$. The matrix $\Ut\cdot\Deltas$ is given in Figure~\ref{fig:stw} (bottom): we encourage the reader to check that the column $\Ut(j)$ is orthogonal to the columns $\U(j)$ and $\U(j+1)$, and that $\<\Ut(j),\U(j-1)\>=(-1)^j$. (Recall also that $\U(j+n)=(-1)^{k-1}\U(j)$ and $\Ut(j+n)=(-1)^{\l-1}\Ut(j)$, which is consistent with Figure~\ref{fig:stw}.)
Thus the output of the stacked twist map in this case is given by
\begin{equation}\label{eq:ex:Vt:Wt}
\Wt\cdot\Deltas=\begin{pmatrix}
    -1 & 4 & -8 & 2 & 0\\
     1 & 0 & -6 & 4 & -2
  \end{pmatrix};\quad \Vt\cdot\Deltas=\begin{pmatrix}
    1 & 0 & 2 & 2 & 0
  \end{pmatrix}. 
\end{equation}

\begin{figure}

\tikzset{%
    add/.style args={#1 and #2}{
        to path={%
 ($(\tikztostart)!-#1!(\tikztotarget)$)--($(\tikztotarget)!-#2!(\tikztostart)$)%
  \tikztonodes},add/.default={.2 and .2}}
}  
\def\nodescl{0.3}
\def\tikzscl{3}
\def\lw{1.5pt}
\def\lwnb{0.5pt}
\def\labelscl{0.8}
\def\minorscl{1}
\def\pink{pink!70!black}
\def\green{green!70!black}
\begin{tabular}{cc}
\begin{tikzpicture}[scale=\tikzscl]
\coordinate (A) at (-1,1);
\coordinate (B) at (0,1);
\coordinate (C) at (0,0);
\coordinate (D) at (-1,0);
\coordinate (E) at (0,-0.33);
\coordinate (F) at (1,0);

   \begin{scope}[even odd rule]
     \clip (D)--(B)--(C)--(D);
     \clip (A)--(E)--(F)--(A);
     \fill[fill=cyan,opacity=0.5] (-1,0) rectangle (0,1);
   \end{scope}
   
   \begin{scope}[even odd rule]
     \clip (-1.5,-0.7) rectangle (1.2,1.5);
     \draw [add= 10 and 10, red, line width=\lw] (D) to (B);
     \draw [add= 10 and 10, \green, line width=\lw] (B) to (C);
     \draw [add= 10 and 10, \pink, line width=\lw] (A) to (E);
     \draw [add= 10 and 10, brown, line width=\lw] (A) to (F);
     \draw [add= 10 and 10, blue, line width=\lw] (D) to (C);
     \draw [add= 10 and 10, \notblack, line width=\lwnb] (A) to (B);
     \draw [add= 10 and 10, \notblack, line width=\lwnb] (A) to (D);
     \draw [add= 10 and 10, \notblack, line width=\lwnb] (E) to (D);
     \draw [add= 10 and 10, \notblack, line width=\lwnb] (E) to (F);
     \draw [add= 10 and 10, \notblack, line width=\lwnb] (B) to (F);
   \end{scope}
   
\node[draw,circle,fill=black,scale=\nodescl] (An) at (A) {};
\node[draw,circle,fill=black,scale=\nodescl] (Bn) at (B) {};
\node[draw,circle,fill=black,scale=\nodescl] (Cn) at (C) {};
\node[draw,circle,fill=black,scale=\nodescl] (Dn) at (D) {};
\node[draw,circle,fill=black,scale=\nodescl] (En) at (E) {};
\node[draw,circle,fill=black,scale=\nodescl] (Fn) at (F) {};

\node[draw,circle,fill=black,scale=\nodescl] (12v34) at (-0.57,0.42) {};
\node[draw,circle,fill=black,scale=\nodescl] (15v34) at (-0.25,0) {};
\node[draw,circle,fill=black,scale=\nodescl] (12v45) at (-0.33,0.66) {};
\node[draw,circle,fill=black,scale=\nodescl] (23v45) at (0,0.5) {};
\node[anchor=north east,scale=\labelscl] (Al) at (A) {$(-1,1)$};
\node[anchor=north west,scale=\labelscl] (Al) at (B) {$(0,1)$};
\node[anchor=north west,scale=\labelscl] (Al) at (C) {$(0,0)$};
\node[anchor=north east,scale=\labelscl] (Al) at (D) {$(-1,0)$};
\node[anchor=north east,scale=\labelscl] (Al) at (E) {$(0,-\frac13)$};
\node[anchor=north,scale=\labelscl] (Al) at (F) {$(1,0)$};
\node[scale=\minorscl] (Al) at (-0.25,0.35) {$\V'\geq0$};
\node[anchor=west,scale=\minorscl] (Al) at (0.25,1.2) {\textcolor{red}{$\Delta_{12}(\V')=a-b+1$}};
\node[anchor=south,scale=\minorscl] (Al) at (0,1.5) {\textcolor{\green}{$\Delta_{23}(\V')=-4a$}};
\node[anchor=north west,scale=\minorscl] (Al) at (0.1,-0.7) {\textcolor{\pink}{$\Delta_{34}(\V')=8a+6b+2$}};
\node[anchor=west,scale=\minorscl] (Al) at (0.22,0.4) {\textcolor{brown}{$\Delta_{45}(\V')=-2a-4b+2$}};
\node[anchor=west,scale=\minorscl] (Al) at (1.2,0) {\textcolor{blue}{$\Delta_{15}(\V')=2b$}};
 \end{tikzpicture}
  
\end{tabular}
  \caption{\label{fig:fiber} Each of the ten lines represents the set of pairs $(a,b)$ for which one of the minors of $\V'$ is zero. We have $\V'\in\Grtnn(k,n)$ precisely when $(a,b)$ belongs to the shaded region.}
\end{figure}
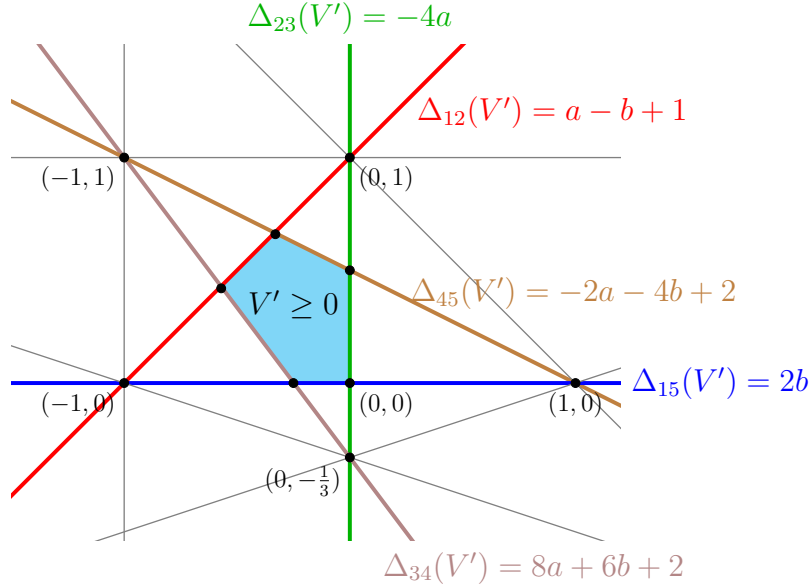  

  We observe that $\alt(\Wt)$ is a totally positive $2\times 5$ matrix, and thus $\Wt\in\Grtperp{\l+m}n$. Similarly, $\Vt$ is a totally nonnegative $1\times 5$ matrix. Moreover, we have $\Vt\in \pc \l {f'}$ for $f'=[2,1,3,5,4]=f^{-1}$, as predicted by Theorem~\ref{thm:stw}, part~\eqref{item:stw:inverse}.

  We now consider the fibers of $Z$. Let $\V,\W$ be given by~\eqref{eq:ex:V:W}, and consider a $4\times 5$ matrix $Z:=\W^\perp$. Suppose that $\V'\cdot Z^\transp=\V\cdot Z^\transp$ for some $\V'\in\Grtnn(k,n)$, then we must have  $\V'=\V+A\cdot \W$
  for some $2\times 1$ matrix $A=\begin{pmatrix}
    a \\
    b
  \end{pmatrix}$ with $a,b\in\R$. The matrix $\U':=\stack(\V',\W)$ is therefore given by
\[\U'= \begin{pmatrix}
a + 1 & -a & 3  a & -2  a - 2 & a - 1 \\
b & -b + 1 & 3  b + 1 & -2  b + 2 & b \\
1 & -1 & 3 & -2 & 1
  \end{pmatrix}.\]
We then find the matrix $\Ut':=\twistpm(\U')$:
\[\Ut'\cdot\Deltas= \begin{pmatrix}
-1 & 4 & -8 & 2 & 0 \\
1 & 0 & -6 & 4 & -2 \\
a - b + 1 & -4  a & 8  a + 6  b + 2 & -2  a - 4  b + 2 & 2  b
\end{pmatrix}.\]
Here $\Deltas=\diag(2,4,8,10,4)$ is unchanged since $\U'$ is obtained from $\U$ by row operations. Setting $(\Wt',\Vt'):=\stw(\V',\W)$, we get
\[\Wt'\cdot\Deltas=\begin{pmatrix}
    -1 & 4 & -8 & 2 & 0\\
     1 & 0 & -6 & 4 & -2
  \end{pmatrix};\quad \Vt'\cdot\Deltas=\begin{pmatrix}
    a - b + 1 & -4  a & 8  a + 6  b + 2 & -2  a - 4  b + 2 & 2  b
  \end{pmatrix}.\]
(Note that by Proposition~\ref{prop:fibers},  $\Wt'=\Wt$. As we will later see in Lemma~\ref{lemma:A^T}, we have $\Vt'=\Vt-A^\transp\cdot\Wt$.) We can now write down the fiber $\Fib(\V,Z)$ as the set of pairs $(a,b)$ for which $\V'=\V+A\cdot\W$ is totally nonnegative. There are ${5\choose 2}=10$ minors of $\V'$, and each of them turns out to be a linear polynomial in the variables $a,b$. For example,
$\Delta_{13}(\V')=(a+1)(3b+1)-3ab=a+3b+1.$

The ten lines $\Delta_I(\V')=0$ in the $(a,b)$-plane are shown in Figure~\ref{fig:fiber}. For example, the line $\Delta_{13}(\V')=0$ is the black line passing through the points $(-1,0)$ and $(0,-\frac13)$. The region where $\V'$ is totally nonnegative is the pentagon shaded in blue. The vertices of this pentagon correspond to $\V'$ being in $\pc k {f'}$ for $f'$ equal to $f$ up to rotation (in particular, the vertex $(0,0)$ corresponds to $f$ itself). Thus $\Fib(\V,Z)$ intersects $\pc k {f'}$ for all such $f'$, and hence $f$ is not $(n,k,m)$-compatible with any of its four rotations. Observe that the inverses of these five permutations correspond to the triangles in a pentagon with vertex sets $\{1,3,4\}$, $\{2,4,5\}$, $\{1,3,5\}$, $\{1,2,4\}$, and $\{2,3,5\}$. No two of these five triangles can appear together in a triangulation of the pentagon, in agreement with part~\eqref{item:main:compatible} of Theorem~\ref{thm:main}.

On the other hand, $\Vt'$ is a $1\times 5$ matrix and thus has only five minors. However, the $(a,b)$-region where all of them are nonnegative turns out to be exactly the same as the shaded pentagon in Figure~\ref{fig:fiber}. More precisely, for each $j\in\Z$, we have $\Delta_{\{j\}}(\Vt'\cdot\Deltas)=\Delta_{\{j,j+1\}}(\V')$ (see Lemma~\ref{lemma:consecutive_V_W} for a general statement). For example,
\[\Delta_1(\Vt'\cdot\Deltas)=\Delta_{12}(\V')=a-b+1.\]
Since the boundary of the shaded region in Figure~\ref{fig:fiber} only involves circular minors of $\V'$ being zero, we see that the fibers $\Fib(\V,Z)$ and $\Fib(\Vt,\Zt)$ are equal as regions in the $(a,b)$-plane. However, the five vertices of the pentagon $\Fib(\Vt,\Zt)$ correspond to the five rotations of $f^{-1}$.
The main idea of our proof in Section~\ref{sec:stw_properties} is based on the observations made in this example.

\subsection{Canonical forms for lower cells}\label{sec:ex_forms}

  Let $k=2$, $\l=1$, $m=2$, $n=5$, $f=[2,3,1,5,4]$. Consider parametrizations of $\pc k f$ and $\Grtperp{k+m}n$ by
  \[\V:=\begin{pmatrix}
      1 & 0 & 0 & -a & -c\\
      0 & 1 & 0 & b & 0
    \end{pmatrix},\quad  \W:=\begin{pmatrix}
      p & -q & 1 & -r & s
    \end{pmatrix},\quad a,b,c,p,q,r,s>0.\]
  We can take $\Ccal:=\{\{1,2\},\{2,4\},\{4,5\},\{2,5\}\}$ to be a cluster for $\V$:
  \[\Delta_{12}(\V)=1,\quad \Delta_{24}(\V)=a,\quad \Delta_{45}(\V)=bc,\quad \Delta_{25}(\V)=c,\]
  and thus
  \[\om k f=\pm\frac{da\wedge d(bc)\wedge dc}{abc^2}=\pm\frac{da\wedge db\wedge dc}{abc}.\]
  Since $\oml=\pm\frac{dp\wedge dq\wedge dr\wedge ds}{pqrs}$, the form $\om k f\wedge \oml$ is given by
  \begin{equation}\label{eq:ex_om_wedge_om}
\om k f\wedge \oml=\pm\frac{da\wedge db\wedge dc\wedge dp\wedge dq\wedge dr\wedge ds}{abcpqrs}.
  \end{equation}
 Let us now see what happens to this form when we apply the stacked twist map. Let $(\Wt,\Vt):=\stw(\V,\W)$. We find
  \[\Wt:=\begin{pmatrix}
-\frac{p}{c p + s} & 1 & -\frac{b}{a} & \frac{s}{c} & 0 \\
\frac{q}{c p + s} & 0 & -1 & \frac{c r + a s}{b c} & -\frac{1}{b} \\
    \end{pmatrix},\quad \Vt:=\begin{pmatrix}
\frac{1}{c p + s} & 0 & 0 & 1 & 0
\end{pmatrix}.\]
Note that $f^{-1}=[3,1,2,5,4]$, and indeed $\Vt\in\pc \l{f^{-1}}$, while $\Wt\in\Grtperp{\l+m}n$ because $\alt(\Wt)\in\Grtp(k,n)$, see~\eqref{eq:alt_duality}. Choosing a cluster $\Ccal':=\{\{4\},\{1\}\}$ for $\pc \l{f^{-1}}$, we get that the pullback of $\om \l{f^{-1}}$ to $\pc k f\times \Grtperp{k+m}n$ under $\stw$ is
\begin{equation}\label{eq:ex_stw_ast_0}
\stw^\ast\om \l{f^{-1}}=\pm\frac{d \left(\frac{1}{c p + s}\right)}{\frac{1}{c p + s}}=\frac{\pm\left(p\, dc  +c\,dp+ds\right)}{c p + s}.
\end{equation}
\def\wdg{\,}
Let us now compute the pullback of $\omk$ to $\pc k f\times \Grtperp{k+m}n$ under $\stw$. Applying row operations to $\Wt$, we transform it into
\[\Wt'=\begin{pmatrix}
-\frac{a p + b q}{{\left(c p + s\right)} a} & 1 & 0 & -\frac{r}{a} & \frac{1}{a} \\
-\frac{q}{c p + s} & 0 & 1 & -\frac{c r + a s}{b c} & \frac{1}{b}
\end{pmatrix}.\]
The submatrix of $\Wt'$ with column set $\{2,3\}$ is an identity matrix, and thus we can use Definition~\ref{dfn:canonical_form_top_cell} (together with the fact that $\omk$ is cyclically invariant) to calculate
\begin{equation}\label{eq:ex_stw_ast_1}
\stw^\ast \omk=\pm\frac{d^{2\times 3}\begin{pmatrix}
-\frac{a p + b q}{{\left(c p + s\right)} a}  & -\frac{r}{a} & \frac{1}{a} \\
-\frac{q}{c p + s}  & -\frac{c r + a s}{b c} & \frac{1}{b}
\end{pmatrix}}{\pmin k {\Wt'}}.
\end{equation}
Computing the circular minors of $\Wt'$, we get
\[\Delta_{12}(\Wt')=\frac{q}{c p + s},\quad \Delta_{23}(\Wt')=1,\quad \Delta_{34}(\Wt')=\frac{r}{a},\quad \Delta_{45}(\Wt')=\frac{s}{bc},\quad \Delta_{15}(\Wt')=\frac{-p}{b(c p + s)}. \]
The denominator in the right hand side of~\eqref{eq:ex_stw_ast_1} is therefore
\[\pmin k{\Wt'}=\frac{-p q r s}{\left(c p + s\right)^2 a b^2 c}.\]
A longer computation is needed to find the numerator in the right hand side of~\eqref{eq:ex_stw_ast_1}, which in the end produces the following answer (the symbol $\wedge$ is omitted):
\begin{equation}\label{eq:ex_stw_ast_2}
\stw^\ast \omk=\pm \left(\frac{d a \wdg d b \wdg d c \wdg d p \wdg d r \wdg d s}{a b c p r s}  -\frac{d a \wdg d b \wdg d c \wdg d q \wdg d r \wdg d s}{a b c q r s}  + \frac{s\,d a \wdg d b \wdg d c \wdg d p \wdg d q \wdg d r}{{\left(c p + s\right)} a b c p q r}  + \frac{d a \wdg d b \wdg d p \wdg d q \wdg d r \wdg d s}{{\left(c p + s\right)} a b p q r}\right).
\end{equation}
The reader is invited to check that the form $\stw^\ast(\omk\wedge \om \l{f^{-1}})=\stw^\ast\omk\wedge \stw^\ast\om \l{f^{-1}}$, obtained by taking the wedge of the right hand sides of~\eqref{eq:ex_stw_ast_0} and~\eqref{eq:ex_stw_ast_2}, equals to $\pm\om k f\wedge \oml$, which is given on the right hand side of~\eqref{eq:ex_om_wedge_om}.

\section{Subdivisions and the universal amplituhedron}\label{sec:subdivisions}
In this section, we introduce a more general set up. First, similarly to triangulations of $\ampl$ defined in Section~\ref{sec:triang-ampl} we define its \emph{subdivisions}. Second, we describe the \emph{universal amplituhedron} (originally studied in the complex algebraic setting in~\cite[Section~18.1]{LamCDM}), and define its subdivisions as well. The universal amplituhedron provides a convenient framework to formulate the results of Section~\ref{sec:stw}, as well as some other results regarding canonical forms as we do in Section~\ref{sec:ampl_form}.

\subsection{Subdivisions}

Let $Z \in \Grtp(k+m,n)$. We denote by $\leq$ the \emph{Bruhat order} on $\Aff$, defined as follows. We say that $t\in\Aff$ is a \emph{transposition} if we have $t(i)\neq i$ for precisely two indices $i\in[n]$, and then we define the covering relation $\lessdot$ of the Bruhat order by
\begin{equation}\label{eq:bruhat_dfn}
f\lessdot g\quad\text{if and only if}\quad  f=g\circ t\text{ for some transposition $t\in\Aff$ and } \inv(g) = \inv(f) + 1.
\end{equation}
Here $\circ$ denotes the composition of bijections $\Z\to\Z$. Given two affine permutations $f,g\in\Aff(-k,n-k)$, we have $f \leq g$ if and only if $\pctnn k f \subseteq \pctnn k g$, see e.g.~\cite[Theorem~8.1]{LamCDM}.

\begin{definition}\label{dfn:subdivision}
Let $f_1,\dots,f_N\in\Aff(-k,n-k)$ be a collection of pairwise $Z$-compatible affine permutations.  We say that $f_1,\dots,f_N$ \emph{form a $Z$-subdivision} of the amplituhedron if the union of the images $Z(\pc k {f_1}),\dots,Z(\pc k {f_N})$ forms a dense subset of $\ampl$,  and we have $\dim (Z(\pc k{f_s}))=km$ for each $1\leq s\leq N$.

More generally, suppose that $g \in \Aff(-k,n-k)$ and $f_1,\dots,f_N\in\Aff(-k,n-k)$ is a collection of pairwise $Z$-compatible affine permutations satisfying $f_s \leq g$ for all $1 \leq s \leq N$.  We say that $f_1,\dots,f_N$ \emph{form a $Z$-subdivision} of  $Z(\pctnn k g)$ if the images $Z(\pc k {f_1}),\dots,Z(\pc k {f_N})$ form a dense subset of $Z(\pctnn k g)$, and we have $\dim (Z(\pc k{f_s}))=\dim (Z(\pc k g))$ for each $1\leq s\leq N$.

We say that $f_1,\dots,f_N$ \emph{form an $(n,k,m)$-subdivision} of the amplituhedron (resp., of $g$) if they form a $Z$-subdivision of the amplituhedron (resp., of $Z(\pctnn k g)$) for any $Z\in\Grtp(k+m,n)$.
\end{definition}
  
It thus follows that a \emph{($Z$- or $(n,k,m)$-) triangulation} of the amplituhedron is a ($Z$- or $(n,k,m)$-) subdivision $f_1,\dots,f_N\in\Aff(-k,n-k)$ of $\ampl$ such that the permutations $f_1,\dots,f_N$ are ($Z$- or $(n,k,m)$-)admissible. 

The following result gives an analog to part~\eqref{item:main:triang} of Theorem~\ref{thm:main}, with the same proof which we omit.

\begin{theorem}\label{thm:main_subd}
  Let $f_1,\dots,f_N\in\Aff(-k,\l)$ be a collection of affine permutations. Then they form an $(n,k,m)$-subdivision of the amplituhedron (resp., of $g \in \Aff(-k,\l)$) if and only if $f_1^{-1},\dots,f_N^{-1}\in\Aff(-\l,k)$ form an $(n,\l,m)$-subdivision of the amplituhedron (resp., of $g^{-1} \in \Aff(-\l,k)$).
\end{theorem}

\subsection{The universal amplituhedron}\label{sec:universal_ampl}

Let $\Fl(\l,k+\l;n)$ be the \emph{$2$-step flag variety} consisting of pairs of subspaces $W \subset U \subset \R^n$ such that $\dim(W) =\l$ and $\dim(U) = k+\l$.  Define the {\it universal amplituhedron} $\am_{n,k,m} \subset \Fl(\l,k+\l;n)$ to be the closure (in the \emph{analytic topology}\footnote{We refer to the standard topology on the real manifolds $\Gr(k,n)$ and $\Fl(\l,k+\l;n)$ as the \emph{analytic topology} in order to distinguish it from \emph{Zariski topology} which we will later use on their complex algebraic versions.})
\[ \am_{n,k,m}:= \overline{ \{(W,U) \mid W\in \Grtperp{k+m}n \text{ and  $\Span(V,W) = U$ for some $V \in \Grtnn(k,n)$}\}}.\]

Let
\[\pi: \Fl(\l,k+\l;n) \to \Gr(\l,n)\]
be the projection $(W,U)\mapsto W$.  Let us fix $Z \in \Grtp(k+m,n)$ and let $W:=Z^\perp$. The fiber $\pi^{-1}(W) \cap \am_{n,k,m}$ can be identified with the $Z$-amplituhedron $\ampl$ via the map sending $U$ to the row span of the $(k+m)\times (k+m)$ matrix $U\cdot Z^\transp$ (which has rank $k$).
Thus we have
$$
\dim(\am_{n,k,m}) = (k+m)(n-k-m) + km = km + k\l + \l m.
$$

\newcommand\amcellC[2]{\am^{\C}_{\shift{#1}{#2}}}
\newcommand\amcell[2]{\am^{>0}_{\shift{#1}{#2}}}
\newcommand\amcelltnn[2]{\am^{\ge0}_{\shift{#1}{#2}}}

We have a map
\[\Zcal: \Grtnn(k,n)\times\Grtperp{k+m}n \to \am_{n,k,m}\]
given by $(V,W) \mapsto (W,\Span(V,W))$.
Since the matrix $U\cdot Z^\transp$ has rank $k$, and the subspace $U$ contains the kernel $W$ of $Z$, it follows that the subspace $\Span(V,W)$ always has dimension $k+\l$.
For $f \in \Aff(-k,n-k)$, define 
\begin{align*}
\amcell k f & := \Zcal(\pc k f \times \Grtperp{k+m}n) \subset \am_{n,k,m}, \\
\amcelltnn k f & := \overline{\amcell k f} \subset \am_{n,k,m}.
\end{align*}

Define $\uamk \subset \am_{n,k,m}$ by 
\[ \uamk:= \{(W,U) \mid W\in \Grtperp{k+m}n \text{ and  $\Span(V,W) = U$ for some $V \in \Grtnnm(k,n)$}\}.\]

The map $\Zcal$ restricts to a map $\Zcal: \Grtnnl(k,n)\times\Grtperp{k+m}n\to \uamk$.  Let $\Zcal^\op:\Grtperp{\l+m}n \times \Grtnnm(k,n) \to \uaml$ denote the similar map sending $(\Wt,\Vt)$ to $(\Wt,\Span(\Wt, \Vt))$.

\begin{theorem}
  The stacked twist map
  \[\stw: \Grtnnl(k,n)\times\Grtperp{k+m}n \to \Grtperp{\l+m}n\times \Grtnnl(\l,n)\]
  descends to a homeomorphism $\tstw: \uamk \to \uaml$, forming a commutative diagram
\begin{equation}\label{eq:commutative_diagram}
\begin{tikzcd}
\Grtnnl(k,n)\times\Grtperp{k+m}n \arrow[d,"\Zcal"] \arrow[r, "\stw"] & \Grtperp{\l+m}n\times \Grtnnl(\l,n)  \arrow[d,"\Zcal^\op"] \\
\uamk \arrow[r, "\tstw"] &  \uaml
\end{tikzcd}
\end{equation}

The map $\tstw$ restricts to homeomorphisms $\tstw:\amcell k f \to \amcell \l {f^{-1}}$ for each $f\in\Aff(-k,\l)$.
\end{theorem}
\begin{proof}
Suppose $(W,U) \in \uamk$.  Let $V \in \Grtnnl(k,n)$ be chosen so that $U = \Span(V,W)$.  Then $\tstw(W,U):=(\Wt,\Ut)$, where $\theta(V,W) = (\Wt,\Vt)$ and $\Ut = \Span(\Wt,\Vt)$.  It follows from Proposition~\ref{prop:fibers} that $(\Wt,\Ut)$ does not depend on the choice of $V$.  Thus $\tstw$ is well-defined.  (Alternatively, the rational map $\tstw: \Fl(\ell,k+\ell;n) \dashedrightarrow \Fl(k,k+\ell;n)$ can be defined directly from Lemma \ref{lemma:inverse_transpose}.)
The last statement is immediate from Theorem \ref{thm:stw}.
\end{proof}

\subsection{Subdivisions of the universal amplituhedron}
We recast the notions of $(n,k,m)$-triangulations/subdivisions in the language of the universal amplituhedron.
\begin{lemma}\leavevmode\label{lemma:subd:admissible_compatible}
  \begin{enumerate}
  \item   Let $f\in\Aff(-k,n-k)$ be an affine permutation. Then $f$ is $(n,k,m)$-admissible if and only if $\pi^{-1}(W)\cap \amcell k f$ contains at most $1$ point for all $W\in\Grtperp{k+m}n$.
\item   Let $f,g\in\Aff(-k,n-k)$ be two affine permutations. Then $f$ and $g$ are $(n,k,m)$-compatible if and only if $\amcell k f$ and $\amcell k g$ do not intersect. 
  \end{enumerate}
\end{lemma}
\begin{proof}
See the proof of Theorem~\ref{thm:main} (parts~\eqref{item:main:admissible} and~\eqref{item:main:compatible}) at the end of Section~\ref{sec:stw}.
\end{proof}

The following is an analog of Definition~\ref{dfn:subdivision}.

\begin{definition}
Let $f_1,\dots,f_N\in\Aff(-k,n-k)$ be a collection of pairwise $(n,k,m)$-compatible affine permutations.  We say that $f_1,\dots,f_N$ \emph{form a subdivision} of the universal amplituhedron if the union of $\amcell k {f_1},\dots,\amcell k {f_N}$ is a dense subset of $\am_{n,k,m}$, and we have $\dim (\amcell k{f_s})=\dim(\am_{n,k,m})$ for each $1\leq s\leq N$.

More generally, suppose that $g \in \Aff(-k,n-k)$ and $f_1,\dots,f_N\in\Aff(-k,n-k)$ is a collection of pairwise $(n,k,m)$-compatible affine permutations satisfying $f_s \leq g$ for all $1 \leq s \leq N$.  We say that $f_1,\dots,f_N$ \emph{form a subdivision} of $\amcelltnn k g$ if the union of $\amcell k {f_1},\dots,\amcell k {f_N}$ is a dense subset of $\amcelltnn k g$, and we have $\dim (\amcell k{f_s})=\dim(\amcell k g)$ for each $1\leq s\leq N$.
\end{definition}

\begin{lemma}\label{lemma:subd:universal}\leavevmode
  \begin{itemize}
  \item Let $f_1,\dots,f_N\in\Aff(-k,n-k)$ be a collection of affine permutations.  Then $f_1,\dots,f_N$ form an $(n,k,m)$-subdivision of the amplituhedron if and only if they form a subdivision of the universal amplituhedron.
  \item Suppose that $g \in \Aff(-k,n-k)$ and $f_1,\dots,f_N\in\Aff(-k,n-k)$ is a collection of affine permutations. Then $f_1,\dots,f_N$ form an $(n,k,m)$-subdivision of $g$ if and only if they form a subdivision of $\amcelltnn k g$.
  \end{itemize}
\end{lemma}
\begin{proof}
Let
$$
\am'_{n,k,m}:=
 \{(W,U) \mid W\in \Grtperp{k+m}n \text{ and  $\Span(V,W) = U$ for some $V \in \Grtnn(k,n)$}\}
 $$
 so that by definition $\am_{n,k,m}$ is the closure of $\am'_{n,k,m}$.
  Suppose that $f_1,\dots,f_N$ form an $(n,k,m)$-subdivision of the amplituhedron. Thus the images $Z(\pc k {f_1}),\dots,Z(\pc k {f_N})$ form a dense subset of $\ampl$ for any $Z\in\Grtp(k+m,n)$. We claim that $\amcell k{f_1},\dots,\amcell k{f_N}$ form a dense subset of $\am'_{n,k,m}$, from which it follows that their union is dense in $\am_{n,k,m}$.  Indeed, take a pair $(W,U)\in\am'_{n,k,m}$ and let $Z:=W^\perp\in\Grtp(k+m,n)$. There exists $1\leq s\leq N$ such that $U\cdot Z^\transp$ belongs $Z(\pctnn k {f_s})$, and thus $(W,U)$ belongs to the closure of $\amcell k{f_s}$.
    
Conversely, suppose that $\amcell k{f_1},\dots,\amcell k{f_N}$ form a dense subset of $\am_{n,k,m}$. Fix $Z\in\Grtp(k+m,n)$, let $W:=Z^\perp$, $V\in\Grtnn(k,n)$, $U:=\Span(V,W)$, and consider the point $Y:=V\cdot Z^\transp$ of the amplituhedron $\ampl$. We want to show that $Y$ belongs to the closure of $Z(\pc k {f_s})$ for some $1\leq s\leq N$.
We know that $(W,U)$ belongs to the closure of $\amcell k{f_s}$ for some $1\leq s\leq N$. By assumption, there exists a sequence $(W_1,U_1),(W_2,U_2),\dots$ converging to $(W,U)$ in $\Fl(\l,k+\l;n)$ such that $(W_i,U_i)\in\amcell k{f_s}$ for all $i\geq 1$. By the definition of $\amcell k{f_s}$, for each $i\geq1$ there exists $V_i\in\pc k{f_s}$ such that $U_i = \Span(V_i,W_i)$. Let us now consider the sequence $(W,\Span(V_i,W))\in\amcell k{f_s}$ for $i\geq 1$.  We claim that a subsequence of the points $\Span(V_i,W)$ converges to $U$ in $\Gr(k+\l,n)$, and thus $(W,U)$ belongs to the closure of $\amcell k{f_s}$.  The points $V_i$ all belong to the compact set $\pctnn k {f_s}$ and therefore a subsequence $V_{i_1},V_{i_2},\ldots$ converges to $V' \in \pctnn k {f_s}$. Any subspace $V''\in\Gr(k,n)$ in the neighborhood of $V'$ satisfies $V''\cap W=\{0\}$, and thus  the map $V''\mapsto\Span(V'',W)$ is continuous at $V'$.  We get that $\Span(V_{i_1},W) = U_{i_1}, \Span(V_{i_2},W) = U_{i_2}, \ldots$ converges to $\Span(W,V')$ in $\Gr(k+\l,n)$, and we must also have $\Span(W,V') = U$.

We have shown the first part of the lemma, and the proof of the second part is completely analogous.
\end{proof}

\begin{definition}
A \emph{triangulation} of the universal amplituhedron (resp., of $\amcelltnn k g$) is a subdivision $f_1,\dots,f_N\in\Aff(-k,n-k)$ of the universal amplituhedron that consists of $(n,k,m)$-admissible affine permutations.
\end{definition}

The following result follows from Lemmas~\ref{lemma:subd:admissible_compatible} and~\ref{lemma:subd:universal} together with Theorem~\ref{thm:main}.
\begin{corollary}\label{cor:triangulations_uampl}
  Let $f_1,\dots,f_N\in\Aff(-k,n-k)$ be a collection of affine permutations. The following are equivalent:
  \begin{itemize}
  \item  $f_1,\dots,f_N$ form an $(n,k,m)$-triangulation of the amplituhedron;
  \item $f_1,\dots,f_N$ form a triangulation of the universal amplituhedron $\am_{n,k,m}$;
  \item $f_1^{-1},\dots,f_N^{-1}$ form a triangulation of the universal amplituhedron $\am_{n,\l,m}$;
  \end{itemize}
\end{corollary}

A similar result applies to triangulations of  $g \in \Aff(-k,n-k)$.

\section{The amplituhedron form}\label{sec:ampl_form}

\def\From{A}
\def\To{B}
\newcommand\pcC[2]{\Pi_{\shift{#1}{#2}}^\C}
\def\GrC{\Gr_\C}
\def\FlC{\Fl_\C}
\def\amplC{\am^\C_{n,k,m}(Z)}
\newcommand\omzce[2]{\omega_{Z(\ce{#1}{#2})}}
\def\omzcekf{\omzce k f}

We review the definition of the \emph{amplituhedron form}, which is a rational
differential form on the Grassmannian. In Section~\ref{sec:ampl_form_example}, the constructions are illustrated by computations for $k=\l=1$, $m=2$, and $n=4$.

It is natural to work over the complex numbers, and we denote by $\GrC(k,n)$ and $\FlC(a,b;n)$ the complex Grassmannian and the complex $2$-step flag variety, respectively. The map $Z$ induces a rational map $Z: \GrC(k,n) \dashedrightarrow \GrC(k,k+m)$, the image of which is equal to $\GrC(k,k+m)$.  Given an affine permutation $f\in\Aff(-k,n-k)$, denote by $\pcC k f\subseteq\GrC(k,n)$ the Zariski closure of $\pc k f$ in $\GrC(k,n)$, called the \emph{positroid variety}. See Definition~\ref{dfn:pos_cell} for an alternative description. 
\subsection{Degree $1$ cells}
\begin{definition}\label{def:degree}
Let $Z \in \GrC(k+m,n)$.  We say that an affine permutation $f\in\Aff(-k,n-k)$ \emph{has $Z$-degree $1$} if $\dim \pcC k f=\dim Z(\pcC k f)=km$ and the rational map $Z:\pcC k f\dashedrightarrow \GrC(k,k+,m)$ is birational, that is, it restricts to an isomorphism between Zariski open subsets $\From\subset \pcC k f$ and $\To\subset\GrC(k,k+m)$.

We say that an affine permutation $f\in\Aff(-k,n-k)$ \emph{has degree $1$} if it has $Z$-degree $1$ for $Z$ belonging to a nonempty Zariski open subset of $\GrC(k+m,n)$.
\end{definition}
Thus in order for $f$ to have $Z$-degree $1$, we must have $\inv(f)=k\l$. We note that if $f$ satisfies $\dim \pcC k f=\dim Z(\pcC k f)=km$ for a generic $Z\in\GrC(k,k+m)$ then $f$ is said to have \emph{kinematical support}, see~\cite[Section~10]{abcgpt}.

\begin{remark}\label{rem:genericdegree}
Suppose that $f$ has degree $1$.  Since $\Grtp(k+m,n)$ is Zariski dense in $\GrC(k+m,n)$, $f$ has $Z$-degree $1$ for $Z$ belonging to an open dense subset of 
$\Grtp(k+m,n)$ in the analytic topology.
\end{remark}
 
\begin{remark}   \label{rem:Zdiffeo}
Suppose that $f$ has $Z$-degree $1$.  Since $\pc k f$ is Zariski dense in $\pcC k f$, it follows that $Z$ restricts to a diffeomorphism on an open dense subset of $\pc k f$ in the analytic topology.  
\end{remark}
Thus if $f$ has $Z$-degree $1$ then $Z$ is injective on an open dense subset of the positroid cell $\pc k f$, but not necessarily on the whole positroid cell. 

\begin{conjecture}\label{conj:degree_1}
Let $f\in\Aff(-k,n-k)$ be an affine permutation with $\inv(f)=k\l$. Then the following are equivalent:
\begin{enumerate}[(a)]
\item \label{item:f_is_nkm_adm}
$f$ is $(n,k,m)$-admissible, 
\item \label{item:f_has_deg_1}
$f$ has degree $1$,
\item\label{item:f_has_Z-deg_1}
$f$ has $Z$-degree $1$ for all $Z \in \Grtp(k+m,n)$.
\end{enumerate}
In each case, $Z$ is a diffeomorphism $\pc kf\to Z(\pc kf)$.
\end{conjecture}
It is clear that~\eqref{item:f_has_Z-deg_1} implies~\eqref{item:f_has_deg_1}.  In what follows, we restrict ourselves to only $(n,k,m)$-admissible affine permutations that have degree $1$.  It is not difficult to see that Conjecture \ref{conj:degree_1} holds for $k =1$, since in this case every affine permutation $f\in\Aff(-k,n-k)$ with $\inv(f)=k\l$ satisfies~\eqref{item:f_is_nkm_adm}, \eqref{item:f_has_deg_1}, and~\eqref{item:f_has_Z-deg_1}. Let us now explain why Conjecture~\ref{conj:degree_1} also holds for $\l= 1$. Let $\l=1$, and consider an affine permutation $f\in\Aff(-k,\l)$ with $\inv(f)=k\l$. Then $f$ satisfies~\eqref{item:f_is_nkm_adm} by Theorem~\ref{thm:main}, part~\eqref{item:main:admissible}. Next, $f$ satisfies~\eqref{item:f_has_deg_1} by Proposition~\ref{prop:degre_1_inverse} below. Finally, we have already noted above that $f^{-1}$ satisfies~\eqref{item:f_has_Z-deg_1}, i.e., that $f^{-1}\in\Aff(-\l,k)$ has $\Zt$-degree $1$ for all $\Zt \in \Grtp(\l+m,n)$. In particular, it is easy to see that for any $\Vt\in\Grtp(\l,n)$, there exists a unique $\Vt'\in \pcC \l{f^{-1}}$ such that $\Vt\cdot \Zt^\transp=\Vt'\cdot \Zt^\transp$. By Lemma~\ref{lemma:A^T}, it follows that for any $Z\in\Grtp(k+m,n)$ and $\V\in\Grtp(k,n)$, there exists a unique $\V'\in \pcC kf$ such that $\V\cdot Z^\transp=\V'\cdot Z^\transp$. Since this holds for all $\V\in\Grtp(k,n)$, it must hold for $\V$ belonging to a Zariski dense subset of $\GrC(k,n)$, and thus $f$ indeed has $Z$-degree $1$ for any $Z\in\Grtp(k+m,n)$. This shows that Conjecture~\ref{conj:degree_1} holds for $\l=1$.

We will see in Lemma~\ref{lemma:boundary_ineq} that if $f\in\Aff(-k,n-k)$ has degree $1$ then we must have $f\in\Aff(-k,\l)$.  By a result of~\cite{Lam14}, $f$ has degree $1$ if and only if the coefficient of a certain rectangular Schur function in the corresponding affine Stanley symmetric function is equal to $1$. An important corollary of this is the following result that we prove in Section~\ref{sec:affst}. 
\begin{proposition}\label{prop:degre_1_inverse}
An affine permutation $f\in\Aff(-k,\l)$ has degree $1$ if and only if $f^{-1}\in\Aff(-\l,k)$ has degree $1$.
\end{proposition}

\def\triang{\Tcal}
\newcommand\omtriang[1]{\omega^{\parr\triang}_{#1}}
\newcommand\omtriangprime[1]{\omega^{\parr{\triang'}}_{#1}}
\begin{definition}
Given $Z\in\Grtp(k+m,n)$,  we say that $f_1,\dots,f_N\in\Aff(-k,n-k)$ form a \emph{$Z$-triangulation of degree $1$} if they form an $Z$-triangulation and $f_s$ has $Z$-degree $1$ for $1\leq s\leq N$.

  We say that $f_1,\dots,f_N\in\Aff(-k,n-k)$ form an \emph{$(n,k,m)$-triangulation of degree $1$} if they form an $(n,k,m)$-triangulation of the amplituhedron and $f_s$ has degree $1$ for $1\leq s\leq N$.
\end{definition}
Note that (without knowing Conjecture \ref{conj:degree_1}) it may not be true that an $(n,k,m)$-triangulation of degree $1$ is a $Z$-triangulation of degree $1$ for all $Z\in\Grtp(k+m,n)$, but only for all $Z$ belonging to an open dense subset of $\Grtp(k+m,n)$.

\begin{corollary}\label{cor:degree_1_triang}
Let $\triang=\{f_1,\dots,f_N\}\subset \Aff(-k,\l)$ be an $(n,k,m)$-triangulation of degree $1$. Define $\triang':=\{f_1^{-1},\dots,f_N^{-1}\}\subset\Aff(-\l,k)$. Then $\triang'$ is an $(n,\l,m)$-triangulation of degree $1$.
\end{corollary}
\begin{proof}
Follows directly from Proposition~\ref{prop:degre_1_inverse} combined with the proof of Theorem~\ref{thm:main}.
\end{proof}

\subsection{The amplituhedron form}

For us, the crucial property of a degree $1$ map is that it allows one to define the \emph{pushforward form} $\omzcekf$ on $\GrC(k,k+m)$ via a pullback. 
Suppose that $f\in\Aff(-k,n-k)$ has $Z$-degree $1$ and $\inv(f) = k\ell$.  Thus $Z$ restricts to an isomorphism between Zariski open subsets $\From\subset \pcC k f$ and $\To\subset\GrC(k,k+m)$.  We define the rational top form $\omzcekf$ on $\To$, and by extension on $\GrC(k,k+m)$, as the pullback
\begin{equation}\label{eq:pushforward}\omzcekf=Z_\ast\om k f:=(Z^{-1})^\ast \om k f,
\end{equation}
where $Z^{-1}:\To\to\From$ denotes the inverse of the isomorphism $Z:\From\to\To$.

Recall from Section~\ref{sec:can_form_lower} that for $f\in\Aff(-k,n-k)$, the canonical form $\om k f$ on $\pc k f$ is defined up to a sign. For the purposes of the amplituhedron form, we need to add up the forms $\omzcekf$ for various $f$.  Thus it is necessary pick a particular sign of $\om k f$ for each $f\in\Aff(-k,n-k)$.

\begin{proposition}\label{prop:Jacobian_constant_sign}
Let $f\in\Aff(-k,n-k)$ be an affine permutation of $Z$-degree $1$.  Recall that by Remark \ref{rem:Zdiffeo}, there exists an open dense subset $A\subset\pc kf$ such that the restriction of $Z$ to $A$ is a diffeomorphism. If  $f$ is $Z$-admissible then the Jacobian of $Z$ on $A$ has constant sign.
\end{proposition}

Thus for an $(n,k,m)$-admissible affine permutation of degree $1$, we can distinguish between $Z$ being orientation reversing or orientation preserving on the real manifold $\pc kf$.

\def\sign{\varepsilon}

Let us choose some rational top form $\omref_{k,k+m}$ on $\Gr(k,k+m)$ that is non-vanishing on $\ampl$. This is always possible since $\ampl$ lies completely within some open Schubert cell (diffeomorphic to the orientable manifold $\R^{km})$ of $\Gr(k,k+m)$, see \cite[Proposition 15.2 and Lemma 15.6]{LamCDM}. We call $\omref_{k,k+m}$ the \emph{reference form}.

\begin{definition}\label{dfn:signs}
Suppose that $f\in\Aff(-k,n-k)$ has $Z$-degree $1$. We say that the form $\om kf$ is \emph{positive with respect to $\omref_{k,k+m}$} if the forms $\omzcekf$ and $\omref_{k,k+m}$ have the same sign when restricted to $Z(\pc kf)$ (whenever both are nonzero).
\end{definition}

We note that $\om kf$ is nonvanishing on $\pc kf$.  By Proposition~\ref{prop:Jacobian_constant_sign}, we may and will pick a sign for $\om kf$ so that the pushforward form $\omzcekf$ is positive with respect to the fixed reference form $\omref_{k,k+m}$ on $\Gr(k,k+m)$.

\begin{definition}
  For $Z\in\Grtp(k+m,n)$ and a $Z$-triangulation $\triang=\{f_1,\dots,f_N\}$ of degree $1$, define the \emph{amplituhedron form} $\omtriang{\ampl}$ on $\GrC(k,k+m)$ by
  \begin{equation}\label{eq:ampl_form_dfn}
\omtriang{\ampl}:=\sum_{s=1}^N\omega_{Z(\ce k {f_s})}.
  \end{equation}
\end{definition}
We emphasize again that the sign of every summand  in the right hand side of the above equation is chosen so that the signs of $\omega_{Z(\ce k {f_s})}$ and $\omref_{k,k+m}$ coincide on $Z(\pc k{f_s})$ for each $1\leq s\leq N$.  The following has been numerically verified in the physics literature in a number of cases when $m = 4$.
\begin{conjecture}\label{conj:ampl_form}
The form $\omtriang{\ampl}$ does not depend on the choice of a triangulation $\triang$.
\end{conjecture}
When Conjecture \ref{conj:ampl_form} holds, we denote by $\omega_{\ampl}$ the triangulation independent \emph{amplituhedron form}.  This conjectural amplituhedron form is the one presented in the original definition of the amplituhedron \cite{AT}.  A different conjectural characterization of the amplituhedron form via ``positive geometries'' is given in \cite{ABL}.  It is an important open problem to give a formula for $\omega_{\ampl}$ without referring to triangulations. See Conjecture~\ref{conj:ampl_form_poly} for related discussion.

\subsection{The universal amplituhedron form}

\def\omreflkl{\omref_{\Fl(\l,k+\l;n)}}
The form $\omtriang{\ampl}$ is a top form on $\GrC(k,k+m)$ that depends on $Z$. We now let $Z$ vary over $\Grtp(k+m,n)$, and define a closely related top form $\omtriang{\uamk}$ on $\FlC(\l,k+\l;n)$. 

As in the construction of $\omref_{k,k+m}$, each $Z$-amplituhedron $\ampl$ is contained in an open Schubert cell $C(Z)$ of $\Gr(k,k+m)$.  Thus the interior of the universal amplituhedron is contained in the orientable submanifold of $\Fl(\l,k+\l;n)$ which maps to $\Grtp(k+m,n)$ with fiber $C(Z)$ over each $Z \in \Grtp(k+m,n)$.  (Note that $\Grtp(k+m,n)$ is diffeomorphic to an open ball and is in particular orientable and contractible.)
We may thus choose a reference top form $\omreflkl$ on $\Fl(\l,k+\l;n)$ so that it is non-vanishing on the interior of $\am_{n,k,m}$.

\begin{proposition}\label{prop:Zcal_degree}
  Suppose that $f\in\Aff(-k,n-k)$ has degree $1$. Then the rational map
  \[\Zcal:\pcC k f \times\GrC(\l,n)\dashedrightarrow \FlC(\l,k+\l;n)\]
  sending $(\V,\W)$ to $(\W, \Span(\V,\W))$ is a birational map.
\end{proposition}

\begin{proposition}\label{prop:Zcal_Jacobian_constant_sign}
Let $f\in\Aff(-k,n-k)$ be an $(n,k,m)$-admissible affine permutation of degree $1$.  Then there exists an open dense subset $A\subset\pc kf \times \Grtperp{k+m}n$ such that the restriction of $\Zcal$ to $A$ is a diffeomorphism, and the Jacobian of $\Zcal$ on $A$ has constant sign.
\end{proposition}

By Proposition~\ref{prop:Zcal_degree}, if $f$ has degree $1$ then the pushforward form $\Zcal_* (\om kf \wedge \oml)$ on $\FlC(\l,k+\l;n)$ can be defined as the pullback via the inverse of $\Zcal$ (see \eqref{eq:pushforward}).   By Proposition~\ref{prop:Zcal_Jacobian_constant_sign}, if $f$ is furthermore $(n,k,m)$-admissible we may choose a sign for $\om k f$ such that the form $\Zcal_* (\om kf \wedge \oml)$ has the same sign as $\omreflkl$ when restricted to $\amcell kf$.

\begin{definition}\label{def:univ_ampl_form}
  Given an $(n,k,m)$-triangulation $\triang=\{f_1,\dots,f_N\}\subset\Aff(-k,\l)$ of degree $1$, the \emph{universal amplituhedron form} $\omtriang{\uamk}$ on $\FlC(\l,k+\l;n)$ is defined by
  \begin{equation}\label{eq:dfn_omtriang_uamk}
\omtriang{\uamk}=\sum_{s=1}^N \Zcal_* (\om k{f_s} \wedge \oml).
  \end{equation}
\end{definition}
Here the signs of the terms in the right hand side are chosen in the same fashion as in~\eqref{eq:ampl_form_dfn}. 
\begin{conjecture}\label{conj:univ_ampl_form}
The form $\omtriang{\uamk}$ does not depend on the choice of a triangulation $\triang$.
\end{conjecture}
It follows from Proposition \ref{prop:univ_form_to_ampl_form} below that Conjecture~\ref{conj:univ_ampl_form} is equivalent to Conjecture~\ref{conj:ampl_form} for all $Z$ such that the triangulations in question have $Z$-degree $1$. Also, Conjectures \ref{conj:ampl_form} and \ref{conj:univ_ampl_form} are known for $k = 1$: a triangulation independent construction of the polytope form is given in \cite{ABL}.  It follows from Theorem \ref{thm:univ_ampl_form} that the two conjectures also hold for $\ell=1$.

We state our main result on the amplituhedron form.

\begin{theorem}\label{thm:univ_ampl_form}
  Suppose that we are given an $(n,k,m)$-triangulation $\triang=\{f_1,\dots,f_N\}\subset\Aff(-k,\l)$ of degree $1$, and let $\triang':=\{f_1^{-1},\dots,f_N^{-1}\}\subset\Aff(-\l,k)$ be the corresponding $(n,\l,m)$-triangulation of degree $1$ (cf. Corollary~\ref{cor:degree_1_triang}). Then the stacked twist map preserves the universal amplituhedron form:
  \begin{equation}\label{eq:univ_ampl_form}
  \tstw^\ast\omtriangprime{\uaml}=\pm\omtriang{\uamk}.
  \end{equation}
\end{theorem}
In \eqref{eq:univ_ampl_form}, the map $\tstw$ can be considered either as a diffeomorphism between the interiors of $\uamk$ and $\uaml$, or as a rational map
$\tstw: \FlC(\l,k+\l,n) \dashedrightarrow \FlC(k,k+\l;n)$.

\begin{remark}
Definition \ref{def:univ_ampl_form} can be extended to forms $\omtriang{\amcelltnn k g}$ for a triangulation $\triang$ of $\amcelltnn k g$.  
It is straightforward to generalize Theorem~\ref{thm:univ_ampl_form} to $\omtriang{\amcelltnn k g}$.  We still expect the analog of Conjecture~\ref{conj:ampl_form} to hold in this case.  Furthermore, assuming that Conjecture~\ref{conj:ampl_form} holds for all $\amcelltnn k g$, the forms are expected to be compatible with subdivisions.  Namely, if $f_1,f_2,\ldots,f_N$ form an $(n,k,m)$-subdivision of $g$, then we should have $\omega_{\amcelltnn k g} = \sum_{i=1}^N \omega_{{\amcelltnn k {f_i}}}$.
\end{remark}

\begin{remark}
We expect that $\am_{n,k,m}$ is a \emph{positive geometry} in the sense of \cite{ABL}, and that the conjectural universal amplituhedron form $\omega_{\uamk}$ we have defined is its \emph{canonical form}.
\end{remark}

\subsection{From the universal amplituhedron form back to the amplituhedron form}
In this section we explain how the amplituhedron form can be recovered from the universal amplituhedron form.

First suppose that $f \in \Aff(-k,\l)$ is $(n,k,m)$-admissible and has degree $1$.  Then we have
$$
\Zcal_*(\om k{f} \wedge \oml) = (\Zcal^{-1})^*(\om k{f} \wedge \oml) = (\Zcal^{-1})^*(\om k{f}) \wedge (\Zcal^{-1})^*(\oml) = \Zcal_*(\om k{f}) \wedge \pi^*(\oml)
$$
where we assume that $\Zcal$ and $\Zcal^{-1}$ have been restricted to the locus where they are isomorphisms.

Suppose that $Z \in \Grtp(k+m,n)$ and $f$ has $Z$-degree $1$.  
Let
\[j_Z: \pi^{-1}(Z^\perp) \hookrightarrow \FlC(\l,k+\l;n)\]
be the inclusion of a fiber of $\pi: \FlC(\l,k+\l;n) \to \GrC(\l,n)$ over $Z^\perp$.  We have an isomorphism $\pi^{-1}(Z^\perp) \cong \GrC(k,k+m)$ and $\Zcal$ restricts to a map $\Zcal: \pcC k f \to \pi^{-1}(Z^\perp)$ that can be identified with $Z:\pcC k f\to \GrC(k,k+m)$.  It follows that 
$$
j_Z^*(\Zcal_*(\om k f)) = Z_* \om k f
$$
as forms on $\pi^{-1}(Z^\perp) \cong \GrC(k,k+m)$.

Now suppose that we are given an $(n,k,m)$-triangulation $\triang=\{f_1,\dots,f_N\}\subset\Aff(-k,\l)$ of degree $1$.  Then
\begin{equation}\label{eq:omtriang_uamk_dfn}
\omtriang{\uamk}=\sum_{s=1}^N \Zcal_* (\om k{f_s} \wedge \oml) 
=\left(\sum_{s=1}^N \Zcal_* (\om k{f_s})\right)\wedge \pi^\ast(\oml)
\end{equation}
and thus when $Z$ belongs to some Zariski open subset of $\Grtp(k+m,n)$, $f_1,f_2,\ldots,f_N$ all have $Z$-degree $1$, and we obtain
$$
j_Z^*\left(\sum_{s=1}^N \Zcal_* (\om k{f_s})\right) =\omtriang{\ampl}
$$
under the isomorphism $\pi^{-1}(Z^\perp) \cong \GrC(k,k+m)$.

Finally, suppose we have the equality $\omtriang{\uamk} = \eta \wedge \oml = \eta' \wedge \oml$ for two rational forms $\eta, \eta'$ of degree $km$ on $\FlC(\l,k+\l;n)$.  Then $(\eta - \eta') \wedge \oml = 0$.  The form $\oml$ contracts to 0  with any \emph{vertical} tangent vector (i.e., a tangent vector of a fiber $\pi^{-1}(Z^\perp)$).  It follows that the top form $j_Z^*(\eta-\eta')$ on $\pi^{-1}(Z^\perp)$ is equal to 0.  
We have thus shown the following.

\begin{proposition}\label{prop:univ_form_to_ampl_form}
Let $Z \in \Grtp(k+m,n)$. Suppose that we have an $(n,k,m)$-triangulation $\triang=\{f_1,\dots,f_N\}\subset\Aff(-k,\l)$ of degree $1$ such that $f_1,f_2,\ldots,f_N$ all have $Z$-degree $1$.
Then there exists a rational form $\eta$ of degree $km$ such that $\omtriang{\uamk} = \eta \wedge \oml$.  Any such form satisfies
\begin{equation}\label{eq:pullbackfiber}
\omtriang{\ampl} = j_Z^*(\eta). 
\end{equation}
\end{proposition}
Thus in order to reconstruct the amplituhedron form $\omtriang{\ampl}$ from the universal amplituhedron form $\omtriang{\uamk}$, one needs to represent $\omtriang{\uamk}$ as a wedge of $\eta$ and $\oml$ for some form $\eta$, and then use~\eqref{eq:pullbackfiber}.
Conversely, if $\omtriang{\ampl}$ exists for generic $Z \in \Grtp(k+m,n)$, it will depend continuously on $Z$, and one can find a form $\eta$ satisfying~\eqref{eq:pullbackfiber} and use it to find $\omtriang{\uamk}=\eta \wedge \oml$. As Proposition~\ref{prop:univ_form_to_ampl_form} and the example in the next section show, with the right choice of coordinates (explained in Section~\ref{sec:coordinates} in full generality), the universal amplituhedron form $\omtriang{\uamk}$ carries essentially the same information and computational complexity as the amplituhedron form $\omtriang{\ampl}$.

\subsection{Example}\label{sec:ampl_form_example}
In this section, we compute the amplituhedron form $\omtriang{\ampl}$ and the universal amplituhedron form $\omtriang{\uamk}$, and explain how the former can be recovered from the latter, in the case $k=\l=1$, $m=2$, $n=4$. In this case, the amplituhedron $\ampl$ is a quadrilateral in $\RP^2$ with vertices being given by the four columns of a totally positive $3\times 4$ matrix $Z$ which we may assume to be given by
\[Z=\begin{pmatrix}
    1 & s & 0 & 0\\
    0 & p & 1 & 0\\
    0 & -q &0 & 1
  \end{pmatrix};\quad W:=Z^\perp=\begin{pmatrix}
-s&1&-p&q
  \end{pmatrix}\]
for some $s,p,q>0$.

There are precisely two $(n,k,m)$-triangulations of the amplituhedron, each of them contains $M(1,1,1)=2$ cells and has degree $1$. We choose the triangulation $\triang=\{f,g\}$ with $f=[2,1,3,4]$ and $g=[1,2,4,3]$ in window notation. The cells $\pc kf$ and $\pc kg$ are parameterized as follows. Each $V\in \pc kf$ can be represented by a matrix $\begin{pmatrix}
1 & 0 & b & c
\end{pmatrix}$ while each $V\in \pc kg$ can be represented by a matrix $\begin{pmatrix}
1 & a & b & 0
\end{pmatrix}$ for some $a,b,c>0$. 

We would like to find $\omtriang{\ampl}$ using~\eqref{eq:ampl_form_dfn}, so let us compute $\omega_{Z(\pc kf)}$ and $\omega_{Z(\pc kg)}$ separately and then add them up. Each generic point in $\Gr(k,k+m)$ is the row span of a matrix $\begin{pmatrix}
1 & x & y
\end{pmatrix}$, so we fix a reference form $\omref_{k,k+m}:=dx\wedge dy$ on $\Gr(k,k+m)$.

The restriction of $Z$ to $\pc kf$ maps $\begin{pmatrix}
1 & 0 & b & c
\end{pmatrix}$ to the row span of $\begin{pmatrix}
1 & b & c
\end{pmatrix}$ in $\Gr(k,k+m)$, and thus we have $x=b$ and $y=c$ on $Z(\pc kf)$. Therefore the map $Z^{-1}$ is given by $b=x$ and $c=y$, and since $\om kf=\frac{db\wedge dc}{b\cdot c}$, the pullback $(Z^{-1})^\ast\om kf$ equals
\[Z_\ast \om kf=\pm\frac{dx\wedge dy}{x\cdot y}.\]
Furthermore, since the functions $x=b$ and $y=c$ are positive on $Z(\pc kf)$, the form $Z_\ast \om kf$ is positive with respect to $\omref_{k,k+m}$ if the sign above is chosen to be $+$.

The computation of $\omega_{Z(\pc kg)}$ is more involved. The restriction of $Z$ to $\pc kg$ maps $\begin{pmatrix}
1 & a & b & 0
\end{pmatrix}$ to the row span of $\begin{pmatrix}
1+sa & pa+b & -qa
\end{pmatrix}\simeq \begin{pmatrix}
1 & \frac{pa+b}{1+sa} & \frac{-qa}{1+sa}
\end{pmatrix}$ in $\Gr(k,k+m)$. Thus we have $x=\frac{pa+b}{1+sa}$ and $y=\frac{-qa}{1+sa}$ on $Z(\pc kg)$. The fact that $g$ has degree $1$ means that $Z$ is birational when restricted to $\pc kg$, equivalently, that we can solve the above equations for $a$ and $b$ and write them as rational functions of $x$ and $y$. Indeed, this can be done, and the map $Z^{-1}:Z(\pc kg)\dashedrightarrow \pc kg$ is given by
\[a=\frac{-y}{sy+q},\quad b=\frac{qx+py}{sy+q}.\]
We can now find the pullback of the form $\om kg=\frac{da\wedge db}{a\cdot b}$ under $Z^{-1}$:
\begin{equation}\label{eq:Z_ast_om_kg_ex}
Z_\ast \om kg=\pm\frac{d \left(\frac{-y}{sy+q}\right)\wedge d \left(\frac{qx+py}{sy+q}\right)}{\frac{-y}{sy+q}\cdot \frac{qx+py}{sy+q}}=\pm\frac{q^2\cdot dx\wedge dy}{(qx+py)(sy+q)y}.
\end{equation}
Here in fact we must choose the $-$ sign, because the functions $qx+py$ and $sy+q$ are positive on $Z(\pc kg)$ while the function $y=\frac{-qa}{1+sa}$ is negative. We are finally able to find the amplituhedron form:
\begin{equation}\label{eq:ex_ampl_form}
\omtriang{\ampl}=Z_\ast \om kf+Z_\ast \om kg=\frac{dx\wedge dy}{x\cdot y}-\frac{q^2\cdot dx\wedge dy}{(qx+py)(sy+q)y}=\frac{(qsx+psy+pq)\cdot dx\wedge dy}{(qx+py)(sy+q)x}.
\end{equation}
We encourage the reader to check that choosing $\triang$ to be the other triangulation of $\ampl$ (consisting of affine permutations $f'=[1,3,2,4]$ and $g'=[0,2,3,5]$) yields the same expression for $\omtriang{\ampl}$.  Agreeing with general expectations (see~\cite{ABL}), the form $\omtriang{\ampl}$ has four poles (along the hypersurfaces $qx+py$, $sy+q$, $x$, and at infinity) which corresponds to the four facets of the quadrilateral $\ampl$.

We now compute the universal amplituhedron form $\omtriang{\uamk}$  on $\Fl(\l,k+\l;n)$. In order for it to be compatible with our previous calculations, we introduce the following coordinates on $\Fl(\l,k+\l;n)$. For each generic pair $(W,U)\in \Fl(\l,k+\l;n)$, there are unique numbers $s,p,q,x,y$ such that $W$ is the row span of $\begin{pmatrix}
-s & 1 & -p & q
\end{pmatrix}$ while $U$ is the row span of  
\begin{equation}\label{eq:coords_Fl_ex}
\begin{pmatrix}
-s& 1& -p & q\\
1 & 0&  x & y
\end{pmatrix}.
\end{equation}
Our goal is to find $\omtriang{\uamk}$ from~\eqref{eq:omtriang_uamk_dfn}. We first compute $\pi^\ast(\oml)$, which is a $(k\l+\l m)$-form on $\Fl(\l,k+\l;n)$. In our coordinates, the map $\pi:\Fl(\l,k+\l;n)\to \Gr(\l,n)$ simply sends the row span of $U$ given by~\eqref{eq:coords_Fl_ex} to the row span of $\begin{pmatrix}
-s& 1& -p & q
\end{pmatrix}$, and thus we have 
\[\pi^\ast(\oml)=\pm\frac{ds\wedge dp\wedge dq}{spq}.\]
We now compute $\Zcal_\ast(\om kf)$. The map $\Zcal$ sends a pair $(V,W)$ to $(W,\Span(V,W))$, and assuming $V$ is the row span of  $\begin{pmatrix}
1 & 0& b & c
\end{pmatrix}$ for some $b,c>0$, we get that our coordinates~\eqref{eq:coords_Fl_ex} for the pair  $(W,\Span(V,W))$ are given by $s=s$, $p=p$, $q=q$, $x=b$, and $y=c$. Thus $\Zcal_\ast(\om kf)=\frac{dx\wedge dy}{x\cdot y}$ as a $km$-form on $\Fl(\l,k+\l;n)$.

Similarly, we find $\Zcal_\ast(\om kg)$. Assuming $V$ is the row span of  $\begin{pmatrix}
1 & a & b & 0
\end{pmatrix}$ for some $a,b>0$, we get that our coordinates~\eqref{eq:coords_Fl_ex} for the pair  $(W,U':=\Span(V,W))$ are given by $s=s$, $p=p$, $q=q$, but in order to find $x$ and $y$, we need to transform the matrix $U':=\begin{pmatrix}
-s& 1& -p & q\\
1 & a & b & 0
\end{pmatrix}$ into the form~\eqref{eq:coords_Fl_ex} by row operations, which yields
\[\begin{pmatrix}
-s& 1& -p & q\\
1 & 0& \frac{b+pa}{1+sa} & \frac{-qa}{1+sa}
\end{pmatrix}.\]
Thus we have $x=\frac{b+pa}{1+sa}$ and $y=\frac{-qa}{1+sa}$, and as we have already computed earlier, solving for $a$ and $b$ yields $a=\frac{-y}{sy+q}$, $b=\frac{qx+py}{sy+q}$. The form $\Zcal_\ast(\om kg)$ now however is given by an expression that is slightly different from~\eqref{eq:Z_ast_om_kg_ex}:
\[\Zcal_\ast (\om kg)=\pm\frac{q^2\cdot dx\wedge dy}{(qx+py)(sy+q)y}+\alpha,\]
where $\alpha$ denotes a sum of terms containing either one of $ds$, $dp$, or $dq$. All those terms vanish after we take the wedge with $\pi^\ast(\oml)$, so we use~\eqref{eq:omtriang_uamk_dfn} to find the universal amplituhedron form:
\begin{equation}\label{eq:ex_uampl_form}
\omtriang{\uamk}=\left(\Zcal_\ast (\om kf)+\Zcal_\ast (\om kg)\right)\wedge \pi^\ast(\oml)=\frac{(qsx+psy+pq)\cdot dx\wedge dy\wedge ds\wedge dp\wedge dq}{(qx+py)(sy+q)x spq}.
\end{equation}
This is a top form on $\Fl(\l,k+\l;n)$, and we can choose the form $\eta$ from Proposition~\ref{prop:univ_form_to_ampl_form} to be
\begin{equation}\label{eq:ex_uampl_eta_form}
\frac{(qsx+psy+pq)\cdot dx\wedge dy}{(qx+py)(sy+q)x}+\beta,
\end{equation}
where $\beta$ is an arbitrary sum of terms containing either one of $ds$, $dp$, or $dq$. Comparing Equations~\eqref{eq:ex_ampl_form}, \eqref{eq:ex_uampl_form}, and~\eqref{eq:ex_uampl_eta_form}, the  similarities between the forms in Proposition~\ref{prop:univ_form_to_ampl_form} become apparent.
We are done with the computation, and will continue the discussion of this example in Section~\ref{sec:ampl_form_poly}.

\section{The interior of the amplituhedron}\label{sec:interior}
In this section, we discuss some topological properties of $\ampl$. In particular, we explain why it is sufficient for our purposes to concentrate on  $\Grtnnm(k,n)$ rather than $\Grtnn(k,n)$.

Recall that $\amplk$ was defined in~\eqref{eq:ampll} to be the image of $\Grtnnm(k,n)$ under $Z$.
Let us also denote by $\amplint$ the interior of the amplituhedron. The following result implies Lemma~\ref{lemma:non-admissible}.
\begin{proposition}\label{prop:open_dense}
We have inclusions
  \begin{equation}\label{eq:interior_inclusions}
Z(\Grtp(k,n))\subset \amplint \subset\amplk\subset\ampl.
  \end{equation}
\end{proposition}
Note that this proposition implies in particular that the amplituhedron is equal to the closure of its interior (see Lemma~\ref{lemma:closure_interior} below), because $\ampl$ is equal to the closure of $Z(\Grtp(k,n))$.

Suppose that $\V\in\Grtnn(k,n)$, $Z\in\Grtp(k+m,n)$, and let $\W:=Z^\perp$. Then by Lemma~\ref{lemma:consecutive}, having $\V\in\Grtnnl(k,n)$ is equivalent to saying that the circular $(k+\l)$-minors of $\VZ:=\stack(\V,\W)$ are all nonzero. For $\V'\in \Fib(\V,Z)$, it follows from~\eqref{eq:VZ'_A} that the matrix $\stack(\V',\Zp)$ has the same maximal minors. Therefore for all $\V'\in\Fib(\V,Z)$ we have $\V'\in\Grtnnl(k,n)$. We have shown the following result.

\begin{proposition}\label{prop:fibers_Grtnnm}
For each point $Y\in\ampl$, exactly one of the following holds:
\begin{itemize}
\item for all $\V\in\Grtnn(k,n)$ such that $\V\cdot Z^\transp=Y$, we have $\V\in\Grtnnl(k,n)$, or
\item for all $\V\in\Grtnn(k,n)$ such that $\V\cdot Z^\transp=Y$, we have $\V\notin\Grtnnl(k,n)$.
\end{itemize}
\end{proposition}
Note also that the first condition holds if and only if $Y\in\amplk$. The next lemma characterizes $\amplk$ as a subset of $\ampl$ satisfying certain boundary inequalities and implies Lemma~\ref{lemma:non-admissible}. Recall that the columns of $Z$ satisfy $Z(j+n)=(-1)^{k-1}Z(j)$.
\begin{lemma}\label{lemma:boundary_ineq}
  Let $Y\in\ampl$ be represented by a $k\times (k+m)$ matrix, and denote by $y_1,\dots,y_{k}\in\R^{k+m}$ the row vectors of $Y$. Then for each $j\in\Z$, the function
  \[\alpha(Y,Z,j):=\det([y_1|y_2|\dots|y_k|Z(j)|Z(j+1)|\dots|Z(j+m-1)])\]
takes a fixed sign on $\ampl$ and  we have $Y\in\amplk$ if and only if $\alpha(Y,Z,j)\neq0$ for all $j\in\Z$.
\end{lemma}
\begin{proof}
  By~\cite[Proposition~20.3]{LamCDM}, $\alpha(Y,Z,j)$ takes a fixed sign on $\ampl$. It remains to show that $Y\in\amplk$ if and only if $\alpha(Y,Z,j)\neq0$. Let $W:=Z^\perp$ and $V\in\Grtnn(k,n)$ be such that $V\cdot Z^\transp=Y$. Define $X:=(\Span(V,W))^\perp$. Thus $X$ is an $m$-dimensional subspace inside the row span of $Z$. By~\cite[Lemma~3.10]{KW}, we have $\Delta_{\{j,j+1,\dots,j+m-1\}}(X)=\alpha(Y,Z,j)$. Let $J\subset [n]$ denote the complement of the set $\{j,j+1,\dots,j+m-1\}$ (viewed modulo $n$) inside $[n]$. By~\eqref{eq:alt_duality}, $\Delta_J(\stack(V,W))=\Delta_{\{j,j+1,\dots,j+m-1\}}(X)$, and thus $\alpha(Y,Z,j)\neq0$ for all $j\in\Z$ if and only if all circular minors of $U:=\stack(V,W)$ are nonzero. This happens precisely when $U\in\topcell$, and the result now follows from Proposition~\ref{prop:fibers_Grtnnm}.
\end{proof}
\begin{proof}[Proof of Proposition~\ref{prop:open_dense}.]
Lemma~\ref{lemma:boundary_ineq} proves the second inclusion $\amplint \subset\amplk$ in~\eqref{eq:interior_inclusions}. The third inclusion $\amplk\subset\ampl$ is obvious, and the first inclusion is the content of Lemma~\ref{lemma:closure_interior} below.
\end{proof}

\begin{lemma}\label{lemma:closure_interior}
  The closure of the interior of $\ampl$ equals $\ampl$, and we have
  \[Z(\Grtp(k,n))\subset \amplint.\]
\end{lemma}
\begin{proof}
  \def\eps{\Ecal}
Let $\V\in\Grtp(k,n)$ and $Y:=\V\cdot Z^\transp$. We would like to show that $Y$ belongs to the interior of $\ampl$. Note that $\GL_{k+m}(\R)$ acts transitively on $\Gr(k,k+m)$, so let $A(\eps)=\Id_{k+m}+\eps\in\GL_{k+m}(\R)$ belong to the neighborhood of the identity for each sufficiently small $(k+m)\times(k+m)$ matrix $\eps$. We need to show that $Y\cdot A(\eps)$ belongs to the image of $\Grtp(k,n)$, equivalently, that there exists $\V'(\eps)\in\Grtp(k,n)$ such that $\V'(\eps)\cdot Z^\transp=Y\cdot A(\eps)$, for all small $\eps$. Let us construct this $\V'(\eps)$ explicitly. Assume without loss of generality that $Z=[\Id_{k+m}\mid S^\transp]$ for an $\l\times (k+m)$ matrix $S$. Then set
  \[\V'(\eps):=\V\cdot \begin{pmatrix}
      \Id_{k+m}+\eps & 0\\
      S\eps & \Id_\l
    \end{pmatrix}.\]
  Clearly for small $\eps$ this matrix is close to $\V$ and thus totally positive. Let us calculate the image of $\V'(\eps)$ under $Z$:
  \[\V'(\eps)\cdot Z^\transp= \V\cdot \begin{pmatrix}
      \Id_{k+m}+\eps\\
      S\eps+S
    \end{pmatrix}=\V\cdot Z^\transp\cdot A(\eps)=Y\cdot A(\eps).\qedhere\]
\end{proof}

\section{The stacked twist map: proofs}\label{sec:stw_properties}

In this section, we examine the stacked twist map from Section~\ref{sec:stw}. We adopt the setting and all notation from that section. Our ultimate goal is to prove Theorem~\ref{thm:stw} and Proposition~\ref{prop:fibers}.

We first discuss the relationship between the stacked twist map and the \emph{twist map} of~\cite{MS}. We remark that in~\cite{MuS}, the twist map of~\cite{MS} was extended to $\U$'s not necessarily belonging to $\topcell$. However, we only focus on $\U\in\topcell$, and mostly use the results of~\cite{MS}.

For two integers $a<b$, we denote $[a,b):=\{a,a+1,\dots,b-1\}\subset\Z$. Let $j\in\Z$, $J:=[j-\l,j+k)$, $\U\in\topcell$, and let $\Ut:=\twistpm(\U)$ be defined by~\eqref{eq:stw}. Denote by $\Ut'$ the result of applying the twist map of~\cite{MS} to $\U$, and let $\Ut''$ be the result of applying the right twist map of~\cite{MuS}. Then we have\footnote{The columns in~\cite{MS,MuS} satisfy $\U(j+n)=\U(j)$ while we impose $\U(j+n)=(-1)^{k-1}\U(j)$. In addition, we have the sign $(-1)^{j}$ in~\eqref{eq:stw} which is not present in~\cite{MS,MuS}. Thus for $k<j\leq n-\l$, the $\pm$ sign in~\eqref{eq:twist_MS} is just $(-1)^{j}$.}
\begin{equation}\label{eq:twist_MS}
\Ut(j)=\frac{\pm1}{\Delta_J(\U)}\Ut'(j+k)=\pm\Ut''(j-\l).
\end{equation}
Here $\Delta_J(\U):=\det(\U(J))$, where $\U(J)$ denotes the submatrix $[\ut(j-\l)|\dots|\ut(j+k-1)]$ of $\U$ with  column set $J$. Thus for example the definition of $\twistpm$ given in~\eqref{eq:stw} can be restated as follows:
\begin{equation}\label{eq:adjoint}
\text{for each $j\in\Z$, $(-1)^{j}\Ut(j)$ is the first row of the matrix $(\U(J))^{-1}$.}
\end{equation}

We now use~\eqref{eq:twist_MS} to apply some results of~\cite{MS} to our setting.

\begin{lemma}\label{lemma:inverse_transpose}
  Let $g\in\GL_{k+\l}(\R)$ be an invertible $(k+\l)\times (k+\l)$ matrix and $\U\in\topcell$. Then
  \[\twistpm(g\cdot \U)=\left(g^{-1}\right)^\transp \twistpm(\U).\]
\end{lemma}  
\begin{proof}
This follows from~\eqref{eq:twist_MS} combined with the proof of~\cite[Lemma~2.3]{MS}.
\end{proof}
\begin{proof}[Proof of Lemma~\ref{lemma:stw:Gr}]
For $g_1\in\GL_k(\R), g_2\in\GL_\l(\R)$, we can take $g:=\begin{pmatrix}
g_1 & 0\\
0 & g_2
\end{pmatrix}$ to be a block matrix, and thus for $\stw(\V,\W)=(\Wt,\Vt)$  we have
  \[\stw\left(g_1\cdot \V,g_2\cdot \W\right)=\left((g_1^{-1})^\transp\cdot \Wt,(g_2^{-1})^\transp\cdot\Vt\right).\qedhere\]
\end{proof}

\begin{lemma}\label{lemma:max_minors_of_U}
  Let $j\in\Z$, $J:=[j-\l,j+k)$, and $K:=[j,j+k+\l)$. Then for every $\U\in\topcell$ and $\Ut:=\twistpm(\U)$, we have
  \begin{equation}\label{eq:max_minors_of_U}
\Delta_K(\Ut)= \frac{(-1)^{j+(j+1)+\dots+(j+k+\l-1)}}{\Delta_J(\U)}.
  \end{equation}
\end{lemma}
\begin{proof}
  This follows from~\cite[Remark~3.7]{MS}, but we deduce this as a simple consequence of Lemma~\ref{lemma:inverse_transpose}. Namely, if the submatrix $\U(J)$ of $\U$ is the identity matrix then the submatrix $\Ut(K)$ is upper triangular with diagonal entries $(-1)^{j},\dots,(-1)^{j+k+\l-1}$, so in this case~\eqref{eq:max_minors_of_U} is obvious. We can write any matrix $\U$ as $\U(J)\cdot \U'$ where $\U'(J)$ is the identity matrix, and thus by Lemma~\ref{lemma:inverse_transpose}, $\Ut=(\U(J)^{-1})^\transp\cdot \twistpm(\U'(J))$, which proves~\eqref{eq:max_minors_of_U} in general.
\end{proof}

For a subspace $\V\in\Gr(k,n)$ and integers $a< b\in\Z$, define
\[\rk(\V;a,b):=\rk([\v(a)|\dots|\v(b-1)]).\]
The positroid cell $\pc k f$ is defined in Definition~\ref{dfn:pos_cell}, however, we will work with a more convenient alternative characterization.
\def\stwp{\stw}
\begin{lemma}\label{lemma:ranks}
  Let $\V\in\Grtnn(k,n)$. Then the unique affine permutation $f\in \Aff(-k,n-k)$ such that $\V\in\pc k f$ is characterized by the property that for each $a< b\in\Z$,
  \begin{equation}\label{eq:ranks}
\rk(\V;a,b)=\#\{i<a\mid a\leq f(i)+k< b\}.
  \end{equation}
\end{lemma}
\begin{proof}
This is well-known, see e.g.~\cite[\S5.2]{KLS}.
\end{proof}

Thus Proposition~\ref{prop:ranks} follows from Lemma~\ref{lemma:ranks} as a simple corollary. We can now describe the relationship between $\pc k f$ and $\pc \l {f^{-1}}$ for $f\in\Aff(-k,\l)$. Note that by Proposition~\ref{prop:ranks}, we have $\rk(\V;a,b)=k$ for $\V\in\Grtnnl(k,n)$ when $b-a\geq k+\l$.
\begin{lemma}
  Let $f\in\Aff(-k,\l),g\in\Aff(-\l,k)$, and consider subspaces $\V\in\pc k f$, $\Vt\in\pc \l {g}$. Then we have $g=f^{-1}$ if and only if for every $a< b\in\Z$ satisfying $b-a<k+\l$, 
  \begin{equation}\label{eq:ranks_inverse}
\rk(\Vt;b-k,a+\l)+b-a=\l+\rk(\V;a,b).
  \end{equation}
\end{lemma}
\begin{proof}
By Lemma~\ref{lemma:ranks},
  \begin{equation*}
    \begin{split}
      \rk(\V;a,b)&=\#\{i<a\mid a\leq f(i)+k< b\},\\
      \rk(\Vt;b-k,a+\l)&=\#\{j<b-k\mid b-k\leq g(j)+\l< a+\l\}.
    \end{split}
  \end{equation*}
  Since $\#\{i<a\mid a\leq f(i)+k\}=k$, we get that $\#\{i<a\mid b\leq f(i)+k\}=k-\rk(\V;a,b)$. Now, using the fact that $i-k\leq f(i)\leq i+\l$, we find that
  \begin{equation}\label{eq:f}
\{i<a\mid b\leq f(i)+k\}=\{b-k-\l\leq i<a\mid b\leq f(i)+k<a+\l+k\}.
  \end{equation}
  On the other hand, since $\#\{j\in\Z\mid b-k\leq g(j)+\l< a+\l\}=\l+k+a-b$, we get that $\#\{b-k\leq j\mid b-k\leq g(j)+\l< a+\l\}=\l+k+a-b-\rk(\Vt;b-k,a+\l)$. Using the fact that $j-\l\leq g(j)\leq j+k$, we find 
  \[\{b-k\leq j\mid b-k\leq g(j)+\l< a+\l\}=\{b-k\leq j<a+\l\mid b-k\leq g(j)+\l< a+\l\}.\]
  We arrive at
  \[\#\{b-k\leq j<a+\l\mid b-k\leq g(j)+\l< a+\l\}=\l+k+a-b-\rk(\Vt;b-k,a+\l).\]
  Rearranging the terms in the right hand side of~\eqref{eq:f}, we get
  \[\#\{b-k\leq i+\l<a+\l\mid b-k\leq f(i)<a+\l\}=k-\rk(\V;a,b).\]
  Thus if $f=g^{-1}$ then $i=g(j)\mapsto j=f(i)$ is a bijection between the sets in the left hand sides of the two equations above, so we get $\l+k+a-b-\rk(\Vt;b-k,a+\l)=k-\rk(\V;a,b)$ which is equivalent to~\eqref{eq:ranks_inverse}. This proves one direction of the lemma.

  Conversely, if ~\eqref{eq:ranks_inverse} holds, then $f$ and $g^{-1}$ must define two positroid cells such that the ranks of columns $a,a+1,\dots, b-1$ are the same for all $a<b$.  By Lemma~\ref{lemma:ranks}, the two cells must therefore be the same and we conclude that $f=g^{-1}$.
\end{proof}

The next lemma proves part~\eqref{item:stw:well_defined} of Theorem~\ref{thm:stw}.
\begin{lemma}\label{lemma:consecutive}
Let $\V\in\Grtnn(k,n)$, $Z\in\Grtp(k+m,n)$, $\Zp:=Z^\perp$, and $\VZ:=\stack(\V,\Zp)$. Let $f\in\Aff(-k,n-k)$ be such that $\V\in\pc k f$. Then $f\in\Aff(-k,\l)$ if and only if $\VZ\in\topcell$. 
\end{lemma}
\begin{proof}
 This is essentially~\cite[Proposition~20.3]{LamCDM}, but we include the proof for completeness.  Let $J:=[j-\l,j+k)$. The minor $\Delta_J(\VZ)$ can be written as a sum
  \[\Delta_J(\VZ)=\sum_{I\in{J\choose k}} (-1)^{\inv(I,J\setminus I)}\Delta_I(\V)\Delta_{J\setminus I}(\Zp),\]
  where $\inv(I,J\setminus I)=\#\{i\in I,i'\in J\setminus I: i'<i\}$. Moreover, $\Delta_I(\V)$ is nonnegative and the sign of $\Delta_{J\setminus I}(\Zp)$ depends only on the parity of $\sum_{i'\in J\setminus I} i'$, because $\alt(\Zp)\in\Grtp(\l,n)$ by~\eqref{eq:alt_duality}. It follows that all the terms in the above sum have the same sign, because the terms corresponding to $I$ and $I\setminus\{i\}\cup\{i+1\}$ have the same sign for each $i\in I$ such that $i+1\in J\setminus I$. Moreover, $\Delta_{J\setminus I}(\Zp)$ is always nonzero, and thus $\Delta_J(\VZ)$ is nonzero if and only if $\Delta_I(\V)$ is nonzero for some $I\subset J$. This is exactly equivalent to saying that the rank of $[\v(j-\l)|\v(j-\l+1)|\dots|\v(j+k-1)]$ equals to $k$, so the result follows by Proposition~\ref{prop:ranks}.
\end{proof}

Let $\topcellC$ denote the set of complex $(k+\l)\times n$ matrices with nonzero circular minors. We give a simple result relating the fibers $\Fib(V,Z)$ with their images under $\stw$.
\begin{lemma}\label{lemma:A^T}
  Let $\V\in\GrC(k,n)$, $\W\in\GrC(\l,n)$ be such that $\stack(\V,\W)\in\topcellC$. Denote $Z:=\W^\perp$,  and let $\V'\in\GrC(k,n)$ be such that $V'\cdot Z^\transp=V\cdot Z^\transp$. Let $(\Zpt,\Vt):=\stwp(\V,\Zp)$ and $(\Zpt',\Vt'):=\stwp(\V',\Zp)$. Then we have $\Zpt'=\Zpt$ and there exists a $k\times \l$ matrix $A$ such that:
  \begin{itemize}
  \item $\V'=\V+A\cdot \Zp$, and
  \item $\Vt'=\Vt-A^\transp\cdot \Zpt$.
  \end{itemize}
\end{lemma}
\begin{proof}
  It is clear that $\V\cdot Z^\transp=\V'\cdot Z^\transp$ is equivalent to saying that $\V'=\V+A\cdot \Zp$ for some $k\times \l$ matrix $A$. But then $\VZ':=\stack(\V',\Zp)$ is obtained from $\VZ:=\stack(\V,\Zp)$ by
  \begin{equation}\label{eq:VZ'_A}
  \VZ'=
    \begin{pmatrix}
      \Id_k & A\\
      0 & \Id_\l
    \end{pmatrix}\cdot \VZ.
  \end{equation}
  Here $\Id_k$ denotes the $k\times k$ identity matrix. By Lemma~\ref{lemma:inverse_transpose}, the matrix $\VZt'=\stack(\Zpt,\Vt')$ is obtained from $\VZt$ by
  \[\VZt'=
    \begin{pmatrix}
      \Id_k & 0\\
      -A^\transp & \Id_\l
    \end{pmatrix}\cdot \VZt,\]
  which finishes the proof.
\end{proof}

\begin{lemma}\label{lemma:ranks_after_twist}
  Let $\V\in\Grtnnl(k,n)$, $Z\in\Grtp(k+m,n)$, $\Zp:=Z^\perp$, and $(\Wt,\Vt):=\stw(\V,\Zp)$. Then for every $a<b\in\Z$, we have the following:
  \begin{enumerate}[\normalfont (a)]
  \item\label{item:ranks_Zp} if $b-a=k$ then the vectors $\zpt(a),\dots,\zpt(b-1)$ form a basis of $\R^k$;
  \item\label{item:ranks_V_big} if $b-a\geq k+\l$ then $\rk(\V;a,b)=k$ and $\rk(\Vt;a,b)=\l$;
  \item\label{item:ranks_V_small} if $b-a<k+\l$ then $\rk(\Vt;b-k,a+\l)+b-a=\l+\rk(\V;a,b).$
  \end{enumerate}
\end{lemma}
\begin{proof}
  \def\Zptd{Q}
  \def\zptd{q}
  \def\zptv{q}
  \def\zero{\mathbf{0}}
  Part~\eqref{item:ranks_V_big} follows from Lemmas~\ref{lemma:max_minors_of_U} and~\ref{lemma:consecutive}. We now prove part~\eqref{item:ranks_Zp}, so let $b:=a+k$. Denote $r:=\rk(\Zpt;a,b)$, and consider the $(k-r)$-dimensional space
  \[\Zptd:=\{\zptv=(\zptd_a,\dots,\zptd_{b-1})\in\R^k\mid \zptd_a\zpt(a)+\dots\zptd_{b-1}\zpt(b-1)=0\}.\]
  Recall that the matrix $\VZt$ has columns $\vzt(j)=(\zpt(j),\vt(j))$, and by~\eqref{eq:max_minors_of_U}, the column vectors $\vzt(a),\dots,\vzt(b-1)$ are linearly independent. Therefore we get an injective linear map $\lambda:\Zptd\to \R^{\l}$ defined by $\lambda(\zptv)=\zptd_a\vt(a)+\dots+\zptd_{b-1}\vt(b-1)$, so that $\zptd_a \vzt(a)+\dots+\zptd_{b-1}\vzt(b-1)=(\zero_k,\lambda(\zptv))\in\R^{k+\l}$. Here $\zero_k\in\R^k$ denotes the zero vector.

  Recall that we have $a+k=b$. By~\eqref{eq:stw}, each of the columns $\vzt(a),\dots,\vzt(b-1)=\vzt(a+k-1)$ is orthogonal to each of the columns $\vz(b-\l)=\vz(a+k-\l),\dots,\vz(a+k-1)$, where $\vz(j)=(\v(j),\zp(j))$. Thus $\lambda(\zptv)$ is orthogonal to $\zp(a+k-\l),\dots,\zp(a+k-1)$. Since these vectors are the columns of $\Zp$, they must form a basis of $\R^\l$, which implies that $\Zptd$ must be a zero-dimensional space, and thus $k-r=0$. This finishes the proof of~\eqref{item:ranks_Zp}.

  The proof of~\eqref{item:ranks_V_small} will be completely similar. Denote $r:=\rk(\V;a,b)$, and consider the $(b-a-r)$-dimensional space
  \[\Zptd:=\{\zptv=(\zptd_a,\dots,\zptd_{b-1})\in\R^{b-a}\mid \zptd_a\v(a)+\dots\zptd_{b-1}\v(b-1)=0\}.\]
  By Lemma~\ref{lemma:consecutive}, we get an injective map $\lambda:\Zptd\to \R^\l$ defined by $\lambda(\zptv)=\zptd_a\zp(a)+\dots+\zptd_{b-1}\zp(b-1)$, so that $\zptd_a\vz(a)+\dots+\zptd_{b-1}\vz(b-1)=(\zero_k,\lambda(\zptv))$. By~\eqref{eq:stw}, each of the vectors $\vz(a),\dots,\vz(b-1)$ is orthogonal to each of the vectors $\vzt(b-k),\dots,\vzt(a+\l-1)$, so $\lambda(\zptv)$ is orthogonal to $\vt(b-k),\dots,\vt(a+\l-1)$. This implies that $\rk(\Vt;b-k,a+\l)\leq\l-\dim(\Zptd)=\l-b+a+\rk(\V;a,b)$. This inequality is in fact an equality, since for any vector $h\in \R^\l$ that is orthogonal to $\vt(b-k),\dots,\vt(a+\l-1)$, the vector $(\zero_k,h)$ is orthogonal to $\vzt(b-k),\dots,\vzt(a+\l-1)$, The latter vectors are linearly independent and span the orthogonal complement of the span of $\vz(a),\dots,\vz(b-1)$ in $\R^{k+\l}$. Thus there are coefficients $(\zptd_a,\dots,\zptd_{b-1})\in\R^{b-a}$ such that $(\zero_k,h)=\zptd_a\vz(a)+\dots+\zptd_{b-1}\vz(b-1)$. It follows that $h$ belongs to the image of $\lambda$, which finishes the proof of~\eqref{item:ranks_V_small}.
\end{proof}
An important special case of part~\eqref{item:ranks_V_small} of Lemma~\ref{lemma:ranks_after_twist} is $b=a+k$, where we get $\rk(\Vt;a,a+\l)+k=\l+\rk(\V;a,a+k)$. Thus we have $\rk(\V;a,a+k)=k$ if and only if $\rk(\Vt;a,a+\l)=\l$.

\begin{proof}[Proof of Theorem~\ref{thm:stw}]
The only ingredient missing to show Theorem~\ref{thm:stw} is to prove that $\Vt\in\Grtnn(\l,n)$ and $\Wt\in\Grtperp{\l+m}n$. This would imply part~\eqref{item:stw:positivity} of the theorem, and then part~\eqref{item:stw:inverse} will follow from Lemma~\ref{lemma:ranks_after_twist} combined with~\eqref{eq:ranks_inverse}. Note that as long as $\V\in\Grtp(k,n)$, Lemma~\ref{lemma:ranks_after_twist} implies that $\rk(\Vt;a,a+\l)=\l$ for all $a\in\Z$. It suffices to find one pair $(\V_0,\W_0)\in\Grtp(k,n)\times\Grtperp{k+m}n$ for which the image $\stw(\V_0,\W_0)=(\Wt_0,\Vt_0)$ belongs to $\Grtperp{\l+m}n\times\Grtp(\l,n)$. Indeed, then for any other pair $(\V,\W)\in\Grtnnl(k,n)\times\Grtperp{k+m}n$, one can construct a path $(\V_t,\W_t)$, $0\leq t\leq 1$ such that $(\V_0,\W_0)$ is as above, $(\V_1,\W_1)=(\V,\W)$, and for $0\leq t<1$, $(\V_t,\W_t)$ belongs to $\Grtp(k,n)\times\Grtperp{k+m}n$. In this case, the image $(\Wt_t,\Vt_t)=\stw(\V_t,\W_t)$ will belong to $\Grtperp{\l+m}n\times\Grtp(\l,n)$, because in order to exit this space, one of the circular minors of either $\Vt_t$ or $\Wt_t$ must vanish. It follows by continuity that $(\Wt,\Vt)\in\Grtnn(\l+m,n)\times\Grtnn(\l,n)$, and by Lemma~\ref{lemma:ranks_after_twist}, $(\Wt,\Vt)$ actually must belong to $\Grtp(\l+m,n)\times\Grtnnk(\l,n)$.

In order to construct this pair $(\V_0,\W_0)$, we make use of the \emph{cyclically symmetric amplituhedron} introduced in~\cite[Section~5]{GKL}. In fact, the computation we need is almost literally the same as the one outlined in~\cite{Scott}, see~\cite{karp_cyclic_shift} or~\cite[Lemma~3.1]{GKL} for details.

Consider the set of $n$-th roots of $(-1)^{k-1}$, and let $z_1,\dots,z_k$ be those $k$ of them with the largest real part and let $z_{k+1},\dots,z_{k+\l}$ be those $\l$ of them with the smallest real part. Let $\V_0$ be (the real $k$-dimensional subspace of) the linear $\C$-span of the vectors $(1,z_i,\dots,z_i^{n-1})$, for $1\leq i\leq k$, and let $\Zp_0$ be (the real $k$-dimensional subspace of) the linear $\C$-span of the vectors $(1,z_i,\dots,z_i^{n-1})$, for $k+1\leq i\leq k+\l$. Then we have $\V_0\in \Grtp(k,n)$ and $\Zp_0\in\Grtperp{k+m}n$, see~\cite[Lemma~3.1]{GKL}.

\begin{proposition}\label{prop:cs}
Let $\stwp(\V_0,\W_0)=(\Wt_0,\Vt_0)$. Then $\Vt_0\in\Grtp(\l,n)$ and $\Wt_0\in\Grtperp{\l+m}n$. Moreover, we have $\Vt_0=\alt(\W_0)$ and $\Wt_0=\alt(\V_0)$.
\end{proposition}
\begin{proof}
  Recall from Section~\ref{sec:stw} that the columns of $\U:=\stack(\V_0,\W_0)$ satisfy $\u(j+n)=(-1)^{k-1}\u(j)$ for all $j\in \Z$. We see that $\U$ is just the Vandermonde matrix with entries $(z_i^{j-1})_{1\leq i\leq k+\l,1\leq j\leq n}$. 
  Thus by~\eqref{eq:adjoint}, the vector $\vzt(j)$ has coordinates given by ratios of Vandermonde determinants. More precisely, its $i$-th coordinate equals
    \begin{equation*}
      \begin{split}
        \Ut_{ij}&=\dfrac{\displaystyle(-1)^{i+j-1}\det(z_p^{s-1})_{\substack{1\leq p\leq k+\l:\,p\neq i;\\ j-\l< s<j+k}}}{\displaystyle\det(z_p^{s-1})_{\substack{1\leq p\leq k+\l;\\ j-\l\leq s<j+k}}} = (-1)^{j}z_i^{-(j-\l-1)} \cdot h(i),
      \end{split}
    \end{equation*}
    where
    \[h(i):=(-1)^{i-1}\dfrac{(z_1\cdots z_{k+\l})^{j-\l-1}\displaystyle\prod_{\substack{1\leq s<t\leq k+\l:\\ s,t\neq i}}(z_t-z_s)}{\displaystyle(z_1\cdots z_{k+\l})^{j-\l}\prod_{1\leq s<t\leq k+\l}(z_t-z_s)}.\]
    After dividing row $i$ of $\VZt$ by $h(i)$ for each $i$ (which does not affect total positivity of either $\Vt_0$ or $\Zpt_0$), the $(i,j)$-th entry of $\VZt$ becomes $(-1)^{j}z_i^{-(j-\l-1)}$. We can now divide row $i$ of $\VZt$ by $z_i^{\l+1}$, and after that the $i$-th entry becomes just $(-1)^jz_i^{-j}$. This matrix is obtained from $\alt(\VZ)$ by switching the rows that correspond to pairs of inverse $z_i$'s. It is clear that each of the sets $\{z_1,\dots,z_k\}$ and $\{z_{k+1},\dots,z_{k+\l}\}$ is closed under taking the inverse, thus  these row operations belong to $\GL_k(\C)\times\GL_\l(\C)\subset\GL_{k+\l}(\C)$. Therefore we indeed get that the span $\Vt_0$ of the last $\l$ rows of $\VZt$ is totally positive and equal to $\alt(\Zp_0)$, while the span  $\Zpt_0$ of the first $k$ rows of $\VZt$ is equal to $\alt(\V_0)$ and thus belongs to $\Grtperp{\l+m}n$ by~\eqref{eq:alt_duality}. We are done with the proof.
  \end{proof}

  As we have explained earlier, Proposition~\ref{prop:cs} finishes the proof of Theorem~\ref{thm:stw}.
\end{proof}

It remains to note that Proposition~\ref{prop:fibers} now follows from Lemma~\ref{lemma:A^T}.

\section{The canonical form: proofs}\label{sec:can_forms_3}

In this section, we prove Theorem~\ref{thm:top_form_stw}, and then deduce its generalization, Theorem~\ref{thm:lower_form_stw}, as a corollary. We start by recalling some notions related to the combinatorics of \emph{clusters} of Pl\"ucker coordinates. 

\subsection{Clusters}\label{sec:clusters}
  Let $f\in\Aff(-k,n-k)$ be an affine permutation. The \emph{Grassmann necklace} associated with $f$ is a sequence $\grneck k f=(I_1,I_2,\dots,I_n)$ of $k$-element subsets of $[n]$ defined by
  \[I_i=\{f(j)+k\mid j<i\text{ and } f(j)+k\geq i \}\]
  for $i\in[n]$ with the indices taken modulo $n$.

  For each $i\in[n]$, define a total order $\preceq_i$ on $[n]$ by
\[i\prec_i i+1\prec_i\dots\prec_i i+n-1,\]
where the indices again are taken modulo $n$. For two sets $S,T\in{[n]\choose k}$, we write $S\preceq_i T$ if $S=\{s_1\prec_i\dots\prec_i s_k\}$, $T=\{t_1\prec_i\dots\prec_i t_k\}$ with $s_j\preceq_i t_j$ for $j\in[k]$. The \emph{positroid} $\positroid k f\subset {[n]\choose k}$ of $f$ is a collection of subsets of $[n]$ given by
\[\positroid k f:=\left\{J\in{[n]\choose k}\mid J\preceq_i I_i \text{ for all $i\in[n]$} \right\},\]
where $\grneck k f=(I_1,\dots,I_n)$ is the Grassmann necklace associated with $f$.
\begin{definition}\label{dfn:pos_cell}
  
  The sets $\pc k f\subset \pctnn k f \subset \Gr(k,n)$ and $\ce k f\subset \GrC(k,n)$, are defined by
  \begin{equation*}
    \begin{split}
\pc k f&:=\{X\in\Grtnn(k,n)\mid \text{$\Delta_J(X)>0$ for $J\in\positroid k f$ and $\Delta_J(X)=0$ for $J\notin\positroid k f$}\}.\\
\pctnn k f&:=\{X\in\Grtnn(k,n)\mid \text{$\Delta_J(X)\geq0$ for $J\in\positroid k f$ and $\Delta_J(X)=0$ for $J\notin\positroid k f$}\}.\\
\ce k f&:=\{X\in\GrC(k,n)\mid \text{$\Delta_J(X)=0$ for $J\notin\positroid k f$}\}.  \\
    \end{split}
  \end{equation*}
\end{definition}

The variety $\ce k f$ can alternatively be defined either as the Zariski closure of $\pc k f$ in $\GrC(k,n)$ or as the set of all $\V\in\GrC(k,n)$ satisfying
\[\rk(\V;a,b)\leq\#\{i<a\mid a\leq f(i)+k< b\}\]
for all integers $a<b$.  See \cite{KLS}.

For instance, when $f=\id 0:\Z\to\Z$ is the identity map, we have $I_i=\{i,i+1,\dots,i+k-1\}$ (modulo $n$) for all $i\in[n]$, $\positroid k f={[n]\choose k}$, $\pc k f=\Grtp(k,n)$, $\pctnn kf=\Grtnn(k,n)$, and $\ce k f = \GrC(k,n)$.

By Definition~\ref{dfn:pos_cell}, an element $X\in\pc k f$ can be described by a collection $\{\Delta_J(X)\mid J\in\positroid k f\}\in\RP_{>0}^{\positroid k f}$ of positive real numbers, however, this data is highly redundant.

\begin{definition}[\cite{LZ}]
We say that two sets $S,T\in{[n]\choose k}$ are \emph{weakly separated} if there do not exist $1\leq a<b<c<d\leq n$ such that $a,c\in S\setminus T$ and $b,d\in T\setminus S$, or vice versa.
\end{definition}

We say that a collection $\Ccal\subset{[n]\choose k}$ is \emph{weakly separated} if any two of its elements are weakly separated. For example, it is not hard to check that $\grneck k f$ is a weakly separated collection for any $f\in\Aff(-k,n-k)$, see~\cite[Lemma~4.5]{OPS}.

\begin{definition}\label{dfn:clusters}
  Given an affine permutation $f\in\Aff(-k,n-k)$, a weakly separated collection $\Ccal\subset\positroid{k}{f}$ is called a \emph{cluster} if it contains $\grneck k f$ and is not contained inside any other weakly separated collection $\Ccal'\subset\positroid{k}{f}$.
\end{definition}

Recall from Proposition~\ref{prop:OPS} that if $\Ccal$ is a cluster for $f$ then the size of $\Ccal$ is given by~\eqref{eq:Ccal_size}, and the map $X\mapsto \{\Delta_J(X)\mid J\in\Ccal\}$, is a diffeomorphism between $\pc k f$ and the positive part $\RP_{>0}^{\Ccal}$ of the $(k(n-k)-\inv(f))$-dimensional projective space.

\subsection{Circular minors}

We state an identity from~\cite{MS} relating certain minors of $\Ut$ to those of $\U$.
\begin{lemma}\label{lemma:I_J}
  Fix $k,\l,m,n$, and let $\U\in\topcell$ and $\Ut:=\twistpm(\U)$. Let $p,q\geq0$ be integers such that $p+q=k+\l$, and choose some $a,b\in[n]$ so that the following two sets (viewed modulo $n$) have size $k+\l$:
  \begin{equation}\label{eq:I_J}
J:=[a,a+p)\cup[b,b+q),\quad I:=[a+k-q,a+k)\cup [b+k-p,b+k).
  \end{equation}
  Then
   \begin{equation}\label{eq:Delta_I_J}
\Delta_J(\Ut)=\pm \frac{\Delta_I(\U)}{\Delta_{[a-\l,a+k)}(\U)\cdot\Delta_{[b-\l,b+k)}(\U)}.
  \end{equation}
\end{lemma}
\begin{proof}
Follows from~\eqref{eq:twist_MS} together with~\cite[Proposition~3.5]{MS}.
\end{proof}

Using Lemma~\ref{lemma:I_J}, we give a formula for circular minors of $\Wt$ and $\Vt$.

\begin{lemma}\label{lemma:consecutive_V_W}
  Let $(\Wt,\Vt):=\stw(\V,\W)$ and $\U:=\stack(\V,\W)$. Then for $j\in\Z$, we have
  \begin{align}
\label{eq:consecutive_W} \Delta_{[j-k,j)}(\Wt)&=\pm\frac{\Delta_{[j-\l,j)}(\W)}{\Delta_{[j-k-\l,j)}(\U)},\\
\label{eq:consecutive_V} \Delta_{[j,j+\l)}(\Vt)&=\pm\frac{\Delta_{[j,j+k)}(\V)}{\Delta_{[j-\l,j+k)}(\U)}.
  \end{align}
\end{lemma}
\begin{proof}
  The proof is similar to that of Lemma~\ref{lemma:max_minors_of_U}. Let us denote $K:=[k+\l]$,  $\Ut:=\twistpm(\U)=\stack(\Wt,\Vt)$, $\U':=(\U(K))^{-1}\cdot \U$, and $\Ut':=\twistpm(\U')$ which by Lemma~\ref{lemma:inverse_transpose} equals $(\U(K))^\transp\cdot \Ut$. Let $\V',\W',\Vt',\Wt'$ be such that $\U'=\stack(\V',\W')$ and $\Ut'=\stack(\Wt',\Vt')$. Note that the submatrix of $\Ut'$ with column set $[\l+1,2\l+k+1)$ is upper triangular with $\pm1$'s on the diagonal. Applying Lemma~\ref{lemma:I_J} to $\U'$ and $\Ut'$ with $a=\l-k+1$, $b=j$, $p=k$, $q=\l$ yields
  \[\Delta_{[j,j+\l)}(\Vt')=\pm \frac{\Delta_{[j,j+k)}(\V')}{1\cdot \Delta_{[j-\l,j+k)}(\U')}.\]
Let $r:=\det(\U(K))$. Applying the same lemma to $\U$ and $\Ut$, we get
\[r^{-1}\cdot \Delta_{[j,j+\l)}(\Vt)=\pm \frac{r\cdot \Delta_{[j,j+k)}(\V)}{r\cdot r\cdot\Delta_{[j-\l,j+k)}(\U)}.\]
This proves~\eqref{eq:consecutive_V}. The proof of~\eqref{eq:consecutive_W} is completely similar, except that we need to set $\U':=w_0\cdot (\U(K))^{-1}\cdot \U$, where $w_0$ is the $(k+\l)\times(k+\l)$ permutation matrix having $(i,j)$-th entry equal to $1$ if $i+j=k+\l+1$ and to $0$ otherwise.
\end{proof}

\begin{example}
  Let $k=2$, $\l=1$, $n=5$, and $\U$ and $\Ut\cdot\Deltas$ be given in Figure~\ref{fig:stw}, with the matrix $\Deltas=\diag(2,4,8,10,4)$ given in~\eqref{eq:ex_D}. Furthermore, let $p=1$, $q=2$, $a=1$, $b=3$, so~\eqref{eq:I_J} yields $J=\{1,3,4\}$ and $I=\{1,2,4\}$. By~\eqref{eq:Delta_I_J}, we have
  \begin{equation}\label{eq:ex_Delta_I_J}
\Delta_{134}(\Ut)=\pm\frac{\Delta_{124}(\U)}{\Delta_{125}(\U)\cdot\Delta_{234}(\U)}.
  \end{equation}
  We note that $\Delta_{134}(\Ut\cdot\Deltas)=20$, so the left hand side of~\eqref{eq:ex_Delta_I_J} is $\Delta_{134}(\Ut)=\frac{20}{2\times 8\times 10}=\frac18$. On the other hand, we have
  \[\Delta_{124}(\U)=2,\quad \Delta_{125}(\U)=2,\quad \Delta_{234}(\U)=-8,\]
  so the right hand side of~\eqref{eq:ex_Delta_I_J} is $\pm\frac18$, in agreement with Lemma~\ref{lemma:I_J}.

  Let us now set $j:=3$, and verify~\eqref{eq:consecutive_W} and \eqref{eq:consecutive_V}. We have
  \[\Delta_{[1,3)}(\Wt\cdot \Deltas)=\det \begin{pmatrix}
      -1 & 4\\
      1 & 0
    \end{pmatrix}=-4,\]
  thus the left hand side of~\eqref{eq:consecutive_W} is $\frac{-4}{2\times 4}=-\frac12$. The right hand side of~\eqref{eq:consecutive_W} is given by $\pm\frac{\Delta_{2}(\W)}{\Delta_{125}(\U)}=\pm\frac{1}{2}$.

  Similarly, the left hand side of~\eqref{eq:consecutive_V} is $\Delta_{3}(\Vt)=\frac28=\frac14$, because $\Delta_3(\Vt\cdot \Deltas)=2$. The right hand side of~\eqref{eq:consecutive_V} is $\pm\frac{\Delta_{34}(\V)}{\Delta_{234}(\U)}=\pm \frac28=\pm\frac14$.
\end{example}

\subsection{Coordinates on $\Gr(k,n)\times \Gr(\l,n)$}\label{sec:coordinates}

\def\N{N}
\def\M{M}
 For a $k\times \l$ matrix $T$ and an $\l\times k$ matrix $S$, define
 \[\M(T,S):=
\begin{pmatrix}
  \Id_k & T\\
  S & \Id_\l
\end{pmatrix}.\]

\begin{remark}
  By the well known formula for the determinant of a block matrix, we have
  \begin{equation}\label{eq:det_block}
\det(\M(T,S))=\det(\Id_k-TS)=\det(\Id_\l-ST).
  \end{equation}
\end{remark}

For $\U,\U'\in\topcell$, we write $\U\simkl \U'$ if $\U=g\cdot \U'$ for some $g\in\GL_k(\R)\times\GL_\l(\R)$.
\begin{lemma}\label{lemma:split}
For $\M(T,S)\in\GL_{k+\l}(\R)$ and $\U\in\topcell$, we have
  \[\twistpm(\M(T,S)\cdot\U)\simkl \M(-S^\transp,-T^\transp)\cdot\twistpm(\U).\]
\end{lemma}
\begin{proof}
  First, note that the matrix $\M(-T,-S)$ must also be invertible because it has the same determinant as $\M(T,S)$ by~\eqref{eq:det_block}. Multiplying them together, we get
  \[\begin{pmatrix}
  \Id_k & -T\\
  -S & \Id_\l
\end{pmatrix}\cdot\begin{pmatrix}
  \Id_k & T\\
  S & \Id_\l
\end{pmatrix}=\begin{pmatrix}
  \Id_k-TS & 0\\
  0 & \Id_\l-ST
\end{pmatrix}.\]
In particular, the matrix on the right hand side (which we denote by $g$) must be invertible: $g\in\GL_k(\R)\times\GL_\l(\R)$. Thus $g^{-1}\cdot\M(-T,-S)=\M(T,S)^{-1}$, i.e., $\M(-T,-S)\simkl\M(T,S)^{-1}$. Lemma~\ref{lemma:split} now follows from Lemma~\ref{lemma:inverse_transpose}.
\end{proof}

We introduce a convenient coordinate chart on $\Gr(k,n)\times\Gr(\l,n)$. Let $T,S,A,B$ be $k\times \l, \l\times k, k\times m,$ and $\l\times m$ matrices respectively. Denote
\[\N(T,S,A,B):=\begin{pmatrix}
    \Id_k & T & A\\
    S & \Id_\l & B
  \end{pmatrix}.\]

Take a generic pair $(\V,\W)\in\Gr(k,n)\times\Gr(\l,n)$, and let $\Span(\V,\W) \in\Gr(k+\l,n)$ be the row span of $\U:=\stack(\V,\W)$. Since the pair $(\V,\W)$ is generic, there exists a unique pair $T,S$ such that $\U([k+\l])\simkl\M(T,S)$, and thus we can write
\begin{equation}\label{eq:STAB}
\U\simkl N(T,S,A+TB,B+SA)=M(T,S)\cdot N(0,0,A,B).
\end{equation}
Here $\V\in\Gr(k,n)$ is the row span of $[\Id_k|T|A+TB]$, and $W\in\Gr(\l,n)$ is the row span of $[S|\Id_\l|B+SA]$.

\def\Tt{{\tilde{T}}}
\def\St{{\tilde{S}}}
\def\ttt{{\tilde{t}}}
\def\stt{{\tilde{s}}}
\def\At{{\tilde{A}}}
\def\Bt{{\tilde{B}}}

\def\DT{\D k \l T}
\def\DS{\D\l k S}
\def\DA{\D k m A}
\def\DB{\D\l m B}
\def\DTt{\D k \l \Tt}
\def\DSt{\D\l k \St}
\def\DAt{\D k m \At}
\def\DBt{\D\l m \Bt}
\def\DATB{\D k m {A+TB}}
\def\DBSA{\D\l m {B+SA}}
Recall from Definition~\ref{dfn:canonical_form_top_cell} that the canonical forms on $\Gr(k,n)$, $\Gr(\l,n)$, and $\Gr(k+\l,n)$ are given in these coordinates by
\[\omega_{k,n}(\V)=\pm\frac{\DT\wedge \DATB}{\pmin k V},\quad \omega_{\l,n}(\W)=\pm\frac{\DS\wedge \DBSA}{\pmin \l W},\]
\[\omega_{k+\l,n}(\U):=\pm\frac{\DA\wedge \DB}{\pmin{k+\l}{N(0,0,A,B)}}.\]

\begin{lemma}\label{lemma:omega_k_wedge_omega_l}
In the above notation, we have
  \[\omega_{k,n}\wedge \omega_{\l,n}=\pm\frac{ (\det(\M(T,S)))^m\cdot\DT\wedge \DS\wedge\DA\wedge\DB }{\pmin k V \cdot \pmin \l W}.\]
\end{lemma}
\begin{proof}
  We need to show that after taking the wedge with $\DT\wedge\DS$ (i.e., treating the entries of $T$ and $S$ as constants), we have
  \[\DATB\wedge\DBSA=\pm(\det(\M(T,S)))^m\cdot\DA\wedge\DB.\]
Consider a linear operator on the $(k+\l)$-dimensional vector space $\R^{k+\l}$ given by the matrix $\M(T,S)$. For vectors $v\in\R^k$ and $w\in\R^\l$, this operator sends $(v,w)\in\R^{k+\l}$ to $(v+Tw,w+Sv)$. Therefore it sends the form $\omega(v,w):=dv_1\wedge \dots\wedge dv_k\wedge dw_1\wedge\dots\wedge dw_\l$ to $\det(\M(T,S))\cdot \omega(v,w)$. Letting $v:=A(j)$ and $w:=B(j)$ for $j\in[n]$, we get that $(A+TB)(j)=v+Tw$ and $(B+SA)(j)=w+Sv$, and thus
  \[\D k 1{(A+TB)(j)}\wedge\D \l 1{(B+SA)(j)}=\pm\det(\M(T,S))\cdot \D k 1{A(j)}\wedge\D \l 1{B(j)}.\]
  The result follows by applying this to $j=1,2,\dots,m$.
\end{proof}

  Lemma~\ref{lemma:omega_k_wedge_omega_l} implies that $\omega_{k,n}\wedge\omega_{\l,n}=\pm \omkl\wedge \omega_\Ker$, where
  \[\omega_\Ker:=\pm\frac{(\det(M(T,S)))^m \cdot \pmin{k+\l}{N(0,0,A,B)} \cdot \DT\wedge \DS}{\pmin k V\cdot \pmin \l W}\]
  
  \def\Ca{\Ccal}
  \def\Cb{\Ccal^\op}
  We first prove Proposition~\ref{prop:top_form_twist} which states that $\omega_{k+\l,n}$ is preserved by $\twistpm$, and then we will concentrate on showing that $\omega_{k+\l,n}\wedge \omega_\Ker$ is preserved by $\stw$.

  \subsection{Proof of Proposition~\ref{prop:top_form_twist}}\label{sec:proof-prop-top-form}
    We start by constructing two collections $\Ca,\Cb$ of $(k+\l)$-element subsets of $[n]$ such that they both form clusters for $\Gr(k+\l,n)$, and whenever $I$ and $J$ are related by~\eqref{eq:I_J}, we have $I\in\Ca$ if and only if $J\in\Cb$. Explicitly, $\Ca$ is the set of all $I\in{[n]\choose k+\l}$ that are unions of two cyclic intervals, one of which is of the form $[c,1+k)$ for $c\in[n]$ (if $c>1+k$ then we view this interval modulo $n$). Either of the two cyclic intervals is allowed to be empty. Similarly, $\Cb$ is the set of all $J\in{[n]\choose k+\l}$ that are unions of two cyclic intervals, one of which is of the form $[1,1+p)$ for some $p\geq0$. It is easy to check\footnote{In the language of~\cite{Pos}, $\Cb$ is the collection of face labels of the corresponding $\Le$-diagram, while $\Ca$ is the collection of face labels of the ``opposite'' $\Le$-diagram.} that both $\Ca$ and $\Cb$ are maximal weakly separated collections of size $(k+\l)m+1$, which is one more than the dimension of $\Gr(k+\l,n)$.

    Let us index the elements of $\Ca$ by $\Ca:=\{I_0,I_1,\dots,I_{(k+\l)m}\}$ and consider the coordinate chart on $\Gr(k+\l,n)$ given by $\{\Delta_I\mid I\in\Ca\}$, rescaled so that $\Delta_{I_0}=1$ for $I_0:=[\l+1,2\l+k+1)$ modulo $n$. Then by Proposition~\ref{prop:form:clusters}, we have \[\omkl(\U)=\pm\frac{d\Delta_{I_1}(\U)}{\Delta_{I_1}(\U)}\wedge\dots\wedge\frac{d\Delta_{I_{(k+\l)m}}(\U)}{\Delta_{I_{(k+\l)m}}(\U)}.\]
  Similarly, letting $\Cb:=\{J_0,J_1,\dots,J_{(k+\l)m}\}$ so that $J_0=[k+\l]$, we have \[\omkl(\Ut)=\pm\frac{d\Delta_{J_1}(\Ut)}{\Delta_{J_1}(\Ut)}\wedge\dots\wedge\frac{d\Delta_{J_{(k+\l)m}}(\Ut)}{\Delta_{J_{(k+\l)m}}(\Ut)}.\]
   If $I\in\Ca$ is a cyclic interval (there are $n$ such intervals in both $\Ca$ and $\Cb$), then Lemma~\ref{lemma:consecutive} provides a cyclic interval $J\in\Cb$ such that
    \[\Delta_J(\Ut)=\frac{\pm1}{\Delta_I(\U)}.\]
    Differentiating both sides, we get that
    \[\frac{d\Delta_J(\Ut)}{\Delta_J(\Ut)}=\pm\frac{d\Delta_I(\U)}{\Delta_I(\U)}.\]
    Thus the expression for $\omkl(\Ut)$ contains $d\Delta_K(\U)$ for all cyclic intervals $K$ of $[n]$ of size $k+\l$. Therefore we may treat them as constants while computing $d\Delta_J(\Ut)$ for other $J\in\Cb$. For each $J\in\Cb$, Lemma~\ref{lemma:I_J} provides a set $I\in\Ca$ and cyclic intervals $K,L\subset [n]$ such that
    \[\Delta_J(\Ut)=\pm\frac{\Delta_I(\U)}{\Delta_K(\U)\cdot\Delta_L(\U)}.\]
    Differentiating both sides and ignoring the terms with $d\Delta_K(\U)$ and $d\Delta_L(\U)$, we get
    \[\frac{d\Delta_J(\Ut)}{\Delta_J(\Ut)}=\pm\frac{d\Delta_I(\U)}{\Delta_I(\U)}.\]
    Applying this to all $J\in\Cb$ finishes the proof.\qed

\subsection{Proof of Theorem~\ref{thm:top_form_stw}}\label{sec:proof-thm-forms}
  \def\stwx{\stw'}
  \def\twistx{\twistpm'}
  Consider the cyclic shift map $\sigma:\topcell\to\topcell$ sending a matrix $\U$ with columns $\U(j), j\in\Z$ such that $\U(j+n)=(-1)^{k-1}\U(j)$ to a matrix $\sigma\U$ defined by $(\sigma\U)(j):=\U(j-1)$. This operator acts on $\V$ and $\W$ similarly. Rather than working with the maps $\stw$ and $\twistpm$, we will work with their cyclically shifted versions:
  \[\twistx:= \sigma^\l\circ \twistpm,\quad \stwx:=\sigma^\l\circ\stw.\]
By~\eqref{eq:twist_MS}, the map $\twistx$ coincides (up to sign) with the right twist map of~\cite{MuS}. Since the map $\stwx$ differs from the map $\stw$ by a cyclic shift, and the forms $\omk$, $\oml$, and $\omkl$ are preserved by $\sigma$, it suffices to prove the version of Theorem~\ref{thm:top_form_stw} with $\stw$ replaced with $\stwx$.

\begin{proof}[Proof of Theorem~\ref{thm:top_form_stw}.]
  We have shown above that $\omkl$ is preserved by $\stw$, and therefore by $\stwx$. What we need to show in order to complete the proof of Theorem~\ref{thm:top_form_stw} is that $\omega_\Ker\wedge\omkl$ is preserved by $\stwx$ up to a sign. Let us split $\omega_\Ker$ into a product of two terms:
  \[\omega_\Ker= \pm\frac{\DT\wedge \DS}{(\det(\M(T,S)))^{k+\l}}\cdot \frac{ \pmin{k+\l}{\U}}{\pmin k \V\cdot \pmin \l \W}.\]
 Here we are using the fact that $\pmin{k+\l}{\U}=(\det(M(T,S)))^{n} \cdot \pmin{k+\l}{\N(0,0,A,B)}$.

 By Lemma~\ref{lemma:consecutive}, we have $\pmin{k+\l}\U\cdot\pmin{k+\l}\Ut=\pm1$, and then Lemma~\ref{lemma:consecutive_V_W} implies
 \begin{equation}\label{eq:circular_minors_Ker}
\frac{\pmin{k+\l}{\U}}{\pmin k \V\cdot \pmin \l \W}
      =\pm\frac{\pmin{k+\l}{\Ut}}{\pmin k {\Wt}\cdot \pmin \l {\Vt}}.
 \end{equation}

 Note that $\omega_\Ker$ itself \emph{may not} be preserved by $\stwx$. Indeed, let  $\Tt,\St,\At,\Bt$ be the unique matrices such that
  \[\twistx(\N(T,S,A+TB,B+SA))\simkl \N(\Tt,\St,\At+\Tt\Bt,\Bt+\St\At).\]
Let us write down $\Tt,\St,\At,\Bt$ more explicitly. Recall that $\N(T,S,A+TB,B+SA)=\M(T,S)\cdot \N(0,0,A,B)$. Thus by Lemma~\ref{lemma:split}, we have
  \[\twistx(\N(T,S,A+TB,B+SA))\simkl \M(-S^\transp,-T^\transp)\cdot H=\M(-S^\transp,-T^\transp)\cdot H([k+\l])\cdot H',\]
  where $H:=\twistx(\N(0,0,A,B))$, $H([k+\l])$ is the square submatrix of $H$ spanned by the first $k+\l$ columns, and $H':=(H([k+\l]))^{-1}\cdot H$. We see that the first $k+\l$ columns of $H'$ form an identity matrix, and therefore we can write $H'=\N(0,0,\At,\Bt)$ for some $\At,\Bt$. Thus $\Tt,\St$ are the unique matrices such that
  \begin{equation}\label{eq:H(k+l)}
\M(\Tt,\St)\simkl \M(-S^\transp,-T^\transp)\cdot H([k+\l]).
  \end{equation}
  (We will later show that such matrices exist, i.e., that the upper left $k\times k$ and the lower right $\l\times \l$ submatrices of the matrix on the right hand side are invertible.) We see that the entries of $H$, and therefore of $\At$ and $\Bt$, depend on the entries of $A$ and $B$ but not on the entries of $T$ and $S$. However, the entries of $\Tt$ and $\St$ depend on the entries of all four matrices $T,S,A,B$. By definition, $\omega_\Ker$ is proportional to $\DT\wedge\DS$, so the pullback of $\omega_\Ker$ under $\stwx$ will be proportional to $\DTt\wedge\DSt$, and therefore will contain $da_{ij}$'s and $db_{ij}$'s as well as $dt_{ij}$'s and $ds_{ij}$'s. This shows that $\omega_\Ker$ is not in general preserved by $\stwx$.

  On the other hand, since $\omkl$ is preserved by $\stwx$ and is proportional to $\DA\wedge\DB$, we only need to show that $\omega_\Ker$ is preserved by $\stwx$ when we treat $A$ and $B$ as constants. We now investigate how exactly the matrices $\Tt$ and $\St$ depend on $T,S,A,B$.

\begin{lemma}\label{lemma:H_is_upper_tr}
The matrix $H([k+\l])$ from~\eqref{eq:H(k+l)} is an upper triangular $(k+\l)\times (k+\l)$ matrix with $\pm1$'s on the diagonal.
\end{lemma}
\begin{proof}
This is clear from the definition of the map $\twistx$: for $j\leq k+\l$, the $j$-th column of $H$ is orthogonal to columns $j+1,j+2,\dots,k+\l$ of $N(0,0,A,B)$, and its scalar product with the $j$-th column of $N(0,0,A,B)$ equals to $\pm1$. Since the first $k+\l$ columns of $N(0,0,A,B)$ form an identity matrix, the result follows.
\end{proof}

In particular, it follows from Lemma~\ref{lemma:H_is_upper_tr} that the upper left $k\times k$ and lower right $\l\times \l$ submatrices of $\M(-S^\transp,-T^\transp)\cdot H([k+\l])$ are invertible, which proves the existence of $\Tt,\St$ satisfying~\eqref{eq:H(k+l)}. Recall that the entries of $H([k+\l])$ depend only on the entries of $A$ and $B$, which we treat as constants since we are taking the wedge with $\DA\wedge\DB$. Our next result is illustrated in Example~\ref{ex:upper_tr}.
    \begin{lemma}\label{lemma:upper_tr}
      Suppose that for some constant upper triangular $(k+\l)\times (k+\l)$ matrix $E$ with $\pm1$'s on the diagonal, we have
      \[M(\Tt,\St) \simkl M(T,S)\cdot E.\]
      Then
      \begin{equation}\label{eq:upper_tr}
\frac{\DTt\wedge \DSt}{(\det(M(\Tt,\St)))^{k+\l}}=\pm\frac{\DT\wedge \DS}{(\det(M(T,S)))^{k+\l}}.
      \end{equation}
    \end{lemma}
    \begin{proof}
      We can write $E$ as a block matrix
      \[E=\begin{pmatrix}
          P & Q\\
          0 & R
        \end{pmatrix},\]
      where $P$ (resp., $R$) is an upper triangular matrix of size $k\times k$ (resp., $\l\times\l$) with $\pm1$'s on the diagonal. In this notation,
      \[\Tt=P^{-1}(Q+TR),\quad \St=(R+SQ)^{-1}SP.\]
      If $Q=0$ then we clearly have
      \[\DTt=\pm\DT,\quad \DSt=\pm\DS,\]
      and by~\eqref{eq:det_block}, $\det(\M(\Tt,\St))=\det(\M(T,S))$, which shows~\eqref{eq:upper_tr} in this case. 
Note that if $E=E^\parr1\cdot E^\parr2\cdots E^\parr N$ is a product of upper triangular matrices with $\pm1$'s on diagonals, then it suffices to prove the lemma for  $E^\parr i$ for each $1\leq i\leq N$, because then the result will follow by induction. The case when $E$ is a diagonal matrix has already been shown above, and we may thus assume that $E$ is a \emph{Chevalley generator}, i.e., a unit upper triangular matrix with a single nonzero off-diagonal entry in position $(i,i+1)$ for some $1\leq i<k+\l$. If $i\neq k$ then $Q=0$ so we are done. Now suppose that $i=k$ and thus $Q$ has a single non-zero entry $Q_{k1}=q$ for some constant $q$. We are going to verify~\eqref{eq:upper_tr} in this case by a direct computation, see also Example~\ref{ex:upper_tr} below.

We have $\Tt=Q+T$ and thus $\DTt=\DT$. Now, we have $\St=(\Id_\l+SQ)^{-1}S$. Denote $S=(s_{ij})$. Let us first compute the matrix $\Id_\l+SQ$ and its inverse:
\[\Id_\l+SQ=\begin{pmatrix}
    1+qs_{1k} & 0 & \dots & 0\\
    qs_{2k} & 1 & \dots & 0\\
    \vdots & 0 & \ddots & 0\\
    qs_{\l k}& 0 & \dots & 1
  \end{pmatrix};\quad (\Id_\l+SQ)^{-1}=\begin{pmatrix}
    \frac1{1+qs_{1k}} & 0 & \dots & 0\\
    -\frac{qs_{2k}}{1+qs_{1k}} & 1 & \dots & 0\\
    \vdots & 0 & \ddots & 0\\
    -\frac{qs_{\l k}}{1+qs_{1k}}& 0 & \dots & 1
  \end{pmatrix}.\]
Thus $\det((\Id_\l+SQ)^{-1})=\frac1{1+qs_{1k}}$. Note that
\[\M(\Tt,\St)=\begin{pmatrix}
    \Id_k & 0\\
   0 & (\Id_\l+SQ)^{-1} 
 \end{pmatrix}\cdot \M(T,S)\cdot \begin{pmatrix}
    \Id_k & Q\\
   0 & \Id_\l
 \end{pmatrix},\]
\def\st{{\tilde{s}}}
and therefore
\begin{equation}\label{eq:1+qs1k}
\det(\M(\Tt,\St))=\frac{\det(\M(T,S))}{1+qs_{1k}}.
\end{equation}
 It remains to compute $\DSt$, so denote $\St=(\st_{ij})$. We have
\begin{equation*}
  \begin{split}
    \st_{1k}&=\frac{s_{1k}}{1+qs_{1k}},\quad d\st_{1k}=\frac{ds_{1k}}{(1+qs_{1k})^2};\\
    \st_{1j}&=\frac{s_{1j}}{1+qs_{1k}},\quad d\st_{1j}=\frac{ds_{1j}}{1+qs_{1k}}+\alpha,\quad (1\leq j<k);\\
    \st_{ik}&=\frac{-qs_{ik}s_{1k}}{1+qs_{1k}}+s_{ik}=\frac{s_{ik}}{1+qs_{1k}},\quad d\st_{ik}=\frac{ds_{ik}}{1+qs_{1k}}+\beta,\quad (1< i\leq \l);\\
    \st_{ij}&=\frac{-qs_{ik}s_{1j}}{1+qs_{1k}}+s_{ij},\quad d\st_{ij}=ds_{ij}+\gamma,\quad (1\leq j<k, 1< i\leq \l).\\
  \end{split}
\end{equation*}
Here $\alpha$ and $\beta$ denote sums of terms involving $ds_{1k}$, while $\gamma$ denotes a sum of terms involving  $ds_{1k}$, $ds_{1j}$, and $ds_{ik}$. Thus each of those terms vanishes when we take the wedge of all $\st_{ij}$'s. We have shown that
\[\DSt=\frac{\DS}{(1+qs_{1k})^{k+\l}},\]
and combining this with~\eqref{eq:1+qs1k}, we get
\[\frac{\DTt\wedge \DSt}{(\det(M(\Tt,\St)))^{k+\l}}=\frac{\DT\wedge \DS}{(1+qs_{1k})^{k+\l}(\det(M(\Tt,\St)))^{k+\l}}=\frac{\DT\wedge \DS}{(\det(M(T,S)))^{k+\l}},\]
finishing the proof of Lemma~\ref{lemma:upper_tr}.
    \end{proof}

Theorem~\ref{thm:top_form_stw} follows from Lemmas~\ref{lemma:H_is_upper_tr} and~\ref{lemma:upper_tr}, combined with Equations~\eqref{eq:circular_minors_Ker} and~\eqref{eq:H(k+l)}.
\end{proof}

    \begin{example}\label{ex:upper_tr}
      Let us illustrate Lemma~\ref{lemma:upper_tr} in the case $k=\l=1$. We have $M(T,S)=\begin{pmatrix}
        1 & t\\
        s & 1
      \end{pmatrix}$, so the right hand side of~\eqref{eq:upper_tr} is $\frac{dt\wedge ds}{(1-st)^2}$. Suppose that $E=\begin{pmatrix}
        1&q\\
        0& 1
      \end{pmatrix}$ for some constant $q$. Thus
      \[M(T,S)\cdot E=\begin{pmatrix}
          1 & q+t\\
          s& 1+qs
        \end{pmatrix}\simkl \begin{pmatrix}
          1 & q+t\\
          \frac{s}{1+qs} & 1
        \end{pmatrix}=\begin{pmatrix}
          1 &\ttt\\
          \stt & 1
        \end{pmatrix}=M(\Tt,\St) .\]
      So $\ttt=q+t$ and $\stt=\frac{s}{1+qs}$. We have:
      \[d\ttt=dt,\quad d\stt=\frac{ds}{(1+qs)^2},\quad \det(M(\Tt,\St))=\frac{1-st}{1+qs},\]
and therefore the left hand side of~\eqref{eq:upper_tr} equals
\[\frac{d\ttt\wedge d\stt}{\det(M(\Tt,\St))^2}=\frac{dt\wedge ds}{(1-st)^2}.\]
\end{example}

\begin{proof}[Proof of Theorem~\ref{thm:lower_form_stw}.]
  We deduce the result as a simple corollary of Theorem~\ref{thm:top_form_stw}. For an affine permutation $f\in\Aff(-k,\l)$, we need to show that $\stw$ sends $\om k f\wedge\oml$ to $\pm\omk\wedge\om \l{f^{-1}}$. We proceed by induction on $\inv(f)$, the base case $\inv(f)=0$ being precisely the content of Theorem~\ref{thm:top_form_stw}. Suppose that $g\lessdot f$ in the affine Bruhat order for some $f,g\in\Aff(-k,\l)$, and that the result has already been shown for $f$. By Proposition~\ref{prop:residues} below, the form $\om k g$ is obtained from $\om k f$ by taking a residue, and the same is true for $f^{-1}$ and $g^{-1}$:
  \[\Res_{\ce k g} \om k f =\pm \om k g,\quad \Res_{\ce \l {g^{-1}}} \om \l {f^{-1}} = \pm \om \l {g^{-1}}\]

Let $A \subset \ce k f \times \GrC(\l,n)$ be the Zariski open subset such that 
\begin{enumerate}
\item $\stack(\V,\W) \in \topcell$ whenever $(\V,\W) \in A$, 
\item $\stw$ extends to an isomorphism $\stw: A \to B$ where $ B \subset \GrC(k,n) \times \ce \l {f^{-1}}$,
\item $\stw$ sends $A \cap (\ce k g \times \GrC(\l,n))$ isomorphically to $B \cap (\GrC(\l,n) \times \ce \l {g^{-1}})$.
\end{enumerate}
Both $A \subset \ce k f \times \GrC(\l,n)$ and $A \cap (\ce k g \times \GrC(\l,n)) \subset \ce k g \times \GrC(\l,n)$ will be nonempty Zariski open subsets, since by Theorem \ref{thm:stw} the corresponding positive points will be contained in this locus.  Taking residues commutes with isomorphisms, so by induction we compute
\begin{align*}
\stw^*(\omega_{k,n} \wedge\om \l{g^{-1}}) & = 
\stw^*(\Res_{B \cap (\GrC(\l,n) \times \ce \l {g^{-1}})} \omega_{k,n} \wedge\om \l{f^{-1}})\\
&= \Res_{A \cap (\ce k g \times \GrC(\l,n))} \stw^*(\omega_{k,n} \wedge\om \l{f^{-1}}) \\
&= \pm \Res_{A \cap (\ce k g \times \GrC(\l,n))} (\om k {f} \wedge \omega_{\l,n})\\
&= \pm \om k g \wedge \omega_{\l,n}.\qedhere
\end{align*}
\end{proof}

\section{The amplituhedron form: proofs}\label{sec:ampl_form_proofs}
In this section, we prove Theorem~\ref{thm:univ_ampl_form}, as well as several other statements given in Section~\ref{sec:ampl_form}. Before we proceed, we recall some further background that was omitted from that section.

\subsection{Residues}\label{sec:residues}
We start with the definition of the \emph{residue operator} $\Res$.
Suppose we are given a rational form $\omega$ on a complex variety $X$. Let $C\subset X$ be an irreducible subvariety of codimension one. Consider some smooth Zariski open subset $U\subset X$ such that $U \cap C$ is Zariski open in $C$. Suppose that $z:U\to \C$ is a regular map such that $C \cap U$ is the set of solutions to $z=0$. Denote the remaining coordinates on $U$ by $u$, so that every point in $U$ is written as $(z,u)$ in this coordinate chart. We say that $\omega$ has a \emph{simple pole} at $C$ if inside $U$ we have
\[\omega(z,u)=\omega'(u)\wedge\frac{dz}{z}+\omega''(z,u),\]
where $\omega''(z,u)$ and $\omega'(u)$ are rational forms on $U$, and $\omega'(u)$ is non-zero. When $\omega$ has a simple pole at $C$, we define the \emph{residue} by
\[\Res_C\omega:=\omega'|_{U \cap C}\]
This defines $\Res_C\omega$ as a rational form on $U\cap C$, and thus on $C$.

\begin{proposition}[{\cite[Theorem~13.2]{LamCDM}}]\label{prop:residues}
  For affine permutations $f,g\in\Aff(-k,n-k)$ such that $g\lessdot f$ in the affine Bruhat order, we have
  \begin{equation}\label{eq:residues}
\Res_{\pcC kg}\om kf=\pm\om k g.
  \end{equation} 
\end{proposition}

\subsection{Degree of a map}\label{sec:degree}
\def\morph{\alpha}
\def\A{A}
\def\B{B}

Let $\morph: X \to Y$ be a morphism between irreducible complex varieties of the same dimension.  If $\dim(\morph(X)) = \dim(X) (=\dim(Y))$, then there exists a nonempty Zariski open subset $\B \subseteq Y$ such that $\morph$ restricts to a $d$-fold covering map from $\morph^{-1}(\B)$ to $\B$.  In other words, $\morph$ is locally a diffeomorphism on $\morph^{-1}(\B)$, and for $y \in \B$ we have $|\morph^{-1}(y)| = d$.  We say that $\morph$ \emph{has degree} $\deg(\morph) = d$.  If $\dim(\morph(X)) < \dim(X)$, we define $\deg(\morph) = \infty$.  If $\morph: X\dashedrightarrow Y$ is instead a rational map, then we have $\deg(\morph) = \deg(\morph|_\A)$ where $\A \subseteq X$ is the locus where $\morph$ is defined.  If $\morph$ has degree $\deg(\morph) = 1$, then it is birational: it restricts to an isomorphism between dense Zariski open subsets of $X$ and $Y$.

In \cite{Lam14}, an integer $d_f \in \{1,2,\ldots,\infty\}$ is assigned to each $f \in \Aff(-k,n-k)$ so that for a dense Zariski open subset $A \subset \GrC(k+m,n)$, the rational map $Z: \GrC(k,n) \dashedrightarrow \GrC(k,k+m)$ has degree $d_f$ whenever $Z \in A$.  Since $\Grtp(k+m,n)$ is Zariski dense in $\GrC(k+m,n)$, this statement also holds for $Z$ in an open dense subset of $\Grtp(k+m,n)$.  An affine permutation $f\in \Aff(-k,n-k)$ such that $\inv(f)=km$ has degree $1$ in the sense of Definition \ref{def:degree} if and only if $d_f = 1$.

\begin{proof}[Proof of Proposition~\ref{prop:Zcal_degree}.]
Suppose that $d_f = 1$.  Let $A \subset \GrC(k+m,n)$ be the Zariski open dense subset such that $Z:\Gr(k,n) \dashedrightarrow \Gr(k,k+m)$ has degree $1$ for $Z\in A$.  For $Z \in A$, let $B_Z \subset \GrC(k,k+m)$ be the dense Zariski open subset such that $Z: Z^{-1}(B_Z) \to B_Z$ is an isomorphism.  Consider the Zariski open subset $C \subset \FlC(\l,k+\l;n)$ defined by
\[C = \{(W,U) \mid Z = W^\perp \in A \text{ and } \Span(U \cdot Z^t) \in B_Z\}.\]
Then $\Zcal: \Zcal^{-1}(C) \to C$ is an isomorphism, and thus $\Zcal: \pcC k f \times \GrC(\l,n) \dashedrightarrow \FlC(\l,k+\l;n)$ has degree one.
\end{proof}

\begin{proof}[Proof of Proposition~\ref{prop:Jacobian_constant_sign}.]
Note that $\pc kf$ is homeomorphic to an open ball. Since $f$ is $(n,k,m)$-admissible, the restriction of $Z$ to $\pc kf$ is injective. By Brouwer's invariance of domain theorem~\cite{Brouwer}, we get that it is a homeomorphism onto its image, because the dimension of $\pc kf$ equals $km$. Therefore the restriction of $Z$ to $\pc kf$ either preserves or reverses its topological orientation (see~\cite[Section~22]{GH} for background on topological orientation).  For smooth manifolds, it is well known that the notions of topological and smooth orientation agree with each other.  Thus whenever the Jacobian of $Z$ is nonzero on $\pc kf$, its sign is the same.
\end{proof}

\begin{proof}[Proof of Proposition~\ref{prop:Zcal_Jacobian_constant_sign}.]
The product $\pc k f \times \Grtperp{k+m}n$ is homeomorphic to an open ball.  When $f$ is $(n,k,m)$-admissible, the restriction of $\Zcal$ to $\pc kf \times \Grtperp{k+m}n$ is injective.  The rest of the proof is identical to that of Proposition~\ref{prop:Jacobian_constant_sign}.
\end{proof}

\begin{remark}
Even when $d_f > 1$, we can still define the pushforward $Z_* \om k f$ of the canonical form, but the definition is more complicated.  See \cite{LamCDM} for details.
\end{remark}

\def\affst{{\tilde F}}
\subsection{Affine Stanley symmetric functions}\label{sec:affst}
We review the results of the second author~\cite{LamAS,Lam14} that allow one to give a purely combinatorial characterization of a degree $1$ positroid cell. We refer the reader to~\cite{Stanley} for background on symmetric functions.

\begin{definition}
An element $v\in\Aff$ is called \emph{cyclically decreasing} if it has a reduced word $v=s_{j_1}s_{j_2}\dots s_{j_N}$ such that the indices $j_1,j_2,\dots,j_N$ are pairwise distinct modulo $n$, and if $j_a=i+1$ and $j_b=i$ modulo $n$ for some $a,b\leq N$ then we have $a<b$. 
\end{definition}

Given an element $f\in\Aff$, define the \emph{affine Stanley symmetric function}
\[\affst_f(x_1,x_2,\dots):=\sum_{f=f_1f_2\dots f_r} x_1^{\inv(f_1)}x_2^{\inv(f_2)}\dots x_r^{\inv(f_r)},\]
where the sum is over all factorizations $f=f_1f_2\dots f_r$ such that $f_1,\dots,f_r\in\Aff$ are cyclically decreasing. We list some properties of $\affst_f$ together with the references to the places where they have been proven in the literature.
\begin{proposition}\label{prop:affst_properties}\leavevmode
\begin{enumerate}
\item For each $f\in\Aff$, $\affst_f$ is a symmetric function~\cite[Theorem~6]{LamAS}.
\item For each $f\in\Aff$, we have $\omega^+ \affst_f=\affst_{f^{-1}}$~\cite[Theorem~15]{LamAS}.
\item If $\inv(f)=k\l$ for some $f\in\Aff(-k,n-k)$ then $f$ has degree $1$ if and only if the coefficient of $s_{\l^k}$ in $\affst_f$ equals to $1$~\cite[Proposition~4.8]{Lam14}.
\end{enumerate}
\end{proposition}
Here $\omega^+$ denotes a certain involution on a subset $\Lambda^\parr n$ of the algebra $\Lambda$ of symmetric functions, and $s_{\l^k}\in\Lambda^\parr n$ denotes the \emph{Schur function} indexed by a rectangular partition $\l^k:=(\underbrace{\l,\l,\dots,\l}_{\text{$k$ times}})$. Since $\omega^+ s_{\l^k}=s_{k^\l}$, Proposition~\ref{prop:affst_properties} implies Proposition~\ref{prop:degre_1_inverse}.

\begin{remark}
Proposition~\ref{prop:degre_1_inverse} can also be deduced directly from Proposition~\ref{prop:fibers}, since $f$ has degree $1$ if and only if $\Fib(\V,Z)\cap \pc k f$ contains one point for a generic pair $(\V,Z)\in\pc kf\times \Grtp(k+m,n)$.
\end{remark}

\subsection{Proof of Theorem~\ref{thm:univ_ampl_form}}

The proof of Theorem~\ref{thm:lower_form_stw} shows that for the rational map
\[\stw:\pcC k f\times \GrC(\l,n)\dashedrightarrow \GrC(k,n)\times \pcC \l {f^{-1}},\]
we have $\stw^*(\omk\wedge \om\l{f^{-1}}) = \pm \om k{f}\wedge \oml$ for each $f\in\Aff(-k,\l)$.
By Theorem~\ref{thm:main} together with Corollary~\ref{cor:degree_1_triang}, the only thing left to show is that the sign is ``$+$'' when both forms are chosen to be positive as in Definition~\ref{dfn:signs}.  We choose some reference form $\omreflkl$ on $\Fl(\l,k+\l;n)$ and some reference form $\omrefklk$ on $\Fl(k,k+\l;n)$ so that $\omreflkl=\tstw^\ast\omrefklk$. The sign of $\om k{f}$ is chosen in such a way that the forms $\Zcal_*(\om k{f}\wedge \oml)$ and $\omreflkl$ have the same sign on $\amcell kf$.  The sign of $ \om\l{f^{-1}}$ is chosen in such a way that the forms $\Zcal_*(\omk\wedge \om\l{f^{-1}})$ and $\omreflkl$ have the same sign on $\amcell \l{f^{-1}}$.  Note that $\tstw$ sends $\amcell kf$ to $\amcell \l{f^{-1}}$ and pulls $\omrefklk$ back to $\omreflkl$.
Since the diagram~\eqref{eq:commutative_diagram} is commutative, we conclude that  $\stw^*(\omk\wedge \om\l{f^{-1}}) = \om k{f}\wedge \oml$ with the above choice of signs, finishing the proof of Theorem~\ref{thm:univ_ampl_form}.\qed

\section{Open questions and future directions}\label{sec:future}
In addition to the multiple conjectures already mentioned in the text
(Conjectures~\ref{conj:physics}, \ref{conj:M(a,b,c)}, \ref{conj:degree_1}, \ref{conj:ampl_form}, \ref{conj:univ_ampl_form}),
we list several other exciting questions.

\subsection{Positive Laurent formulas for the twisted Pl\"ucker coordinates}
Let $\stw(\V,\W)=(\Wt,\Vt)$. Theorem~\ref{thm:stw}~\eqref{item:stw:positivity} asserts that if the Pl\"ucker coordinates of $\V$ and $Z:=\W^\perp$ are nonnegative then so are the Pl\"ucker coordinates of $\Zt:=\Wt^\perp$ and $\Vt$. In fact, we believe that an even stronger property holds.
\begin{conjecture}\label{conj:Laurent}
  Let $\U:=\stack(\V,\W)\in\topcell$, $\stw(\V,\W)=(\Wt,\Vt)$, $\Ut:=\twistpm(\U)=\stack(\Wt,\Vt)$, $Z:=\W^\perp$, and $\Zt:=\Wt^\perp$. Then there exist some polynomials $H$ and $H'$ in the maximal minors of $\V$ and $Z$ with nonnegative integer coefficients such that 
  \begin{enumerate}
  \item the maximal minor $\Delta_I(\Zt)$ for $I\in{[n]\choose \l+m}$ is equal to
    \[\frac{H}{\pm \prod_{j\in I}\Delta_{[j-\l,j+k)}(\U)};\]
  \item the maximal minor $\Delta_I(\Vt)$ for $I\in{[n]\choose \l}$ is equal to
    \[\frac{H'}{\pm \prod_{j\in I}\Delta_{[j-\l,j+k)}(\U)}.\]
  \end{enumerate}
\end{conjecture}
  Here the sign $\pm$ in the denominator is chosen so that it is positive. (Recall that the sign of $\Delta_{[j-\l,j+k)}(\U)$ is predicted by Lemma~\ref{lemma:consecutive}.)

Note that the case of this conjecture where $I$ is a circular interval of $[n]$ follows from Lemma~\ref{lemma:consecutive_V_W}.

\begin{example}\label{ex:positive_Laurent}
For the case $k=2$, $\l=1$, $m=2$, $n=5$, let us suppose that $\W$ is the row span of the $1\times 5$ matrix $(p,-q,r,-s,t)$, so that the maximal minors of $Z$ are precisely $p,q,r,s,t\geq0$. Then we have
\[\Delta_{245}(\Zt)=\Delta_{13}(\alt(\Wt))=\frac{pr\Delta_{24}(\V)+qr\Delta_{14}(\V)+ps\Delta_{23}(\V)+qs\Delta_{13}(\V)}{-\Delta_{125}(\U)\Delta_{234}(\U)}.\]
Specifically, for the matrices $\V$ and $\W$ considered in~\eqref{eq:ex:V:W}, we have
\[p=1,\quad q=1,\quad r=3,\quad s=2,\quad t=1,\quad \Delta_{125}(\U)=2,\quad \Delta_{234}(\U)=-8,\]
thus the formula above yields
\[\Delta_{13}(\alt(\Wt))=\frac{3\Delta_{24}(\V)+3\Delta_{14}(\V)+2\Delta_{23}(\V)+2\Delta_{13}(\V)}{16}=\frac{7}{8},\]
which is indeed the case as we see by computing $\Delta_{13}(\alt(\Wt))$ directly  from~\eqref{eq:ex:Vt:Wt}.
\end{example}

\begin{remark}
There is a formal similarity between our stacked twist map and the \emph{geometric $R$-matrix} of Frieden \cite{Fri}, which is essentially also a birational map $R: \Gr(k,n) \times \Gr(\ell,n) \dashedrightarrow \Gr(\ell,n) \times \Gr(k,n)$.  Furthermore, the geometric $R$-matrix can be written in terms of subtraction-free rational formulae, and tropicalizes to the \emph{combinatorial $R$-matrix}.  We do not know of a direct relation between the stacked twist map and the geometric $R$-matrix.
\end{remark}

  The twist map $\twistpm:\topcell\to\topcell$ satisfies Lemma~\ref{lemma:inverse_transpose}. Given any subgroup $G\subset\GL_{k+\l}(\R)$, one can consider the subgroup $G^\transp:=\{g^\transp\mid g\in G\}$. By Lemma~\ref{lemma:inverse_transpose}, if we have $\U'=g\cdot \U$ for $\U,\U'\in\topcell$ and $g\in G$ then the matrices $\Ut:=\twistpm(\U)$ and $\Ut':=\twistpm(\U')$ are related by $\Ut'=(g^{-1})^\transp\cdot \Ut$. Therefore $\twistpm$ reduces to a map $\topcell/G\to\topcell/G^\transp$, where we are taking the quotients modulo the left action of $G$ (resp., $G^\transp$). There are several interesting choices for $G$.
\begin{itemize}
\item Taking $G:=\GL_k(\R)\times\GL_\l(\R)$ recovers the stacked twist map.
\item Taking $G$ to be a maximal parabolic subgroup consisting of block matrices $\begin{pmatrix}
    A & B\\
    0 & C
  \end{pmatrix}$, where $A$ is $k\times k$, $B$ is $k\times \l$, and $C$ is $\l\times \l$, yields a map on two-step flag varieties: $\twistpm:\Fl^\circ(\l,k+\l;n)\to\Fl^\circ(k,k+\l;n)$ sending $(\W,\U)$ to $(\Wt,\Ut)$. Conjecturally, the maximal minors of $\Wt$ can be written as rational functions in terms of the maximal minors of $\W$ and $\U$. For instance, in the setting of Example~\ref{ex:positive_Laurent}, we could also write \[\Delta_{13}(\alt(\Wt))=\frac{p\Delta_{234}(\U)+q\Delta_{134}(\U)}{\Delta_{125}(\U)\Delta_{234}(\U)},\]
  which for $\W$ and $\U$ given in~\eqref{eq:ex:V:W} yields $\frac{1\times(-8)+1\times(-6)}{2\times(-8)}=\frac78$. However, the total positivity of $\alt(\Wt)$ is obscure in this expression, because $\Delta_{134}(\U)$ can be both negative and positive when $\V\in\Grtp(k,n)$ and $\Wt\in\Grtperp{k+m}n$.
\item Finally, taking $G$ to be the subgroup of lower-triangular matrices and letting $r:=k+\l$ yields a map $\twistpm:\Fl^\circ(1,2,\dots,r;n)\to\Fl^\circ(1,2,\dots,r;n)$ that takes a partial flag $\V_\bullet:=\V_1\subset \V_2\subset \dots\subset\V_r$ (where $\V_i$ is the span of the first $i$ rows of $\U$) to a partial flag $\Vt_\bullet:=\Vt_1\subset\Vt_2\subset\dots\subset\Vt_r$ (where $\Vt_i$ is the span of the last $i$ rows of $\Ut$). The partial flag $\V_\bullet$ is determined by the \emph{flag minors} of $\U$, where a flag minor is a minor of $\U$ with row set $[i]$ and column set $J\in{[n]\choose i}$ for $1\leq i\leq r$. Similarly, the partial flag $\Vt_\bullet$ is determined by the \emph{opposite flag minors} of $\U$, i.e., minors of $\U$ with row set $\{n-i+1,\dots,n\}$ and column set $J\in{[n]\choose i}$ for $1\leq i\leq r$. In fact, an explicit combinatorial formula in terms of matchings in a plabic graph~\cite{Pos} for the opposite flag minors of $\Ut$ in terms of flag minors of $\U$ can be proven in  a straightforward but tedious fashion. To do so, one needs to combine the technique used in the proof of Lemma~\ref{lemma:consecutive_V_W} with~\cite[Theorem~1.1]{MS}. As a consequence of their formula, we get the following result.
  \begin{proposition}
If $\Ut=\twistpm(\U)$ for some $\U\in\topcell$ then the opposite flag minors of $\Ut$ are positive Laurent polynomials in the flag minors of $\U$.
\end{proposition}
\end{itemize}

\subsection{A polynomial expression for the amplituhedron form}\label{sec:ampl_form_poly}
Recall that in our example in Section~\ref{sec:ex_forms}, we have computed the amplituhedron form $\omtriang{\ampl}$ for $k=\l=1$, $m=2$, $n=4$, see~\eqref{eq:ex_ampl_form}. We observe that this expression equals
\[\omtriang{\ampl}=\frac{H(s,p,q,x,y)\cdot dx\wedge dy}{\pmin{k+\l}U},\]
where $U$ is given by~\eqref{eq:omtriang_uamk_dfn} and $H(s,p,q,x,y)=qsx+psy+pq$. Let us generalize this observation to other values of $k,\l,m,n$ using our coordinates from Section~\ref{sec:coordinates}.

Similarly to~\eqref{eq:coords_Fl_ex}, the coordinates that we assign to a pair $(W,U)\in\Fl(\l,k+\l;n)$ are going to be matrices $S,A,B$ of respective sizes $\l\times k$, $k\times m$, and $\l\times m$ so that $W$ is the row span of $[S|\Id_\l|B]$ while $U$ is the row span of
\begin{equation}\label{eq:coords_Fl_gen}
\begin{pmatrix}
  S & \Id_\l & B\\
  \Id_k & 0 & A
\end{pmatrix}.
\end{equation}
Taking
\begin{equation}\label{eq:coords_Fl_gen_Z}
Z:=\begin{pmatrix}
  \Id_k & -S & 0\\
  0&    -B & \Id_\l
\end{pmatrix},
\end{equation}
we get $Z^\perp=W$.

\begin{conjecture}\label{conj:ampl_form_poly}
There exists a polynomial $H(S,A,B)$ in the entries of the matrices $S,A,B$ such that 
\[\omtriang{\ampl}=\frac{H(S,A,B)\cdot \D kmA}{\pmin{k+\l}U},\]
where $U$ is given by~\eqref{eq:coords_Fl_gen} and $Z$ is given by~\eqref{eq:coords_Fl_gen_Z}. The universal amplituhedron form can be written in the coordinates~\eqref{eq:coords_Fl_gen} as
\[\omtriang{\uamk}=\frac{H(S,A,B)\cdot \D kmA\wedge \D \l k S \wedge \D \l m B}{\pmin{k+\l}U\cdot\pmin{\l}W},\]
where $W=Z^\perp$ is the row span of $[S|\Id_\l|B]$.
\end{conjecture}
The polynomial $H(S,A,B)$ from the above conjecture should take positive values whenever there exists $V\in\Grtp(k,n)$ such that $\Span(V,W)=U$ for $U$ being given by~\eqref{eq:coords_Fl_gen} and $W\in\Grtperp{k+m}n$ being the row span of $[S|\Id_\l|B]$.

\newcommand\pcperp[2]{\Pi_{\shift{#1}{#2}}^{\perp,>0}}
\newcommand\pcalt[2]{\alt\left(\pc{#1}{#2}\right)}
\def\fop{{f^\op}}
\subsection{Stackable cells}
Even though for the amplituhedron we naturally have $\W\in\Grtperp{k+m}n$, the stacked twist map construction suggests that there should be symmetry between $\V$ and $\W$. In this section, we generalize our previous results to the case when both $\V$ and $Z:=\W^\perp$ belong to lower positroid cells. Let us define 
\[\pcalt k f:=\{\alt(\V)\mid \V\in\pc kf\}\]

\begin{definition}
For two affine permutations $f\in\Aff(-k,n-k)$ and $g\in\Aff(-\l,n-\l)$, we say that $f$ and $g$ are \emph{stackable} if for all $\V\in\pc k f$ and $\W\in\pcalt \l g$, we have $\stack(\V,\W)\in\topcell$.
\end{definition}
As we saw in the proof of Lemma~\ref{lemma:consecutive}, $f$ and $g$ are stackable if and only if for every $j\in\Z$, there exists $I\in\positroid{k}{f}$ and $J\in\positroid{\l}{g}$ such that $I\cup J=\{j-\l,j-\l+1,\dots,j+k-1\}$ modulo $n$. This operation corresponds to taking the \emph{union} of matroids, more precisely, $f$ and $g$ are stackable if the \emph{positroid envelope} (in the sense of~\cite{KLS}) of the matroid union $\positroid k f\vee \positroid{\l}{g}$ is the uniform matroid of rank $k+\l$ on $n$ elements.

Of course one can check whether $f$ and $g$ are stackable by going through all $j\in[n]$ and all pairs of $I\in \positroid{k}{f}$ and $J\in\positroid{\l}{g}$, and this algorithm works in general for unions of matroids. We ask whether the fact that the matroids $\positroid{k}{f}$ and $\positroid{\l}{g}$ are \emph{positroids} makes this easier to check.
\begin{problem}
Find a simple necessary and sufficient condition for $f$ and $g$ to be stackable.
\end{problem}

The following result extends Theorems~\ref{thm:main} and~\ref{thm:lower_form_stw}, and has a similar proof.
\begin{theorem}
  If $f\in\Aff(-k,\l)$ and $g\in\Aff(-\l,k)$ are stackable cells then the stacked twist map descends to a diffeomorphism
  \[\stw:\pc k f\times \pcalt \l g\to \pcalt \l{g^{-1}}\times \pc k{f^{-1}}\]
  satisfying
    \begin{equation}
  \stw^*(\om \l {g^{-1}} \wedge\om k {f^{-1}}) =  \pm \om k f\wedge\om \l g.
  \end{equation}
\end{theorem}

We note that however in the original description of the problem coming from physics, the matrices $Z$ and $V$ play asymmetric roles.  Namely, $Z$ encodes \emph{external data}, i.e. the momenta of the particles in the experiment, while the matrix $\V$ are \emph{hidden variables} or degrees of freedom that are integrated out in the final scattering amplitude~\cite{AT}.
\begin{problem}
Explain the symmetry between $\V$ and $\W=Z^\perp$ from the point of view of scattering amplitudes.
\end{problem}

\subsection{Bistellar flips and the generalized Baues problem}\label{sec:baues}

Recall that for $k=1$, the amplituhedron $\ampl$ is just a cyclic polytope $C(n,m)$ in the $m$-dimensional projective space~\cite{sturmfels_88}. Thus the amplituhedron is a direct generalization of a cyclic polytope, and in this section we list some properties of cyclic polytopes that we expect to generalize to amplituhedra. One of such properties is already stated as Conjecture~\ref{conj:M(a,b,c)}.

Recall that for any $g\in\Aff(-k,\l)$, we have defined triangulations of $Z(\pc k g)$ in Section~\ref{sec:subdivisions}. For example if $g\in\Aff(-k,\l)$ satisfies $\inv(g)=k\l$ and is $(n,k,m)$-admissible then $Z(\pc k g)$ has only one triangulation $\{g\}$. Let us say that $g\in\Aff(-k,n-k)$ is \emph{nearly $(n,k,m)$-admissible} if  $\inv(g)=k\l-1$,  $\dim(Z(\pc k g))=km$ for all $Z\in\Grtp(k+m,n)$, and for any $h\lessdot g$, we have either $\deg h=1$ for all $Z\in\Grtp(k+m,n)$ or $\deg h=\infty$ for all $Z\in\Grtp(k+m,n)$.
\begin{conjecture}\label{conj:flip}
Let $g\in\Aff(-k,n-k)$ be a nearly $(n,k,m)$-admissible affine permutation. Then there exist two disjoint sets $\triang_g=\{h_1,\dots,h_N\}\subset\Aff(-k,\l)$ and $\triang_g'=\{h'_1,\dots,h'_{N}\}\subset\Aff(-k,\l)$ of affine permutations of the same size such that for each $Z\in\Grtp(k+m,n)$, $\triang_g$ and $\triang_g'$ are the only two triangulations of $Z(\pc k g)$. The union $\triang_g\cup\triang_g'$ consists of all affine permutations $h\lessdot g$ such that $\deg h=1$.
\end{conjecture}

Assuming that the above conjecture is true, we say that two $(n,k,m)$-triangulations $\triang,\triang'\subset\Aff(-k,n-k)$ of the amplituhedron \emph{differ by a flip} if there exists a nearly $(n,k,m)$-admissible affine permutation $g\in\Aff(-k,n-k)$ such that $\triang_g\subset\triang$, $\triang_g'\subset\triang'$, and
\[\triang\setminus\triang_g=\triang'\setminus \triang_g'.\]
This is a direct generalization of \emph{bistellar flips} of triangulations of polytopes. In particular, it was shown in~\cite[Theorem~1.1]{Rambau} that any two triangulations of a cyclic polytope $C(n,m)$ are connected by a sequence of flips.

\begin{conjecture}\label{conj:flip_connected}
Any two triangulations of the amplituhedron are connected by a sequence of flips.
\end{conjecture}

We note that Conjectures~\ref{conj:flip} and~\ref{conj:flip_connected} together imply that all triangulations of the amplituhedron have the same cardinality. This is a major step towards Conjecture~\ref{conj:M(a,b,c)}, but the latter additionally gives a simple formula for the size of any such triangulation.

\begin{example}\label{ex:four_mass_box}
  Let $k=2$, $\l=2$, $m=4$, $n=8$, and consider an affine permutation $h\in\Aff(-k,\l)$ given by $h=[2,1,4,3,6,5,8,7]$ in window notation. The corresponding positroid cell $\pc k h$, known as the \emph{four-mass box}~\cite[Section~11]{abcgpt}, satisfies $\deg h=2$. Let us now take $g:=[1,2,4,3,6,5,8,7]$, so that $h\lessdot g$ in the affine Bruhat order. Thus $g$ is not nearly $(n,k,m)$-admissible. It is not clear to us whether $g$ has any proper subdivisions, and whether there exists a subdivision of the amplituhedron which involves $g$ as one of the cells. If such a subdivision exists and $g$ has no triangulations then it is not the case that every subdivision of the amplituhedron can be further refined to a triangulation.

  Additionally, we note that apart from $h$ there are $8$ affine permutations $h'\lessdot g$ satisfying $\deg h'\neq \infty$, and all of them actually have degree $1$. Together with $h$, these $9$ affine permutations give rise to an identity involving pushforwards of canonical forms, as explained in~\cite[Section~11]{abcgpt}. Moreover, there exists an expression for $\omega_{\ampl}$ of the form~\eqref{eq:ampl_form_dfn} with $M(2,2,2)=20$ terms in the right hand side, one of which is $\omega_{Z(\pc k h)}$. Even though by Conjectures~\ref{conj:M(a,b,c)} and~\ref{conj:ampl_form}, each $(n,k,m)$-triangulation of the amplituhedron has $20$ cells and yields an expression for $\omega_{\ampl}$ with $20$ terms, the corresponding $20$ affine permutations mentioned above do not form an $(n,k,m)$-triangulation since $h$ is not $(n,k,m)$-admissible, and some other permutations in this collection are not $(n,k,m)$-compatible with each other. We thank Jake Bourjaily for sharing the details of this example with us.
\end{example}

\def\baues{\omega}
Consider a simple graph $G$ whose vertices are triangulations of the amplituhedron, and two triangulations are connected by an edge if they differ by a flip. Thus Conjecture~\ref{conj:flip_connected} states that this graph is connected. One can continue this process by ``gluing'' a higher-dimensional cell to this graph corresponding to each coarser subdivision of the amplituhedron. In the case of polytopes, the homotopy type of the resulting CW complex is the subject of the \emph{generalized Baues problem} of~\cite{BKS}, see~\cite{Reiner} for a survey. Formally, define the \emph{Baues poset} $\baues(\ampl)$ to be the partially ordered set consisting of all proper\footnote{By definition, the only subdivision of the amplituhedron that is not proper is the one that consists of the top cell $\{\id 0\}$.} $(n,k,m)$-subdivisions of the amplituhedron, ordered by refinement.

\begin{problem}\label{problem:Baues}
Is it true that the Baues poset $\baues(\ampl)$ is homotopy equivalent to a $(k\l-1)$-dimensional sphere?
\end{problem}

For the case $k=1$, the statement of Problem~\ref{problem:Baues} was shown in~\cite{RS}. We note that Theorem~\ref{thm:main_subd} shows that the Baues poset of the $(n,k,m)$-amplituhedron is isomorphic to that of the $(n,\l,m)$-amplituhedron, and thus Problem~\ref{problem:Baues} has a positive answer also in the case $\l=1$. We do not know whether Example~\ref{ex:four_mass_box} provides any obstructions for the statement of Problem~\ref{problem:Baues} to be true.

For polytopes, the generalized Baues problem is closely related to \emph{regular triangulations} and \emph{secondary polytopes} of~\cite{GKZ}.

\begin{problem}
Generalize the notion of a regular triangulation to the amplituhedron case.
\end{problem}

\begin{example}
We have verified our conjectures computationally in the case $k=\l=m=2$ and $n=6$. There are $48$ $(n,k,m)$-admissible affine permutations $f\in\Aff(-k,n-k)$ such that $\inv(f)=k\l$, and they are precisely all affine permutations that have degree $1$ (see Conjecture~\ref{conj:degree_1}). They form $120$ triangulations, and each of those consists of $M(2,2,1)=6$ cells (see Conjecture~\ref{conj:M(a,b,c)}). The flip graph $G$ in this case has $120$ vertices and $278$ edges, and is connected (see Conjecture~\ref{conj:flip_connected}). Finally, there are $696$ proper subdivisions of the amplituhedron, and the corresponding Baues poset has homology groups equal to that of a $3$-dimensional sphere, showing some substantial evidence towards the positive answer to Problem~\ref{problem:Baues}.
\end{example}

\subsection{BCFW cells}\label{sec:BCFW}
For $m=4$, there is a certain interesting collection $\triang_{k,\l}=\{f_1,\dots,f_N\}\subset\Aff(-k,\l)$ of positroid cells for each $k$ and $\l$, see~\cite{BCFW,abcgpt,AT}. Up to a cyclic shift, we get $\triang_{\l,k}=\{f_1^{-1},\dots,f_N^{-1}\}$, as we have noted in Remark~\ref{rmk:BCFW}. This collection of cells, known as the \emph{BCFW triangulation}, is believed to have all the nice properties that we have considered in this paper, in particular, it can be used (and in fact is widely used) to compute the amplituhedron form.
\begin{conjecture}\label{conj:BCFW}
The BCFW triangulation $\triang_{k,\l}$ is an $(n,k,m)$-triangulation of degree $1$ for each $k$ and $\l$ (and $m=4$).
\end{conjecture}
This conjecture holds in the case $k=1$ of cyclic polytopes (and therefore it holds for $\l=1$ as well by Corollary~\ref{cor:degree_1_triang}). A part of Conjecture~\ref{conj:BCFW} (that the elements of $\triang_{k,\l}$ are pairwise $(n,k,m)$-compatible) has been proven for $k=2$ in~\cite{KWZ}. As before, our results imply that the elements of $\triang_{\l,k}$ are pairwise $(n,\l,m)$-compatible for $k=2$.

\subsection{Cohomology classes}
Let $X$ be a smooth projective complex algebraic variety and $Y \subset X$ a subvariety.  Let $[Y] \in H^\ast(X,\Z)$ denote the cohomology class of $Y$.  The cohomology classes $[Z(\pcC k f)] \in H^\ast(\GrC(k,k+m),\Z)$ for generic $Z \in \GrC(k+m,n)$ (or, equivalently, for $Z$ belonging to a dense subset of $\Grtp(k+m,n)$) were determined in \cite{Lam14}.  Recall that in Section \ref{sec:affst}, we have defined the affine Stanley symmetric function $\affst_f$.  In \cite{Lam14}, it is shown that the cohomology class $[Z(\pcC k f)]$ is represented by the symmetric function $\tau_{k+m}(\affst_f)$, defined by
$$
\tau_{k+m}(\affst_f) = \sum_{\mu} c_{\mu^{+\ell}} s_\mu \qquad \text{if} \qquad \affst_f = \sum_{\lambda} c_\lambda s_\lambda.
$$
Here $\mu^{+\ell}$ is the partition obtained from $\mu$ by adding $\ell$ columns of height $k$ to the left of $\mu$ (in English notation).  The term is omitted if this is not a partition.  

It is an open problem to understand the cohomology classes $[Z(\pcC k f)]$ when $Z$ is not generic.  Note that in the $k = 1$ polytope case, $Z(\pcC k f)$ is always a linear hypersubspace of projective space, so the cohomology class $[Z(\pcC k f)] \in H^\ast(\GrC(k,k+m),\Z)$ is simply the effective generator of $H^d(\GrC(k,k+m),\Z)$ where $d$ is the codimension of $Z(\pcC k f)$.

It would be interesting to extend the calculation of \cite{Lam14} to the universal setting.

\begin{problem}\label{prob:cohomology}
What is the cohomology class $[\amcellC kf]$ in $H^\ast(\FlC(k,k+\l;n),\Z)$ and how is it related to $[\amcellC \l{f^{-1}}]$?
\end{problem}

\newcommand{\arxiv}[1]{\href{https://arxiv.org/abs/#1}{\textup{\texttt{arXiv:#1}}}}

\bibliographystyle{alpha_tweaked}
\bibliography{duality}

\end{document}